\documentclass{article}
\usepackage{fullpage}
\usepackage{fancyhdr}
\usepackage{amsmath, amssymb, latexsym, amscd, amsthm,amsfonts,amstext}
\usepackage{subfigure}
\usepackage{enumerate}
\usepackage{xcolor}
\usepackage{textcomp}
\usepackage{mathtools}

\def\div{\operatorname{div}}
\usepackage{graphicx}
\newtheorem{theorem}{Theorem}[section]
\newtheorem{lemma}[theorem]{Lemma}
\newtheorem{proposition}[theorem]{Proposition}
\newtheorem{corollary}[theorem]{Corollary}
\numberwithin{equation}{section}
\newtheorem{definition}[theorem]{Definition}

\newenvironment{indeed}[0]{Indeed, }{}
\numberwithin{equation}{section}

\usepackage[mathscr]{eucal}
\usepackage{graphicx}
\usepackage{subfigure}
\usepackage{cite}
\usepackage{stfloats}
\usepackage{fancyhdr}
\usepackage{enumerate}
\usepackage{xcolor}

\def\div{\operatorname{div}}
\usepackage{url}

\def\div{\operatorname{div}}
\allowdisplaybreaks
\title{The Foldy-Lax approximation of the scattered waves by many small bodies for the Lam\'e system}
\author{
Durga Prasad Challa\footnote{C\lowercase{orresponding author} : Durga Prasad Challa}
\thanks{RICAM, Austrian Academy of Sciences,
Altenbergerstrasse 69, A-4040, Linz, Austria.
Email:durga.challa@oeaw.ac.at, Tel: +43 (0)732 2468 5234.
Supported by the Austrian Science Fund (FWF): P22341-N18.}
\and  Mourad Sini
\thanks{RICAM, Austrian Academy of Sciences,
Altenbergerstrasse 69, A-4040, Linz, Austria.
Email:mourad.sini@oeaw.ac.at, Tel: +43 (0)732 2468 5258.
Partially supported by the Austrian Science Fund (FWF): P22341-N18.}
}

\begin{document}
\graphicspath{{Smallscattersreport_Figures_50incobs/}}
\maketitle
\begin{abstract}
 We are concerned with the linearized, isotropic and homogeneous elastic scattering problem by many small rigid obstacles of arbitrary, Lipschitz regular, shapes in 3D case.
 We prove that there exists two constant $a_0$ and $c_0$, depending only on the Lipschitz character of the obstacles, such that under the conditions $a\leq a_0$ and $\sqrt{M-1}\frac{a}{d} \leq c_0$ on the number $M$ of the obstacles, their maximum diameter $a$ and
the minimum distance between them $d$, the corresponding Foldy-Lax approximation of the farfields is valid. In addition, we provide the error of this approximation explicitly in terms of the three parameters $M, a$ and $d$.
These approximations can be used, in particular, in the identification problems (i.e. inverse problems) and in the design problems (i.e. effective medium theory).

\end{abstract}



\pagestyle{myheadings}
 \thispagestyle{plain}

\section{Introduction and statement of the results}\label{Introduction-smallac-sdlp}

 Let $B_1, B_2,\dots, B_M$ be $M$ open, bounded and simply connected sets in $\mathbb{R}^3$ with Lipschitz boundaries, 
containing the origin.
We assume that their sizes and Lipschitz constants are uniformly bounded.
We set $D_m:=\epsilon B_m+z_m$ to be the small bodies characterized by the parameter
$\epsilon>0$ and the locations $z_m\in \mathbb{R}^3$, $m=1,\dots,M$.
\par
 Assume that the Lam\'e coefficients $\lambda$ and $\mu$ are constants satisfying $ \mu > 0 \mbox{ and } 3\lambda+2\mu >0$.
Let $U^{i}$ be a solution of the Navier equation $(\Delta^e + \omega^{2})U^{i}=0 \mbox{ in } \mathbb{R}^{3}$, $\Delta^{e}:=(\mu\Delta+(\lambda+\mu)\nabla \div)$.
We denote by  $U^{s}$ the elastic field scattered by the $M$ small bodies $D_m\subset \mathbb{R}^{3}$ due to
the incident field $U^{i}$. We restrict ourselves to the scattering by rigid bodies. Hence the total field $U^{t}:=U^{i}+U^{s}$
satisfies the following exterior Dirichlet problem of the elastic waves
\begin{equation}
(\Delta^e + \omega^{2})U^{t}=0 \mbox{ in }\mathbb{R}^{3}\backslash \left(\mathop{\cup}_{m=1}^M \bar{D}_m\right),\label{elaimpoenetrable}
\end{equation}
\begin{equation}
U^{t}|_{\partial D_m}=0,\, 1\leq m \leq M\label{elagoverningsupport}
\end{equation}
with the Kupradze radiation conditions (K.R.C)
\begin{equation}\label{radiationcela}
\lim_{|x|\rightarrow\infty}|x|(\frac{\partial U_{p}}{\partial|x|}-i \kappa_{p^\omega}U_{p})=0,  \mbox{  and  }
\lim_{|x|\rightarrow\infty}|x|(\frac{\partial U_{s}}{\partial|x|}-i \kappa_{s^\omega}U_{s})=0,
\end{equation}
where the two limits are uniform in all the directions $\hat{x}:=\frac{x}{|x|}\in \mathbb{S}^{2}$.
Also, we denote $U_{p}:=-\kappa_{p^\omega}^{-2}\nabla (\nabla\cdot U^{s})$ to be the longitudinal
(or the pressure or P) part of the field $U^{s}$ and $U_{s}:=\kappa_{s^\omega}^{-2}\nabla\times(\nabla\times U^{s})$ to be the transversal (or the shear or S) part of
the field $U^{s}$ corresponding to the Helmholtz decomposition $U^{s}=U_{p}+U_{s}$. The constants $\kappa_{p^\omega}:=\frac{\omega}{c_p}$ and
$\kappa_{s^\omega}:=\frac{\omega}{c_s}$ are known as the longitudinal and transversal wavenumbers, $c_p:=\sqrt{\lambda+2\mu}$ and $c_s:=\sqrt{\mu}$ are the corresponding phase velocities, respectively and $\omega$ is the frequency.
\par
The scattering problem (\ref{elaimpoenetrable}-\ref{radiationcela}) is well posed in the H\"{o}lder or Sobolev spaces,
see \cite{C-K:1998,C-K:1983,Kupradze:1965, K-G-B-B:1979} for instance, and the scattered field $U^s$ has the following asymptotic expansion:

\begin{equation}\label{Lamesystemtotalfieldasymptoticsmall}
 U^s(x) := \frac{e^{i\kappa_{p^\omega}|x|}}{|x|}U^{\infty}_{p}(\hat{x}) +
\frac{e^{i\kappa_{s^\omega}|x|}}{|x|}U^{\infty}_{s}(\hat{x}) + O(\frac{1}{|x|^{2}}),~|x|\rightarrow \infty
\end{equation}
uniformly in all directions $\hat{x}\in \mathbb{S}^{2}$. The longitudinal part of the far-field, i.e. $U^{\infty}_{p}(\hat{x})$ is normal to $\mathbb{S}^{2}$
while the transversal part $U^{\infty}_{s}(\hat{x})$ is tangential to $\mathbb{S}^{2}$.
As usual in scattering problems we use plane incident waves in this work. For the Lam\'e system, the full plane incident wave is of the form
$U^{i}(x,\theta):=\alpha\theta\,e^{i\kappa_{p^\omega}\theta\cdot x}+\beta\theta^\bot\,e^{i\kappa_{s^\omega}\theta\cdot x}$,
where $\theta^{\bot}$ is any direction in $\mathbb{S}^{2}$ perpendicular to the incident direction $\theta \in \mathbb{S}^2$, $\alpha,\beta$
are arbitrary constants.  In particular, the pressure and shear incident waves are $
U^{i,p}(x,\theta) := \theta e^{i\kappa_{p^\omega}\theta\cdot x}\mbox{ and } U^{i,s}(x,\theta) := \theta^{\bot} e^{i\kappa_{s^\omega}\theta\cdot x}$, respectively.
Pressure incident waves propagate in the direction of $\theta$, whereas shear
incident waves propagate in the direction of $\theta^{\bot}$. The functions
$U^{\infty}_p(\hat{x}, \theta):=U^{\infty}_p(\hat{x})$ and $U^{\infty}_s(\hat{x}, \theta):=U^{\infty}_s(\hat{x})$ for $(\hat{x}, \theta)\in \mathbb{S}^{2} \times\mathbb{S}^{2}$
are called the P part and the S part of the far-field pattern respectively.
\newline

\begin{definition} 
\label{Def1}
We define
\begin{enumerate}
 \item $a$ as the maximum among the diameters, $diam$, of the small bodies $D_m$, i.e.
 \begin{equation}\label{def-a-ela}
a:=\max\limits_{1\leq m\leq M } diam (D_m) ~~\big[=\epsilon \max\limits_{1\leq m\leq M } diam (B_m)\big],
\end{equation}
 \item  $d$ as the minimum distance  between the small bodies $\{D_1,D_2,\dots,D_m\}$, i.e.
$$d:=\min\limits_{\substack{m\neq j\\1\leq m,j\leq M }} d_{mj},$$
$\text{where}\,d_{mj}:=dist(D_m, D_j)$. We assume that
\begin{equation}\label{def-dmax-elasmall}
0\,<\,d\,\leq\,d_{\max},
\end{equation}
and $d_{\max}$ is given.
\item $\omega_{\max}$ as the upper bound of the used frequencies, i.e. $\omega\in[0,\,\omega_{\max}]$.
\item $\Omega$ to be a bounded domain in $\mathbb{R}^3$ containing the small bodies $D_m,\,m=1,\dots,M$.
\end{enumerate}
\end{definition}
\bigskip

The main result of this paper is the following theorem.
\begin{theorem}\label{Maintheorem-ela-small-sing}
 There exist two positive constants $a_0$ and $c_0$ depending only on the size of $\Omega$, the
Lipschitz character of $B_m,m=1,\dots,M$, $d_{\max}$ and $\omega_{\max}$ such that
if
\begin{equation}\label{conditions-elasma}
a \leq a_0 ~~ \mbox{and} ~~ \sqrt{M-1}\frac{a}{d}\leq c_0
\end{equation}
then the P-part, $U^\infty_p(\hat{x},\theta)$, and the S-part, $U^\infty_s(\hat{x},\theta)$, of the far-field pattern have the following asymptotic expressions
 \begin{eqnarray}
  U^\infty_p(\hat{x},\theta)&=&\hspace{-.1cm}\frac{1}{4\pi\,c_p^{2}}(\hat{x}\otimes\hat{x})\hspace{-.1cm}\left[\sum_{m=1}^{M}e^{-i\frac{\omega}{c_p}\hat{x}\cdot z_{m}}Q_m\right.\left.+O\left(M a^2+M(M-1)\frac{a^3}{d^2}
+M(M-1)^2\frac{a^4}{d^3}\right)\right], \label{x oustdie1 D_m farmainp}\\
 U^\infty_s(\hat{x},\theta)&=&\hspace{-.1cm} \frac{1}{4\pi\,c_s^{2}}(I- \hat{x}\otimes\hat{x})\hspace{-.1cm}\left[\sum_{m=1}^{M}e^{-i\frac{\omega}{c_s}\hat{x}\cdot\,z_m}Q_m\right.\left.+O\left(M a^2+M(M-1)\frac{a^3}{d^2}
+M(M-1)^2\frac{a^4}{d^3}\right)\right] \label{x oustdie1 D_m farmains}
  \end{eqnarray}
uniformly in $\hat{x}$ and $\theta$ in $\mathbb{S}^2$. The constant appearing in the estimate $O(.)$ depends only on the size of $\Omega$,
the Lipschitz character of the reference bodies $B_m$, $a_0$, $c_0$ and $\omega_{max}$. The vector coefficients $Q_m$, $m=1,..., M,$ are the solutions of the following linear algebraic system
\begin{eqnarray}\label{fracqcfracmain}
 C_m^{-1}Q_m &=&-U^{i}(z_m, \theta)-\sum_{\substack{j=1 \\ j\neq m}}^{M} \Gamma^{\omega}(z_m,z_j)Q_j,~~
\end{eqnarray}
for $ m=1,..., M,$ with $\Gamma^{\omega}$ denoting the Kupradze matrix of the fundamental solution to the Navier equation with frequency $\omega$, $C_m:=\int_{\partial D_m}\sigma_m(s)ds$ and $\sigma_{m}$ is
the solution matrix of the integral equation of the first kind
\begin{eqnarray}\label{barqcimsurfacefrm1main}
\int_{\partial D_m}\Gamma^{0}(s_m,s)\sigma_{m} (s)ds&=&\rm \textbf{I},~ s_m\in \partial D_m,
\end{eqnarray}
with $\rm \textbf{I}$ the identity matrix of order 3. The algebraic system \eqref{fracqcfracmain} is invertible 
under the condition:
\begin{eqnarray}\label{invertibilityconditionsmainthm-ela}
\frac{a}{d}&\leq&c_1t^{-1}
\end{eqnarray}
with 
\begin{center}
$
t:=\left[\frac{1}{c_p^2}-2diam(\Omega)\frac{\omega}{c_s^3}\left(\frac{1-\left(\frac{1}{2}\kappa_{s^\omega}diam(\Omega)\right)^{N_\Omega}}{1-\left(\frac{1}{2}\kappa_{s^\omega}diam(\Omega)\right)}+\frac{1}{2^{N_{\Omega}-1}}\right)-diam(\Omega)\frac{\omega}{c_p^3}\left(\frac{1-\left(\frac{1}{2}\kappa_{p^\omega}diam(\Omega)\right)^{N_\Omega}}{1-\left(\frac{1}{2}\kappa_{p^\omega}diam(\Omega)\right)}+\frac{1}{2^{N_{\Omega}-1}}\right)\right],
$
\end{center}
which is assumed to be positive\footnote{If, in particular, $diam(\Omega)\max\{\kappa_{s^\omega},\kappa_{p^\omega}\}e^2<1$, then $N_\Omega=1$ and hence
$t=\left[\frac{1}{c_p^2}-4diam(\Omega)\left(\frac{\omega}{c_s^3}+\frac{\omega}{2c_p^3}\right)\right]$. Assuming the Lam\'e coefficient $\lambda$ to be positive, then $c_p>c_s$. Hence, in this case,
if $\Omega$ is such that $diam(\Omega)<\frac{c_s}{\omega}\min\left\{\frac{1}{e^2},\frac{c_s^2}{6c_p^2}\right\}$ then $t>0$.}
and $N_{\Omega}:=[2diam(\Omega)\max\{\kappa_{s^\omega},\kappa_{p^\omega}\}e^2]$, where $[\cdot]$ denotes the integral part and $\ln e=1$.
The constant $c_1$ depends only on the Lipschitz character of the reference bodies $B_m$, $m=1,\dots,M$.
\end{theorem}
\bigskip

%
%
%

We call the expressions (\ref{x oustdie1 D_m farmainp}-\ref{x oustdie1 D_m farmains}) the elastic Foldy-Lax approximation since the dominant terms are reminiscent to
the exact form (called also the Foldy or the Foldy-Lax form) of the farfields derived in the scattering by finitely many point-like scatterers, see \cite{Hu-Si:2013} for instance. These asymptotic expansions are useful for at least two reasons.
       \par
       First, to estimate approximately the far-field, one needs only to compute the constant vectors $Q_m$ which are solutions of a linear algebraic system, i.e. \eqref{fracqcfracmain}.
       This reduces considerably the computational effort comparing it to the methods based on integral equations, for instance, especially for a large number of obstacles.
       If the number of obstacles is actually very large then these asymptotics suggest the kind of effective medium that can produce the same far-fields and provides
       the error rate between the fields generated by the obstacles and those generated by the effective medium.
       \par Second, using formulas of the type \eqref{x oustdie1 D_m farmainp} and \eqref{x oustdie1 D_m farmains},
       one can solve the inverse problems which consists of localizing the centers, $z_m$, of the obstacles from the
       far-field measurements using MUSIC type algorithm, for instance, and also estimating their sizes from the computed capacitances $C_m$.

As a first reference on this topic, we mention the book by P. Martin \cite{Martin:2006} where the multiple scattering issue is well discussed and documented in its different scales.
When the obstacles are distributed periodically in the whole domain, then homogenization techniques apply, see for instance \cite{B-L-P:1978, J-K-O:1994, M-K:2006}.
As we see it in the previous theorem, we assume no periodicity. For such media, the type of result presented here are formally derived, for the acoustic and electromagnetic models,  in a series of works by A. Ramm,
see \cite{MRAMAG-book1, RAMM:2007, RAMM:2011} and the references therein for his recent related results, where he used the (rough) condition
$\frac{a}{d}\ll 1$.
Recently, in \cite{DPC-SM13}, we derived such approximation errors under a quite general condition on the denseness of the scatterers (i.e. involving $M$, a and $d$),
i.e. of the form (\ref{conditions-elasma}). The analysis is based on the use of integral equation methods and in particular the precise scaling of the surface layer potential operators.
 As it was mentioned in \cite{DPC-SM13}, the integral equation methods are also used in such a context, see for instance the series of works by H. Ammari and H. Kang
 and their collaborators, as \cite{A-K:2007} and the references therein. The difference between their
 asymptotic expansion and the one described in the previous theorem is that their polarization tensors are build up from densities which are solutions of a system of
integral equations while in the previous theorem the approximating terms are build up from the linear algebraic system (\ref{fracqcfracmain}). It is obvious that
solving  an algebraic system is less expansive than solving a system of integral equations, especially when dealing with many scatterers.
In addition, due to motivations coming from inverse problems, apart from few works as \cite{A-H-K-L}, they consider well separated scatterers and  hence
their asymptotic expansions are given only in terms of the diameters $a$ of the scatterers. We should, however, emphasize that they provide asymptotic expansions at all the higher orders and they are valid also for
extended scatterers. This opens door for many interesting applications, see \cite{A-G-J-K-L-S-W} for instance.

Let us mention the variational approach by V. Maz'ya, A. Movchan \cite{M-M:MathNach2010} and by V. Maz'ya, A. Movchan and M. Nieves \cite{M-M-N:MMS:2011} where they study the Poisson problem and
 obtain estimates in forms similar to the previous theorem with weaker conditions of the form $\frac{a}{d} \leq c$, or
$\frac{a}{d^2} \leq c$, (where, here $d$ is the smallest distance between the centers of the scatterers).
In their analysis, they rely on the maximum principle to treat the boundary estimates.
To avoid the use of the maximum principle, which is not valid due to the presence of the wave number $\kappa$,
we use boundary integral equation methods. The price to pay is the need of the stronger assumption $\sqrt{M-1}\frac{a}{d}\leq c_0$.
Another approach, based on the self-adjoint extensions of elliptic operator, is discussed in the works by S. A. Nazarov, and J. Sokolowski, see section 4 of \cite{Na-So:2006} for instance, where they derive
the asymptotic expansions for the Poisson problem. Let us finally, mention that the particular case where the obstacles have circular shapes has been considered recently by M. Cassier and C. Hazard
in \cite{C-H:2013} for the scalar acoustic model.

%

Regarding the Lam\'e system, we cite the works \cite{A-K-N-T:JE2003, A-K-L-LameCPDE2007, A-C-I:SIAM2008, A-B-G-J-K-W}
where, as we just mentioned, the asymptotics are given in terms of the size of the scatterers only. In these works, the authors considered transmission problems and showed that
the corresponding moment tensors are in general anisotropic. If the inclusions are spherical, including the extreme cases of soft or rigid inclusions
under certain conditions on the Lam\'e parameters, then these moment tensors are isotropic. Let us also mention the recent book \cite{M-M-N:Springerbook:2013}, and the references therein,
where an asymptotic expansion of the Green's tensor corresponding to the Dirichlet-Lam\'e problem (with zero frequency) in a bounded domain containing many small holes is derived.

The goal of the present work is to extend the results in \cite{DPC-SM13} to the Lam\'e system and derive the error of the approximation explicitly in terms of the whole denseness of the scatterers, i.e. $M, a$ and $d$.
To our best knowledge this is the first result with this generality for the Lam\'e system, compare to \cite{A-K:2007} and \cite{M-M-N:Springerbook:2013}. The analysis is based on the use of the integral equation methods
and the different scaling of the corresponding boundary integral operators. Due to the coupling of the two fundamental waves, i.e. the P-waves and the S-waves, at the boundaries of the obstacles, the analysis cannot
be reduced to the one of our previous work \cite{DPC-SM13}. Indeed, a considerable work is needed to derive explicitly these scaling, characterize the dominant parts of the elastic fields and justify the invertibility of
corresponding algebraic system \eqref{fracqcfracmain}.

Before concluding the introduction, we state the following corollary where more precise estimates than those given in Theorem \ref{Maintheorem-ela-small-sing} are presented under some additional conditions on the scatterers.
\begin{corollary}\label{corMaintheorem-ela-small-sing}
 Assume that the conditions of Theorem \ref{Maintheorem-ela-small-sing} are fulfilled.

 \begin{enumerate}
  \item  We assume, in addition, that $D_m$ are balls with the same diameter $a$ for $m=1,\dots,M$, then we have the following asymptotic expansion
 for the P-part, $U^\infty_p(\hat{x},\theta)$, and the S-part, $U^\infty_s(\hat{x},\theta)$, of the far-field pattern:
 \begin{eqnarray}
  &\hspace{-1cm}U^\infty_p(\hat{x},\theta)=&\frac{1}{4\pi\,c_p^{2}}(\hat{x}\otimes\hat{x})\left[\sum_{m=1}^{M}e^{-i\frac{\omega}{c_p}\hat{x}\cdot z_{m}}Q_m\right.\nonumber\\
  &&\hspace{-2cm}\left.+O\left(M\left[a^2+\frac{a^3}{d^{5-3\alpha}}+\frac{ a^4}{d^{9-6\alpha}}\right]+M(M-1)\left[\frac{a^3}{d^{2\alpha}}+\frac{a^4}{d^{4-\alpha}}+\frac{a^4}{d^{5-2\alpha}}\right]+M(M-1)^2\frac{a^4}{d^{3\alpha}}\right) \right],
   \label{x oustdie1 D_m farmainp-near}\\
 &\hspace{-1cm}U^\infty_s(\hat{x},\theta)=& \frac{1}{4\pi\,c_s^{2}}(I- \hat{x}\otimes\hat{x})\left[\sum_{m=1}^{M}e^{-i\frac{\omega}{c_s}\hat{x}\cdot\,z_m}Q_m\right.\nonumber\\
 &&\hspace{-2cm}\left.+O\left(M\left[a^2+\frac{a^3}{d^{5-3\alpha}}+\frac{ a^4}{d^{9-6\alpha}}\right]+M(M-1)\left[\frac{a^3}{d^{2\alpha}}+\frac{a^4}{d^{4-\alpha}}+\frac{a^4}{d^{5-2\alpha}}\right]+M(M-1)^2\frac{a^4}{d^{3\alpha}}\right) \right],
 \label{x oustdie1 D_m farmains-near}
  \end{eqnarray}
 where $0<\alpha\leq1$.
\par
Consider now the special case $d=a^t,\,M=a^{-s}$ with  $t,s>0$. Then the asymptotic expansions (\ref{x oustdie1 D_m farmainp-near}-\ref{x oustdie1 D_m farmains-near}) can be rewritten as
\begin{eqnarray*}
 U^\infty_p(\hat{x},\theta)&=&\frac{1}{4\pi\,c_p^{2}}(\hat{x}\otimes\hat{x})\big[\sum_{m=1}^{M}e^{-i\frac{\omega}{c_p}\hat{x}\cdot z_{m}}Q_m\nonumber\\
 &&+O\left(a^{2-s}+a^{3-s-5t+3t\alpha}+a^{4-s-9t+6t\alpha}+a^{3-2s-2t\alpha}+a^{4-3s-3t\alpha}+a^{4-2s-5t+2t\alpha}\right)\big],\nonumber\\ \label{x oustdie1 D_m farmainp-near*}\\
 U^\infty_s(\hat{x},\theta)&=& \frac{1}{4\pi\,c_s^{2}}(I- \hat{x}\otimes\hat{x})\big[\sum_{m=1}^{M}e^{-i\frac{\omega}{c_s}\hat{x}\cdot\,z_m}Q_m\nonumber\\
 &&+O\left(a^{2-s}+a^{3-s-5t+3t\alpha}+a^{4-s-9t+6t\alpha}+a^{3-2s-2t\alpha}+a^{4-3s-3t\alpha}+a^{4-2s-5t+2t\alpha}\right) \big].\nonumber\\ \label{x oustdie1 D_m farmains-near*}
 \end{eqnarray*}

As the diameter $a$ tends to zero the error term tends to zero for $t$ and $s$ such that 
$0<t<1$ and $0<s<\min\{2(1-t),\,\frac{7-5t}{4},\,\frac{12-9t}{7},\frac{20-15t}{12},\frac{4}{3}-t\alpha\}$.
In particular for $t=\frac{1}{3}$, $s=1$, we have
\begin{eqnarray}
 U^\infty_p(\hat{x},\theta)&=&\frac{1}{4\pi\,c_p^{2}}(\hat{x}\otimes\hat{x})\left[\sum_{m=1}^{M}e^{-i\frac{\omega}{c_p}\hat{x}\cdot z_{m}}Q_m
 +O\left(a+a^{2\alpha}+a^{1-\alpha}+a^{\frac{1+2\alpha}{3}}\right)\right]\nonumber\\
 &=&\frac{1}{4\pi\,c_p^{2}}(\hat{x}\otimes\hat{x})\left[\sum_{m=1}^{M}e^{-i\frac{\omega}{c_p}\hat{x}\cdot z_{m}}Q_m+O\left(a^{\frac{1}{2}}\right)\right],
 \quad[\mbox{obtained for } \alpha=\frac{1}{4}], \label{x oustdie1 D_m farmainp-near**}\\
 U^\infty_s(\hat{x},\theta)&=& \frac{1}{4\pi\,c_s^{2}}(I- \hat{x}\otimes\hat{x})\left[\sum_{m=1}^{M}e^{-i\frac{\omega}{c_s}\hat{x}\cdot\,z_m}Q_m
 +O\left(a+a^{2\alpha}+a^{1-\alpha}+a^{\frac{1+2\alpha}{3}}\right)\right]\nonumber\\
 &=& \frac{1}{4\pi\,c_s^{2}}(I- \hat{x}\otimes\hat{x})\left[\sum_{m=1}^{M}e^{-i\frac{\omega}{c_s}\hat{x}\cdot\,z_m}Q_m+O\left(a^{\frac{1}{2}}\right)\right],
 \quad[\mbox{obtained for } \alpha=\frac{1}{4}].\label{x oustdie1 D_m farmains-near**}
 \end{eqnarray}

 \item Actually, the results (\ref{x oustdie1 D_m farmainp-near}-\ref{x oustdie1 D_m farmains-near}) and (\ref{x oustdie1 D_m farmainp-near**}-\ref{x oustdie1 D_m farmains-near**}) are valid
  for the non-flat Lipschitz obstacles $D_m=\epsilon{B}_m+z_m, m=1,\dots,M$ with the same diameter $a$, i.e. $D_m$'s are Lipschitz obstacles and there exist constants $t_m \in (0, 1]$ such that
 \begin{equation}\label{non-flat-condition}
 B^{3}_{t_m\frac{a}{2}}(z_m)\subset\,D_m\subset\,B^{3}_{\frac{a}{2}}(z_m),
 \end{equation}
 where $t_m $ are assumed to be uniformly bounded from below by a positive constant.
%
  \end{enumerate}
 \end{corollary}
 \bigskip

 The results of this corollary can be used to derive the effective medium by perforation using many small bodies.  In addition,
 \eqref{x oustdie1 D_m farmainp-near**} and \eqref{x oustdie1 D_m farmains-near**}  ensure the rate of the error in deriving such an effective medium.
 Details on this topic will be reported in a future work.

The rest of the paper is organized as follows. In section 2, we give the proof of the asymptotic expansion (\ref{x oustdie1 D_m farmainp}-\ref{x oustdie1 D_m farmains}).
In section 3, we study the solvability of the linear algebraic system (\ref{fracqcfracmain}). Finally, in section 4, as an appendix, we derive some needed properties
of the layer potentials.

\section{Proof of Theorem \ref{Maintheorem-ela-small-sing}}\label{Proof of Theorem Small-ela}
We wish to warm the reader that in our analysis we use sometimes the parameter $\epsilon$
and some other times the parameter $a$ as they appear naturally in the estimates.
But we bear in mind the relation (\ref{def-a-ela}) between $a$ and $\epsilon$.
\subsection{The fundamental solution}\label{fdelsms}
The Kupradze matrix $\Gamma^\omega=(\Gamma^\omega_{ij})^3_{i,j=1}$ of the fundamental solution to the Navier equation is given by
\begin{eqnarray}\label{kupradzeten}
 \Gamma^\omega(x,y)=\frac{1}{\mu}\Phi_{\kappa_{s^\omega}}(x,y)\rm \textbf{I}+\frac{1}{\omega^2}\nabla_x\nabla_x^{\top}[\Phi_{\kappa_{s^\omega}}(x,y)-\Phi_{\kappa_{p^\omega}}(x,y)],
\end{eqnarray}
where $\Phi_{\kappa}(x,y)=\frac{1}{4\pi}\exp(i\kappa|x-y|)$ denotes the free space fundamental solution of the Helmholtz equation $(\Delta+\kappa^2)\,u=0$ in
$\mathbb{R}^3$.
The asymptotic behavior of Kupradze tensor at infinity is given as follows
\begin{equation}\label{elafundatensorasymptotic}
 \Gamma^{\omega}(x,y)=\frac{1}{4\pi\,c_p^{2}}\hat{x}\otimes\hat{x} \frac{e^{i\kappa_{p^\omega}|x|}}{|x|}e^{-i\kappa_{p^\omega}\hat{x}\cdot\,y} +
\frac{1}{4\pi\,c_s^{2}}(I- \hat{x}\otimes\hat{x}) \frac{e^{i\kappa_{s^\omega}|x|}}{|x|}e^{-i\kappa_{s^\omega}\hat{x}\cdot\,y}+O(|x|^{-2})
\end{equation}
with $\hat{x}=\frac{x}{|x|}\in\mathbb{S}^{2}$ and $I$ being the identity matrix in $\mathbb{R}^{3}$, see \cite{A-K:IMA2002} for instance.
As mentioned in \cite{A-K-L-LameCPDE2007},  \eqref{kupradzeten} can also be represented as
\begin{eqnarray}\label{kupradzeten1}
 \Gamma^\omega(x,y)&=&\frac{1}{4\pi}\sum_{l=0}^{\infty}\frac{i^l}{l!(l+2)}\frac{1}{\omega^2}\left((l+1)\kappa_{s^\omega}^{l+2}+\kappa_{p^\omega}^{l+2}\right)|x-y|^{l-1}\rm \textbf{I}\nonumber\\
         & &-\frac{1}{4\pi}\sum_{l=0}^{\infty}\frac{i^l}{l!(l+2)}\frac{(l-1)}{\omega^2}\left(\kappa_{s^\omega}^{l+2}-\kappa_{p^\omega}^{l+2}\right)|x-y|^{l-3}(x-y)\otimes(x-y),
               \end{eqnarray}
from which we can get the gradient 
\begin{eqnarray}\label{gradkupradzeten1}
\hspace{-.5cm}\nabla_y \Gamma^\omega(x,y)&=&-\frac{1}{4\pi}\sum_{l=0}^{\infty}\frac{i^l}{l!(l+2)}\frac{(l-1)}{\omega^2}\left[\left((l+1)\kappa_{s^\omega}^{l+2}+\kappa_{p^\omega}^{l+2}\right)|x-y|^{l-3}(x-y)\otimes\rm \textbf{I}\right.\nonumber\\
&&-\left.\left(\kappa_{s^\omega}^{l+2}-\kappa_{p^\omega}^{l+2}\right)|x-y|^{l-3}\left((l-3)|x-y|^{-2}\otimes^{3}(x-y)+\rm \textbf{I}\otimes(x-y)+(x-y)\otimes\rm \textbf{I}\right)\right].
\end{eqnarray}
\subsection{The representation via double layer potential}\label{DLPR-1}
We start with the following proposition on the solution of the problem (\ref{elaimpoenetrable}-\ref{radiationcela}) via the method of integral equations.
\begin{proposition} \label{existence-of-sigmasdbl}
There exists $a_0>0$, such that if $a< a_0$,\footnote{The condition on $a$ can be replaced by a condition on $\omega$ as it can be seen from the proof.} the solution of the problem (\ref{elaimpoenetrable}-\ref{radiationcela}) is of the form
 \begin{equation}\label{qcimprequiredfrm1dbl}
  U^{t}(x)=U^{i}(x)+\sum_{m=1}^{M}\int_{\partial D_m}\frac{\partial\Gamma^\omega(x,s)}{\partial \nu_m(s)}\sigma_{m} (s)ds,\,x\in\mathbb{R}^{3}\backslash\left(\mathop{\cup}_{m=1}^M \bar{D}_m\right),
\end{equation}
where $\sigma_m\in H^{r}(\partial D_m)$, with $\,r\in[0,\,1]$, for $m=1,2,\dots,M$, 
and  $\frac{\partial}{\partial \nu_m}(\cdot)$ denotes the co-normal derivative on $\partial D_m$ and is defined as
\begin{equation}\label{conormaldblela}
 \frac{\partial}{\partial \nu_m}(\cdot):=\lambda(\div\,\cdot)N_m+\mu(\nabla\cdot+\nabla\cdot^{\top})N_m\quad\text{on }\,\partial D_m,
\end{equation}
with $N_m$ the outward unit normal vector of $\partial D_m$.
\end{proposition}
\begin{proof}{\it{of Proposition  \ref{existence-of-sigmasdbl}}.}
  We look for the solution of the problem (\ref{elaimpoenetrable}-\ref{radiationcela}) of the form \eqref{qcimprequiredfrm1dbl},
  then from the Dirichlet boundary condition \eqref{elagoverningsupport} and the jumps of the double layer potentials, we obtain
\begin{equation}\label{qcimprequiredfrm1bddbl}
\frac{\sigma_{j} (s_j)}{2}+\int_{\partial D_j}\frac{\partial\Gamma^\omega(s_j,s)}{\partial \nu_j(s)}\sigma_{j} (s)ds+\sum_{\substack{m=1\\m\neq j}}^{M}\int_{\partial D_m}\frac{\partial\Gamma^\omega(s_j,s)}{\partial \nu_m(s)}\sigma_{m} (s)ds=-U^i{(s_j)},\,\forall s_j\in \partial D_j,\, j=1,\dots,M.
\end{equation}
One can write the system \eqref{qcimprequiredfrm1bddbl} in a compact form as
$(\frac{1}{2}\textbf{I}+DL+DK)\sigma=-U^{In}$ with $\textbf{I}:=(\textbf{I}_{mj})_{m,j=1}^{M}$, $DL:=(DL_{mj})_{m,j=1}^{M}$ and $DK:=(DK_{mj})_{m,j=1}^{M}$, where
\begin{eqnarray}\label{definition-DL_DK}
\textbf{I}_{mj}=\left\{\begin{array}{ccc}
            I,\, \text{Identity operator} &\hspace{-.2cm} m=j\\
            0,\, \text{zero~operator}~~~~~~ &\hspace{-.2cm} else
           \end{array}\right.,\,
&
DL_{mj}=\left\{\begin{array}{ccc}
            \mathcal{D}_{mj} &\hspace{-.2cm} m=j\\
            0 &\hspace{-.2cm} else
           \end{array}\right.,\,
&
DK_{mj}=\left\{\begin{array}{ccc}
            \mathcal{D}_{mj} &\hspace{-.2cm} m\neq j\\
            0 &\hspace{-.2cm} else
           \end{array}\right.,
\end{eqnarray}
%
\begin{eqnarray}
U^{In}=U^{In}(s_1,\dots,s_M):=\left(U^i(s_1),\dots,U^i(s_M)\right)^T\\
 \mbox{ and }\sigma=\sigma(s_1,\dots,s_M):=\left(\sigma_1(s_1),\dots,\sigma_M(s_M)\right)^T.
 \end{eqnarray}
Here, for the indices  $m$ and $j$ fixed, $\mathcal{D}_{mj}$ is the integral operator
\begin{eqnarray}\label{defofDmjed}
 \mathcal{D}_{mj}(\sigma_j)(t):=\int_{\partial D_j}\frac{\partial\Gamma^\omega(t,s)}{\partial \nu_j(s)}\sigma_j(s)ds.
\end{eqnarray}
The operator $\frac{1}{2}I+\mathcal{D}_{mm}:H^{r}(\partial D_m)\rightarrow H^{r}(\partial D_m)$
is Fredholm with zero index and for $m\neq j$, $\mathcal{D}_{mj}:H^{r}(\partial D_j)\rightarrow H^{r}(\partial D_m)$
is compact for $0\leq r\leq 1$, when $\partial D_m$ has a Lipschitz regularity,
see \cite{MO-MM:TJFAA2000,MS-MM:IMSC2006,MD:InteqnsaOpethe1997}.\footnote{In \cite{MO-MM:TJFAA2000,MS-MM:IMSC2006,MD:InteqnsaOpethe1997}, this property is proved for the case $\omega=0$.
By a perturbation argument, we have the same results for every $\omega$ in $[0,\,\omega_{\max}]$,
assuming that $\omega_{\max}$ is smaller than the first eigenvalue $w_{el}$ of the Dirichlet-Lam\'{e} operator in $D_m$.
By a comparison theorem, see \cite[(6.131) in Lemma 6.3.6]{M-M-N:Springerbook:2013} for instance,
we know that $\mu w_L\leq w_{el}$ where $w_L$
is the first eigenvalue of the Dirichlet-Laplacian operator in $D_m$. Now, we know that $\left(\frac{1}{a}\sqrt[3]{\frac{4\pi}{3}}{\rm j}_{1/2,1}\right)^2\leq w_L$.
Then, we need $\omega_{\max}<\frac{\sqrt{\mu}}{a}\sqrt[3]{\frac{4\pi}{3}}{\rm j}_{1/2,1}$
 which is satisfied if $a<\frac{\sqrt{\mu}}{\omega_{\max}}\sqrt[3]{\frac{4\pi}{3}}{\rm j}_{1/2,1}:=a_0$. 
Here ${\rm j}_{1/2,1}$ is the 1st positive zero of the Bessel function ${\rm J}_{1/2}$.} 
So, $(\frac{1}{2}\textbf{I}+DL+DK):\prod\limits_{m=1}^{M}H^{r}(\partial D_m)\rightarrow \prod\limits_{m=1}^{M}H^{r}(\partial D_m)$ is Fredholm with zero index.
We induce the product of spaces by the maximum of the norms of the space.
To show that $(\frac{1}{2}\textbf{I}+DL+DK)$ is invertible it is enough to show that it is injective. i.e. $(\frac{1}{2}\textbf{I}+DL+DK)\sigma=0$ implies $\sigma=0$.
\newline
Write,
$$\tilde{U}(x)=\sum_{m=1}^{M}\int_{\partial D_m}\frac{\partial\Gamma^\omega(s_j,s)}{\partial \nu_m(s)}\sigma_m(s)ds, \mbox{ in } \mathbb{R}^3\backslash \left(\mathop{\cup}_{m=1}^{M}\bar{D}_m\right)$$
and
$$\tilde{\tilde{U}}(x)=\sum_{m=1}^{M}\int_{\partial D_m}\frac{\partial\Gamma^\omega(s_j,s)}{\partial \nu_m(s)}\sigma_m(s)ds, \mbox{ in } \mathop{\cup}_{m=1}^{M}D_m.$$
Then $\tilde{U}$ satisfies $\Delta^e\tilde{U}+\omega^2\tilde{U}=0$ for $x\in\mathbb{R}^{3}\backslash\left(\mathop{\cup}\limits_{m=1}^{M}\bar{D}_m\right)$,
with K.R.C and $\tilde{U}(x)=0$ on $\mathop{\cup}\limits_{m=1}^{M}\partial D_m$.
Similarly, $\tilde{\tilde{U}}$ satisfies $\Delta^e\tilde{\tilde{U}}+\omega^2\tilde{\tilde{U}}=0$ for $x\in\mathop{\cup}\limits_{m=1}^{M}D_m$ with  $\tilde{\tilde{U}}(x)=0$
on $\mathop{\cup}\limits_{m=1}^{M}\partial D_m$.
Taking the trace on $\partial D_m$, $m=1,\dots,M$,
\begin{eqnarray}\label{nbc1dbl}
\tilde{U}(s)=0&\Longrightarrow&\mathcal{D}_{mm}(\sigma_m)(s)+\frac{\sigma_m(s)}{2}+\sum\limits_{j\neq m}\mathcal{D}_{mj}(\sigma_j)(s)=0
\end{eqnarray}
and
\begin{eqnarray}\label{nbc2dbl}
\tilde{\tilde{U}}(s)=0&\Longrightarrow&\mathcal{D}_{mm}(\sigma_m)(s)-\frac{\sigma_m(s)}{2}+\sum\limits_{j\neq m}\mathcal{D}_{mj}(\sigma_j)(s)=0
\end{eqnarray}
 for $s\in \partial D_m$ and for $m=1,\dots,M$.
Difference between \eqref{nbc1dbl} and \eqref{nbc2dbl} implies that, $\sigma_m=0$ for all $m$. \\
We conclude then that $\frac{1}{2}\textbf{I}+DL+DK:=\frac{1}{2}\textbf{I}+\mathcal{D}:\prod\limits_{m=1}^{M}H^{r}(\partial D_m)\rightarrow \prod\limits_{m=1}^{M}H^{r}(\partial D_m)$ is invertible.
\end{proof}
\subsection{An appropriate estimate of the densities $\sigma_m,\,m=1,\dots,M$}\label{DLPR-2}
From the above theorem, we have the following representation of $\sigma$:
\begin{eqnarray}\label{invDLplusDK}
 \sigma&=&(\frac{1}{2}\textbf{I}+DL+DK)^{-1}U^{In} \nonumber\\
       &=&(\frac{1}{2}\textbf{I}+DL)^{-1}(\textbf{I}+(\frac{1}{2}\textbf{I}+DL)^{-1}DK)^{-1}U^{In} \nonumber\\
       &=&(\frac{1}{2}\textbf{I}+DL)^{-1}\sum_{l=0}^{\infty}\left(-(\frac{1}{2}\textbf{I}+DL)^{-1}DK\right)^{l}U^{In}, \hspace{.5cm} \mbox{if } \left\|(\frac{1}{2}\textbf{I}+DL)^{-1}DK\right\|<1.
\end{eqnarray}
\noindent
The operator $\frac{1}{2}\textbf{I}+DL$ is invertible since it is Fredholm of index zero and injective. 
This implies that
\begin{eqnarray}\label{nrminvDLplusDK}
 \left\|\sigma\right\| 
                                  &\leq&\frac{\left\|(\frac{1}{2}\textbf{I}+DL)^{-1}\right\|}{1-\left\|(\frac{1}{2}\textbf{I}+DL)^{-1}\right\|\left\|DK\right\|}\left\|U^{In}\right\|.
\end{eqnarray}
Here,
\begin{eqnarray}
\left\|DK\right\|&:= &\left\|DK\right\|_{\mathcal{L}\left(\prod\limits_{m=1}^{M}L^{2}(\partial D_m),\prod\limits_{m=1}^{M}L^{2}(\partial D_m)\right)}\nonumber\\
    &\equiv&\max_{m=1}^{M}\sum_{j=1}^{M}\left\|DK_{mj}\right\|_{\mathcal{L}\left(L^{2}(\partial D_j),L^{2}(\partial D_m)\right)}\nonumber\\
    &=&\max\limits_{m=1}^{M}\sum_{\substack{j=1\\j\neq\,m}}^{M}\left\|\mathcal{D}_{mj}\right\|_{\mathcal{L}\left(L^{2}(\partial D_j),L^{2}(\partial D_m)\right)},\label{DKnrm}\\
\left\|(\frac{1}{2}\textbf{I}+DL)^{-1}\right\|&:= &\left\|(\frac{1}{2}\textbf{I}+DL)^{-1}\right\|_{\mathcal{L}\left(\prod\limits_{m=1}^{M}L^{2}(\partial D_m),\prod\limits_{m=1}^{M}L^{2}(\partial D_m)\right)}\nonumber\\
     &\equiv&\max_{m=1}^{M}\sum_{j=1}^{M}\left\|{(\frac{1}{2}\textbf{I}+DL)^{-1}_{mj}}\right\|_{\mathcal{L}\left(L^{2}(\partial D_m),L^{2}(\partial D_j)\right)}\nonumber\\
     &=&\max\limits_{m=1}^{M}\left\|(\frac{1}{2}\textbf{I}+\mathcal{D}_{mm})^{-1}\right\|_{\mathcal{L}\left(L^{2}(\partial D_m),L^{2}(\partial D_m)\right)},\label{invDLnrm}\\
\left\|\sigma\right\|&:=& \left\|\sigma\right\|_{\prod\limits_{m=1}^{M}L^{2}(\partial D_m)}
    \equiv\max\limits_{1\leq m \leq M}\left\|\sigma_{m}\right\|_{L^{2}(\partial D_m)},\label{sigmaU^{In}nrmdbl-elasmall}\\
\mbox{and}\hspace{2cm}
\left\|U^{In}\right\|&:=&\left\|U^{In}\right\|_{\prod\limits_{m=1}^{M}L^{2}(\partial D_m)}
    \equiv\max\limits_{1\leq m \leq M}\left\|U^{i}\right\|_{L^{2}(\partial D_m)}. \label{sigmaU^{In}nrmdbl1-elasmall}
\end{eqnarray}
In the following proposition, we provide conditions under which $\left\|(\frac{1}{2}\textbf{I}+DL)^{-1}\right\|\left\|DK\right\|<1$ and then estimate $\left\|\sigma\right\|$ via \eqref{nrminvDLplusDK}.
\begin{proposition}\label{normofsigmastmtdblela}
There exists a constant $\grave{c}$ depending only on the size of $\Omega$, the
Lipschitz character of $B_m,m=1,\dots,M$, $d_{\max}$, and $\omega_{\max}$ such that if
\begin{eqnarray*}
 \sqrt{M-1}\epsilon<\grave{c}d,&\mbox{ then }& \left\|\sigma_m\right\|_{L^{2}(\partial D_m)}\,\leq\,  c\epsilon
\end{eqnarray*}
where c is a positive constant depending only on the Lipschitz character of $B_m$.

 \end{proposition}
\noindent
\ ~ \ \\
\textit{Proof of Proposition \ref{normofsigmastmtdblela}}.
\ ~ \ \\
For any functions $f,g$ defined on $\partial D_\epsilon$ and $\partial B$ respectively, we define
\begin{eqnarray}\label{elasmalla-wedgevee-defntn}
 (f)^\wedge(\xi)\,:=\,\hat{f}(\xi)\,:=\,f(\epsilon\xi+z)&\mbox{ and }& (g)^\vee(x)\,:=\, \check{g}(x)\,:=\,g\left(\frac{x-z}{\epsilon}\right).
\end{eqnarray}
Let $T_1$ and $T_2$ be an orthonormal basis for the tangent plane to $\partial D_\epsilon$ at x and let $\partial\slash\partial\,T=\sum\limits_{l=1}^{2}\partial\slash{\partial\,T_p}\,~T_p$,
denote the tangential derivative on $\partial D_\epsilon$. Then the space $H^{1}(\partial D_\epsilon)$ is defined as
\begin{eqnarray}\label{defofH1spaceonbdry}
 H^{1}(\partial D_\epsilon)&:=&\{\phi\in\,L^{2}(\partial D_\epsilon);\partial\phi\slash\partial\,T\in\,L^{2}(\partial D_\epsilon)\}.
\end{eqnarray}
We have the following lemma from \cite{DPC-SM13}.
\begin{lemma} \label{L2H1estimates}
Suppose $0<\epsilon\leq1$ and $D_\epsilon:=\epsilon B+z\subset\mathbb{R}^n$. Then for every $\psi \in L^{2}(\partial D_\epsilon)$ and $\phi\in H^{1}(\partial D_\epsilon)$, we have
 \begin{equation}\label{habib2*}
||\psi||_{L^{2}(\partial D_\epsilon)}=\epsilon^\frac{n-1}{2}||\hat{\psi}||_{L^{2}(\partial B)}
\end{equation}
\noindent
and
\begin{equation}\label{habib1*}
\epsilon^{\frac{n-1}{2}}\vert\vert \hat{\phi} \vert\vert_{H^1(\partial B)}~\leq ~\vert\vert \phi \vert\vert_{H^{1}(\partial D_\epsilon)}
~\leq~\epsilon^{\frac{n-3}{2}}\vert\vert \hat{\phi} \vert\vert_{H^{1}(\partial B)}.
\end{equation}
\end{lemma}
\par We divide the rest of the proof of Proposition \ref{normofsigmastmtdblela} into two steps. In the first step, we assume we have a single obstacle and then in the second step we deal with the multiple obstacle case.
\subsubsection{The case of a single obstacle}\label{singleobscasedbl}
Let us consider a single obstacle $D_\epsilon:=\epsilon B+z$. 
Then define the operator
$\mathcal{D}_{ D_\epsilon}:L^2(\partial D_{\epsilon})\rightarrow L^2(\partial D_{\epsilon})$ by
\begin{eqnarray}\label{defofDpartialDeela}
\left(\mathcal{D}_{ D_\epsilon}\psi\right) (s)=\int_{\partial D_\epsilon}\frac{\partial\Gamma^\omega(s,t)}{\partial\nu(t)}\psi(t)dt.
\end{eqnarray}
Following the arguments in the proof of Proposition  \ref{existence-of-sigmasdbl}, the integral operator
$\frac{1}{2}I+\mathcal{D}_{ D_\epsilon}:L^2(\partial D_{\epsilon})\rightarrow L^2(\partial D_{\epsilon})$ is invertible.
If we consider the problem (\ref{elaimpoenetrable}-\ref{radiationcela})
in $\mathbb{R}^{3}\backslash \bar{D}_\epsilon$, we obtain
  $$\sigma=(\frac{1}{2}I+\mathcal{D}_{ D_\epsilon})^{-1}U^{i},\,\text{where}\,DL+DK=:\mathcal{D}_{ D_\epsilon}$$
and then
\begin{equation}\label{estsigm1dblela}
||\sigma||_{L^2(\partial D_{\epsilon})}\leq ||(\frac{1}{2}I+\mathcal{D}_{ D_\epsilon})^{-1}||_{\mathcal{L}\left(L^2(\partial D_{\epsilon}),L^2(\partial D_{\epsilon})\right)}||U^i||_{L^2(\partial D_{\epsilon})}.
\end{equation}
\begin{lemma}\label{rep1dbllayerela}
 Let $\phi,\psi \in L^{2}(\partial D_\epsilon)$. Then,
\begin{equation}\label{rep1dbllayer0ela}
 \mathcal{D}_{ D_\epsilon}\psi= (\mathcal{D}^\epsilon_B \hat{\psi})^\vee,
\end{equation}
\begin{equation}\label{rep1dbllayer1ela}
 \left(\frac{1}{2}I+\mathcal{D}_{ D_\epsilon}\right)\psi=  \left(\left(\frac{1}{2}I+\mathcal{D}^\epsilon_B \right)\hat{\psi}\right)^\vee,
\end{equation}
 \begin{equation}\label{rep1dbllayer2ela}
 {\left(\frac{1}{2}I+\mathcal{D}_{ D_\epsilon}\right)}^{-1}\phi= \left({\left(\frac{1}{2}I+\mathcal{D}^\epsilon_B\right)}^{-1} \hat{\phi}\right)^\vee
\end{equation}
\begin{equation}\label{nrm1dbllayer2ela}
 \left\|{\left(\frac{1}{2}I+\mathcal{D}_{ D_\epsilon}\right)}^{-1}\right\|_{\mathcal{L}\left(L^2(\partial D_\epsilon), L^2(\partial D_\epsilon) \right)}=
\left\|{\left(\frac{1}{2}I+\mathcal{D}^\epsilon_B\right)}^{-1}\right\|_{\mathcal{L}\left(L^2(\partial B), L^2(\partial B) \right)}
\end{equation}
and
\begin{equation}\label{nrm1dbllayer2-1ela}
 \left\|{\left(\frac{1}{2}I+\mathcal{D}_{ D_\epsilon}\right)}^{-1}\right\|_{\mathcal{L}\left(H^1(\partial D_\epsilon), H^1(\partial D_\epsilon) \right)}\leq
\epsilon^{-1}\left\|{\left(\frac{1}{2}I+\mathcal{D}^\epsilon_B\right)}^{-1}\right\|_{\mathcal{L}\left(H^1(\partial B), H^1(\partial B) \right)}
\end{equation}
with $\mathcal{D}^\epsilon_B \hat{\psi}(\xi):=\int_{\partial B}\frac{\partial\Gamma^{\epsilon\omega}(\xi,\eta)}{\partial \nu(\eta)}\hat{\psi}(\eta) d\eta$.
\end{lemma}
\begin{proof}{\it{of Lemma \ref{rep1dbllayerela}}.}
\begin{itemize}
\item We have,
 \begin{eqnarray*}
 \mathcal{D}_{ D_\epsilon}\psi(s)&=&\int_{\partial D_\epsilon}\frac{\partial \Gamma^{\omega}(s,t)}{\partial \nu(t)}\psi(t) dt\\
&=&\int_{\partial D_\epsilon}\left[\lambda\left(\div_t\,\Gamma^{\omega}(s,t)\right)N_t+\mu\left(\nabla_t\Gamma^{\omega}(s,t)+(\nabla_t\Gamma^{\omega}(s,t))^{\top}\right)N_t)\right]\psi(t) dt\nonumber\\
&=&\int_{\partial B}\epsilon^{-2}\left[\lambda\left(\div_\eta\,\Gamma^{\epsilon\omega}(\xi,\eta)\right)N_\eta+\mu\left(\nabla_\eta\Gamma^{\epsilon\omega}(\xi,\eta)+(\nabla_\eta\Gamma^{\epsilon\omega}(\xi,\eta))^{\top}\right)N_\eta\right]\psi(\epsilon\eta+z)\epsilon^2d\eta\nonumber\\
          &=&\int_{\partial B}\frac{\partial \Gamma^{\epsilon\omega}(\xi,\eta)}{\partial \nu(\eta)}\psi(\epsilon\eta+z) d\eta\\
          &=&\mathcal{D}^\epsilon_B \hat{\psi}(\xi).
 \end{eqnarray*}
The above gives us \eqref{rep1dbllayer0ela}. From \eqref{rep1dbllayer0ela}, we can obtain \eqref{rep1dbllayer1ela}.
\item The following equalities
\begin{eqnarray*}
\left(\frac{1}{2}I+\mathcal{D}_{ D_\epsilon}\right)\left({\left(\frac{1}{2}I+\mathcal{D}^\epsilon_B\right)}^{-1} \hat{\phi}\right)^\vee
                   &\,\substack{=\\\eqref{rep1dbllayer1ela}}\,& \left(\left(\frac{1}{2}I+\mathcal{D}^\epsilon_B\right){\left(\frac{1}{2}I+\mathcal{D}^\epsilon_B\right)}^{-1} \hat{\phi}\right)^\vee
                   \,= \, \hat{\phi}^{\vee}\,=\, \phi
\end{eqnarray*}
provide us \eqref{rep1dbllayer2ela}.
\item The following equalities
\begin{eqnarray*}
 \left\|\left(\frac{1}{2}I+\mathcal{D}_{ D_\epsilon}\right)^{-1}\right\|_{\mathcal{L}\left(L^2(\partial D_\epsilon), L^2(\partial D_\epsilon) \right)}&:=&
\substack{Sup\\ \phi(\neq0)\in L^{2}(\partial D_\epsilon)} \frac{\left\|\left(\frac{1}{2}I+\mathcal{D}_{ D_\epsilon}\right)^{-1}\phi\right\|_{L^2(\partial D_\epsilon)}}{||\phi||_{L^2(\partial D_\epsilon)}}\\
&\substack{=\\\eqref{habib2*},\eqref{habib1*}}&\,\substack{Sup\\ \phi(\neq0)\in L^{2}(\partial D_\epsilon)}
\frac{\epsilon~\left\|\left(\left(\frac{1}{2}I+\mathcal{D}_{ D_\epsilon}\right)^{-1}\phi\right)^{\wedge}\right\|_{L^2(\partial B)}}{\epsilon~||\hat{\phi}||_{L^2(\partial B)}}\\
&\substack{=\\ \eqref{rep1dbllayer2ela}}&\substack{Sup\\ \hat{\phi}(\neq0)\in L^{2}(\partial D_\epsilon)}
\frac{\left\| {\left(\frac{1}{2}I+\mathcal{D}^\epsilon_B\right)}^{-1} \hat{\phi}\right\|_{L^2(\partial B)}}{||\hat{\phi}||_{L^2(\partial B)}} \\
&=&\left\|{\left(\frac{1}{2}I+\mathcal{D}^\epsilon_B\right)}^{-1}\right\|_{\mathcal{L}\left(L^2(\partial B), L^2(\partial B) \right)} 
\end{eqnarray*}
provide us \eqref{nrm1dbllayer2ela}. By proceeding in the similar manner we can obtain \eqref{nrm1dbllayer2-1ela} as mentioned below,
\begin{eqnarray*}
 \left\|\left(\frac{1}{2}I+\mathcal{D}_{ D_\epsilon}\right)^{-1}\right\|_{\mathcal{L}\left(H^1(\partial D_\epsilon), H^1(\partial D_\epsilon) \right)}&:=&
\substack{Sup\\ \phi(\neq0)\in H^{1}(\partial D_\epsilon)} \frac{\left\|\left(\frac{1}{2}I+\mathcal{D}_{ D_\epsilon}\right)^{-1}\phi\right\|_{H^1(\partial D_\epsilon)}}{||\phi||_{H^1(\partial D_\epsilon)}}\\
&\substack{\leq\\\eqref{habib2*},\eqref{habib1*}}&\,\substack{Sup\\ \phi(\neq0)\in H^{1}(\partial D_\epsilon)}
\frac{\left\|\left(\left(\frac{1}{2}I+\mathcal{D}_{ D_\epsilon}\right)^{-1}\phi\right)^{\wedge}\right\|_{H^1(\partial B)}}{\epsilon~||\hat{\phi}||_{H^1(\partial B)}}\\
&\substack{=\\ \eqref{rep1dbllayer2ela}}&\epsilon^{-1}\substack{Sup\\ \hat{\phi}(\neq0)\in H^{1}(\partial D_\epsilon)}
\frac{\left\| {\left(\frac{1}{2}I+\mathcal{D}^\epsilon_B\right)}^{-1} \hat{\phi}\right\|_{H^1(\partial B)}}{||\hat{\phi}||_{H^1(\partial B)}} \\
&=&\epsilon^{-1}\left\|{\left(\frac{1}{2}I+\mathcal{D}^\epsilon_B\right)}^{-1}\right\|_{\mathcal{L}\left(H^1(\partial B), H^1(\partial B) \right)}. 
\end{eqnarray*}
\end{itemize}
\end{proof}
\noindent
The next lemma provides us with an estimate of the left hand side of \eqref{nrm1dbllayer2ela} by a constant $C$ with a useful dependence of $C$ in terms of $B$ through its Lipschitz character and $\omega$.
\begin{lemma}\label{lemmanrm1dbllayer31ela}
  The operator norm of $\left(\frac{1}{2}I+\mathcal{D}_{ D_\epsilon}\right)^{-1}:L^{2}(\partial D_\epsilon)\rightarrow L^{2}(\partial D_\epsilon)$
  satisfies the estimate
\begin{eqnarray}\label{nrm1dbllayer31ela}
 \left\|\left(\frac{1}{2}I+\mathcal{D}_{ D_\epsilon}\right)^{-1}\right\|_{\mathcal{L}\left(L^2(\partial D_\epsilon), L^2(\partial D_\epsilon) \right)}
&\leq&\grave{C}_6,
\end{eqnarray}
with
$\grave{C}_6:=\frac{4\pi\left\|\left(\frac{1}{2}I+{\mathcal{D}_{B}^{i_{\omega}}}\right)^{-1}\right\|_{\mathcal{L}\left(L^2(\partial B), L^2(\partial B) \right)}}
{4\pi-\left[\frac{4\lambda+17\mu}{2c_s^4}+\frac{12\lambda+9\mu}{2c_p^4}\right]\omega^2\epsilon^2|\partial B|\left\|\left(\frac{1}{2}I+{\mathcal{D}_{B}^{i_{\omega}}}\right)^{-1}\right\|_{\mathcal{L}\left(L^2(\partial B), L^2(\partial B) \right)}}$.
Here, $\mathcal{D}_{B}^{i_{\omega}}:L^2(\partial B)\rightarrow L^2(\partial B)$ is the double layer potential with the zero frequency.
\end{lemma}
\noindent

Here we should mention that if
$\epsilon^2\leq\frac{\pi}{\left[\frac{4\lambda+17\mu}{2c_s^4}+\frac{12\lambda+9\mu}{2c_p^4}\right]\omega^2|\partial B|\left\|\left(\frac{1}{2}I+{\mathcal{D}_{B}^{i_{\omega}}}\right)^{-1}\right\|_{\mathcal{L}\left(L^2(\partial B), L^2(\partial B) \right)}}$,
then $\grave{C}_6$ is bounded by $\frac{4}{3}\left\|\left(\frac{1}{2}I+{\mathcal{D}_{B}^{i_{\omega}}}\right)^{-1}\right\|_{\mathcal{L}\left(L^2(\partial B), L^2(\partial B) \right)}$,
which is a universal constant depending only on $\partial B$ through its Lipschitz character.
\begin{proof}{\it{of Lemma \ref{lemmanrm1dbllayer31ela}}.} To estimate the operator norm of $\left(\frac{1}{2}I+\mathcal{D}_{ D_\epsilon}\right)^{-1}$ we decompose
$\mathcal{D}_{ D_\epsilon}=:\mathcal{D}_{ D_\epsilon}^\omega=\mathcal{D}_{ D_\epsilon}^{i_{\omega}}+\mathcal{D}_{ D_\epsilon}^{d_{\omega}}$ into
two parts $\mathcal{D}_{ D_\epsilon}^{i_{\omega}}$ ( independent of $\omega$ ) and  $\mathcal{D}_{ D_\epsilon}^{d_{\omega}}$ ( dependent of $\omega$ ) given by
\begin{eqnarray}
\mathcal{D}_{ D_\epsilon}^{i_{\omega}} \psi(s):=\int_{\partial D_\epsilon} \left(\frac{\partial}{\partial \nu(t)}\Gamma^{0}(s,t)\right)\psi(t) dt,\label{def-of_Dikela}\\
\mathcal{D}_{ D_\epsilon}^{d_{\omega}} \psi(s):=\int_{\partial D_\epsilon} \left(\frac{\partial}{\partial \nu(t)}[\Gamma^{\omega}(s,t)-\Gamma^0(s,t)]\right)\psi(t) dt.\label{def-of_Ddkela}
\end{eqnarray}
 With this definition, $\frac{1}{2}I+\mathcal{D}_{ D_\epsilon}^{i_{\omega}}:L^2(\partial D_\epsilon)\rightarrow L^2(\partial D_\epsilon)$ is invertible, see \cite{MO-MM:TJFAA2000,MS-MM:IMSC2006,MD:InteqnsaOpethe1997}. Hence,
$\frac{1}{2}I+\mathcal{D}_{ D_\epsilon}=\left(\frac{1}{2}I+\mathcal{D}_{ D_\epsilon}^{i_{\omega}}\right)\left(I+\left(\frac{1}{2}I+\mathcal{D}_{ D_\epsilon}^{i_{\omega}}\right)^{-1}\mathcal{D}_{ D_\epsilon}^{d_{\omega}}\right)$ and so
\begin{eqnarray}
 &&\hspace{-1cm}\left\|\left(\frac{1}{2}I+\mathcal{D}_{ D_\epsilon}\right)^{-1}\right\|_{\mathcal{L}\left(L^2(\partial D_\epsilon), L^2(\partial D_\epsilon) \right)}\nonumber\\
 &&=\left\|\left(I+\left(\frac{1}{2}I+\mathcal{D}_{ D_\epsilon}^{i_{\omega}}\right)^{-1}\mathcal{D}_{ D_\epsilon}^{d_{\omega}}\right)^{-1}\left(\frac{1}{2}I+\mathcal{D}_{ D_\epsilon}^{i_{\omega}}\right)^{-1}\right\|_{\mathcal{L}\left(L^2(\partial D_\epsilon), L^2(\partial D_\epsilon) \right)}\nonumber \\
&&\leq \left\|\left(I+\left(\frac{1}{2}I+\mathcal{D}_{ D_\epsilon}^{i_{\omega}}\right)^{-1}\mathcal{D}_{ D_\epsilon}^{d_{\omega}}\right)^{-1}\right\|_{\mathcal{L}\left(L^2(\partial D_\epsilon), L^2(\partial D_\epsilon) \right)}
 \left\|\left(\frac{1}{2}I+\mathcal{D}_{ D_\epsilon}^{i_{\omega}}\right)^{-1}\right\|_{\mathcal{L}\left(L^2(\partial D_\epsilon), L^2(\partial D_\epsilon) \right)}.\label{nrm1dbllayer3ela}
\end{eqnarray}
So, to estimate the operator norm of $\left(\frac{1}{2}I+\mathcal{D}_{ D_\epsilon}\right)^{-1}$ one needs to estimate the
operator norm of $\left(I+\left(\frac{1}{2}I+\mathcal{D}_{ D_\epsilon}^{i_{\omega}}\right)^{-1}\mathcal{D}_{ D_\epsilon}^{d_{\omega}}\right)^{-1}$.
In particular one needs to have the knowledge about the operator norms of
 $\left(\frac{1}{2}I+\mathcal{D}_{ D_\epsilon}^{i_{\omega}}\right)^{-1}$ and $\mathcal{D}_{ D_\epsilon}^{d_{\omega}}$ to apply the Neumann series.
 For that purpose,  we can estimate the operator norm of  $\left(\frac{1}{2}I+\mathcal{D}_{ D_\epsilon}^{i_{\omega}}\right)^{-1}$ from \eqref{nrm1dbllayer2ela} by
\begin{eqnarray}\label{optnormdikela}
  \left\|\left(\frac{1}{2}I+\mathcal{D}_{ D_\epsilon}^{i_{\omega}}\right)^{-1}\right\|_{\mathcal{L}\left(L^2(\partial D_\epsilon), L^2(\partial D_\epsilon) \right)}
   &=& \left\|\left(\frac{1}{2}I+\mathcal{D}_{B}^{i_{\omega}}\right)^{-1}\right\|_{\mathcal{L}\left(L^2(\partial B), L^2(\partial B) \right)}.
 \end{eqnarray}
Here $\mathcal{D}^{i_{\omega}}_B \hat{\psi}(\xi):=\int_{\partial B} \left(\frac{\partial}{\partial \nu(\eta)}\Gamma^0(\xi,\eta)\right)\hat{\psi}(\eta) d\eta$.
From the definition of the operator $\mathcal{D}_{ D_\epsilon}^{d_{k}}$ in \eqref{def-of_Ddkela}, 
we deduce that
\begin{eqnarray}\label{D^{d_{k}}ela}
 \mathcal{D}_{ D_\epsilon}^{d_{\omega}} \psi(s)&=&\int_{\partial B} \left(\frac{\partial}{\partial \nu(\eta)}[\Gamma^{\epsilon\omega}(\xi,\eta)-\Gamma^0(\xi,\eta)]\right)\hat{\psi}(\eta) d\eta\nonumber\\
&=&\int_{\partial B}\left[\lambda\left(\div_\eta\,[\Gamma^{\epsilon\omega}(\xi,\eta)-\Gamma^0(\xi,\eta)]\right)N_\eta\right.\nonumber\\
&&\qquad\left.+\mu\left(\nabla_\eta[\Gamma^{\epsilon\omega}(\xi,\eta)-\Gamma^0(\xi,\eta)]+(\nabla_\eta[\Gamma^{\epsilon\omega}(\xi,\eta)-\Gamma^0(\xi,\eta)])^{\top}\right)N_\eta\right]\hat{\psi}(\eta) d\eta\nonumber\\
&=&\int_{\partial B}\left[\lambda\, I_{1}\otimes\,N_\eta+\mu\left(I_2+I_2^{\top}\right)N_\eta\right]\hat{\psi}(\eta) d\eta,
\end{eqnarray}
where the vector $I_1$ and the third order tensor $I_2$ are estimated by using \eqref{kupradzeten1} and \eqref{gradkupradzeten1} as
\begin{eqnarray}
I_1&=&-\frac{\epsilon^2}{4\pi}\sum_{l=2}^{\infty}\frac{\epsilon^{l-2}i^l}{l!(l+2)}\frac{(l-1)}{\omega^2}\left[
-2\kappa_{s^{\omega}}^{l+2}+(l+4)\kappa_{p^{\omega}}^{l+2}\right]|\xi-\eta|^{l-3}(\xi-\eta),\label{elaI_1}\\
I_2&=&-\frac{\epsilon^2}{4\pi}\sum_{l=2}^{\infty}\frac{\epsilon^{l-2}i^l}{l!(l+2)}\frac{(l-1)}{\omega^2}\left[
\left((l+1)\kappa_{s^{\omega}}^{l+2}+\kappa_{p^{\omega}}^{l+2}\right)|\xi-\eta|^{l-3} (\xi-\eta)\otimes\rm \textbf{I}\right. \nonumber\\
       &&  \left.-\left(\kappa_{s^{\omega}}^{l+2}-\kappa_{p^{\omega}}^{l+2}\right)|\xi-\eta|^{l-3}\left((l-3)|\xi-\eta|^{-2}\otimes^3(\xi-\eta)+\rm \textbf{I}\otimes(\xi-\eta)+(\xi-\eta)\otimes\rm \textbf{I}\right)\right].\label{elaI_2}
\end{eqnarray}
Using the observation `$\left\|~|x|^p\right\|_{L^2(D)}\leq\left\|x\right\|^p_{L^2(D)}|D|^{\frac{1-p}{2}}$', we obtain
\begin{eqnarray}\label{modD^{d_{k}}ela}
&&\hspace{-.5cm}\left|\mathcal{D}_{ D_\epsilon}^{d_{\omega}} \psi(s)\right|\nonumber\\
&\leq&\lambda\frac{\epsilon^2}{4\pi}\sum_{l=2}^{\infty}\frac{\epsilon^{l-2}}{l!(l+2)}\frac{(l-1)}{\omega^2}\left(2\kappa_{s^{\omega}}^{l+2}+(l+4)\kappa_{p^{\omega}}^{l+2}\right)
\int_{\partial B}|\xi-\eta|^{l-2}\vert\hat{\psi}(\eta)\vert d\eta+ \nonumber\\
&&2\mu\frac{\epsilon^2}{4\pi}\left[\sum_{l=3}^{\infty}\frac{\epsilon^{l-2}}{l!(l+2)}\frac{(l-1)}{\omega^2}\left(2l\kappa_{s^{\omega}}^{l+2}+l\kappa_{p^{\omega}}^{l+2}\right)
\int_{\partial B}|\xi-\eta|^{l-2}\vert\hat{\psi}(\eta)\vert d\eta+\frac{\left(6\kappa_{s^{\omega}}^{4}+4\kappa_{p^{\omega}}^{4}\right)}{8\omega^2}\int_{\partial B}\vert\hat{\psi}(\eta)\vert d\eta\right] \nonumber\\
&\leq&\lambda\frac{\epsilon^2}{4\pi}\|\hat{\psi}\|_{L^2(\partial B)}\vert\partial B\vert^\frac{1}{2}\left[\sum_{l=2}^{\infty}\frac{\epsilon^{l-2}}{l!(l+2)}\frac{(l-1)}{\omega^2}\left(2\kappa_{s^{\omega}}^{l+2}+(l+4)\kappa_{p^{\omega}}^{l+2}\right)
\left\|\xi-\cdot\right\|^{l-2}_{L^2(\partial B)}\vert\partial B\vert^\frac{2-l}{2}\right]+ \nonumber\\
&&2\mu\frac{\epsilon^2}{4\pi}\|\hat{\psi}\|_{L^2(\partial B)}\vert\partial B\vert^\frac{1}{2}\left[\sum_{l=2}^{\infty}\frac{\epsilon^{l-2}}{(l-2)!(l+2)}\frac{1}{\omega^2}\left(2\kappa_{s^{\omega}}^{l+2}+\kappa_{p^{\omega}}^{l+2}\right)
\left\|\xi-\cdot\right\|^{l-2}_{L^2(\partial B)}\vert\partial B\vert^\frac{2-l}{2}+\frac{\left(\kappa_{s^{\omega}}^{4}+\kappa_{p^{\omega}}^{4}\right)}{4\omega^2}\right] \nonumber\\
&\leq&\omega^2\frac{\epsilon^2}{4\pi}\|\hat{\psi}\|_{L^2(\partial B)}\vert\partial B\vert^\frac{1}{2}\left(\frac{1}{c_s^4}\left[\frac{\mu}{2}+(\lambda+4\mu)\sum_{l=0}^{\infty}\left(\frac{1}{2}\epsilon\kappa_{s^{\omega}}
\left\|\xi-\cdot\right\|_{L^2(\partial B)}\vert\partial B\vert^\frac{-1}{2}\right)^{l}\right]\right.\nonumber\\
&&\hspace{3.5cm}\left.+\frac{1}{c_p^4}\left[\frac{\mu}{2}+(3\lambda+2\mu)\sum_{l=0}^{\infty}\left(\frac{1}{2}\epsilon\kappa_{p^{\omega}}
\left\|\xi-\cdot\right\|_{L^2(\partial B)}\vert\partial B\vert^\frac{-1}{2}\right)^{l}\right]\right) \nonumber\\
&=&\omega^2\frac{\epsilon^2}{4\pi}\|\hat{\psi}\|_{L^2(\partial B)}\vert\partial B\vert^\frac{1}{2}\left(\frac{1}{c_s^4}\left[\frac{\mu}{2}+\frac{\lambda+4\mu}{1-\frac{1}{2}\epsilon\kappa_{s^{\omega}}
\left\|\xi-\cdot\right\|_{L^2(\partial B)}\vert\partial B\vert^\frac{-1}{2}}\right]\right.\nonumber\\
&&\hspace{3.5cm}\left.+\frac{1}{c_p^4}\left[\frac{\mu}{2}+\frac{3\lambda+2\mu}{1-\frac{1}{2}\epsilon\kappa_{p^{\omega}}
\left\|\xi-\cdot\right\|_{L^2(\partial B)}\vert\partial B\vert^\frac{-1}{2}}\right]\right),\,\text{for}\,\epsilon<\frac{2\min{\{c_s,c_p\}}}{\omega_{\max}\max_{m}diam(B_m)}\nonumber\\
&\leq&\grave{C}_1\omega^2\epsilon^{2}
\|\hat{\psi}\|_{L^2(\partial B)}, \text{for } \epsilon\leq\frac{\min{\{c_s,c_p\}}}{\omega_{\max}\max_{m}diam(B_m)},
 \end{eqnarray}
with $\grave{C}_1:=\frac{\vert\partial B\vert^\frac{1}{2}}{4\pi}\left[\frac{4\lambda+17\mu}{2c_s^4}+\frac{12\lambda+9\mu}{2c_p^4}\right]$.
From this we obtain,
\begin{eqnarray*}
 \left\|\mathcal{D}_{ D_\epsilon}^{d_{\omega}} \psi\right\|^2_{L^2(\partial D_\epsilon)}
            &=&\int_{\partial D_\epsilon}\left|\mathcal{D}_{ D_\epsilon}^{d_{\omega}} \psi(s)\right|^2 ds\nonumber \\
            &\substack{\leq\\ \eqref{modD^{d_{k}}ela}}&\int_{\partial D_\epsilon}\left[\grave{C}_1 \omega^2\epsilon^{2}
 \|\hat{\psi}\|_{L^2(\partial B)}\right]^2 ds\nonumber\\
 &=&\grave{C}_1^2\omega^4\epsilon^{6}|\partial B|
 \|\hat{\psi}\|^2_{L^2(\partial B)}.\nonumber
\end{eqnarray*}
Hence
\begin{eqnarray}\label{normD^{d_{k}}ela}
 \left\|\mathcal{D}_{ D_\epsilon}^{d_{\omega}} \psi\right\|_{L^2(\partial D_\epsilon)}&\leq&
\grave{C}_1 \omega^2\epsilon^{3}|\partial B|^\frac{1}{2}
  \|\hat{\psi}\|_{L^2(\partial B)}.
\end{eqnarray}
We estimate the norm of the operator $\mathcal{D}_{ D_\epsilon}^{d_{\omega}}$ as
\begin{eqnarray}\label{fnormD^{d_{k}}L2ela}
 \left\|\mathcal{D}_{ D_\epsilon}^{d_{\omega}}\right\|_{\mathcal{L}\left(L^2(\partial D_\epsilon), L^2(\partial D_\epsilon)\right)}&=&
\substack{Sup\\ \psi(\neq0)\in L^{2}(\partial D_\epsilon)} \frac{||\mathcal{D}_{ D_\epsilon}^{d_{\omega}}\psi||_{L^2(\partial D_\epsilon)}}{||\psi||_{L^2(\partial D_\epsilon)}}\nonumber\\
&\,\substack{\leq\\\eqref{normD^{d_{k}}ela},\,\eqref{habib2*}}\,&
 \substack{Sup\\ \hat{\psi}(\neq0)\in L^{2}(\partial B)} \frac{\grave{C}_1\omega^2\epsilon^{3}|\partial B|^\frac{1}{2}\|\hat{\psi}\|_{L^2(\partial B)}}
{\epsilon~\|\hat{\psi}\|_{L^2(\partial B)}}\nonumber\\
&=&\grave{C}_1\omega^2\epsilon^{2}|\partial B|^\frac{1}{2}.
\end{eqnarray}
Hence, we get
 \begin{eqnarray}\label{norms0dkela}
&  \left\|\left(\frac{1}{2}I+{\mathcal{D}_{ D_\epsilon}^{i_{\omega}}}\right)^{-1}\mathcal{D}_{ D_\epsilon}^{d_{\omega}}\right\|_{\mathcal{L}\left(L^2(\partial D_\epsilon), L^2(\partial D_\epsilon) \right)}\nonumber\\
&\leq&
\left\|\left(\frac{1}{2}I+{\mathcal{D}_{ D_\epsilon}^{i_{\omega}}}\right)^{-1}\right\|_{\mathcal{L}\left(L^2(\partial D_\epsilon), L^2(\partial D_\epsilon) \right)}
\left\|\mathcal{D}_{ D_\epsilon}^{d_{\omega}}\right\|_{\mathcal{L}\left(L^2(\partial D_\epsilon), L^2(\partial D_\epsilon) \right)}\nonumber\\
&\substack{\leq\\
\eqref{optnormdikela},\eqref{fnormD^{d_{k}}L2ela}}&\left\|\left(\frac{1}{2}I+{\mathcal{D}_{ B}^{i_{\omega}}}\right)^{-1}\right\|_{\mathcal{L}\left(L^2(\partial B), L^2(\partial B) \right)}\grave{C}_1\omega^2\epsilon^{2}|\partial B|^\frac{1}{2}\nonumber\\
&\leq&\grave{C}_2\omega^2\epsilon^2 ,
 \end{eqnarray}
where $\grave{C}_2:=\grave{C}_1|\partial B|^\frac{1}{2}\left\|\left(\frac{1}{2}I+{\mathcal{D}_{B}^{i_{\omega}}}\right)^{-1}\right\|_{\mathcal{L}\left(L^2(\partial B), L^2(\partial B) \right)}$.
%
Assuming $\epsilon$ to satisfy the condition $\epsilon<\frac{1}{\sqrt{\grave{C}_2}\omega_{\max}}$, then
$\left\|\left(\frac{1}{2}I+{\mathcal{D}_{ D_\epsilon}^{i_{\omega}}}\right)^{-1}\mathcal{D}_{ D_\epsilon}^{d_{\omega}}\right\|_{\mathcal{L}\left(L^2(\partial D_\epsilon), L^2(\partial D_\epsilon) \right)}<1$
and hence by using the Neumann series we obtain the following bound
\begin{eqnarray*}
 \left\|\left(I+\left(\frac{1}{2}I+{\mathcal{D}_{ D_\epsilon}^{i_{\omega}}}\right)^{-1}\mathcal{D}_{ D_\epsilon}^{d_{\omega}}\right)^{-1}\right\|_{\mathcal{L}\left(L^2(\partial D_\epsilon), L^2(\partial D_\epsilon) \right)}
 &\leq&
 \frac{1}{1-\left\|\left(\frac{1}{2}I+{\mathcal{D}_{ D_\epsilon}^{i_{\omega}}}\right)^{-1}\mathcal{D}_{ D_\epsilon}^{d_{\omega}}\right\|_{\mathcal{L}\left(L^2(\partial D_\epsilon), L^2(\partial D_\epsilon) \right)}}\\
 &\substack{\leq\\ \eqref{norms0dkela}}&\grave{C}_3:=\frac{1}{1-\grave{C}_2 \omega^2\epsilon^2}.\\
\end{eqnarray*}
 By substituting the above and \eqref{optnormdikela} in \eqref{nrm1dbllayer3ela}, we obtain the required result \eqref{nrm1dbllayer31ela}.
\end{proof}
\subsubsection{The multiple obstacle case}\label{gbmulobscasedbl}
\ ~ \ \\
\begin{lemma}\label{stirlingapproxlemma}
 For each $k>0$ and for every $n\in\mathbb{Z}^{+}$ with $n\geq{ke^2}\, [=:N(k)]$ we have $n!\geq k^{n-1}$.
\end{lemma}
\begin{proof}{\it{of Lemma \ref{stirlingapproxlemma}}.} The result is true for $n=1$. The proof goes as follows for $n>1$:
\begin{eqnarray*}
 n\geq {ke^2} &\Longrightarrow&\ln{k}\leq\ln{n}-2\\
 &\footnotemark\Longrightarrow&\ln{k}\leq \frac{\ln\sqrt{2\pi}-n}{n-1}+\frac{\left(n+\frac{1}{2}\right)}{n-1}\ln{n}  \\
 &\Longrightarrow&(n-1)\ln{k}\leq \ln\sqrt{2\pi}+\left(n+\frac{1}{2}\right)\ln{n}-n\\
 &\Longrightarrow& k^{n-1}\leq\sqrt{2\pi\,n}\left(\frac{n}{e}\right)^n,
\end{eqnarray*}
\footnotetext{Since,$\frac{\left(n+\frac{1}{2}\right)}{n-1}>1,\frac{\ln\sqrt{2\pi}}{n-1}>0$ and  $0<\frac{n}{n-1}<2$ for $n>1$.}
Now, we obtain the result using  Stirling’s approximation $n!\sim\sqrt{2\pi\,n}\left(\frac{n}{e}\right)^n$, precisely $\sqrt{2\pi\,n}\left(\frac{n}{e}\right)^n\leq{n!}$,
see \cite{NV:Stirling:AMM:1986} for instance.
\end{proof}
\begin{proposition}\label{propsmjsmmestdblela}
 For $m,j=1,2,\dots,M$, the operator $\mathcal{D}_{mj}:L^2(\partial D_j)\rightarrow L^{2}(\partial D_m)$ defined in Proposition \ref{existence-of-sigmasdbl}, see \eqref{defofDmjed},
enjoys the following estimates,
\begin{itemize}
\item
For $j=m$,
\begin{eqnarray}\label{estinvsmmdblela}
 \left\|\left(\frac{1}{2}I+\mathcal{D}_{mm}\right)^{-1}\right\|_{\mathcal{L}\left(L^2(\partial D_m), L^2(\partial D_m) \right)}&\leq&\grave{C}_{6m}, \label{invfnormD_{ii}1ela}
\end{eqnarray}
where $\grave{C}_{6m}:=\frac{4\pi\left\|\left(\frac{1}{2}I+\mathcal{D}^{i_{\omega}}_{B_m}\right)^{-1}\right\|_{\mathcal{L}\left(L^2(\partial B_m), L^2(\partial B_m) \right)}}
{4\pi-\left[\frac{4\lambda+17\mu}{2c_s^4}+\frac{12\lambda+9\mu}{2c_p^4}\right]\omega^2\epsilon^2|\partial B_m|\left\|\left(\frac{1}{2}I+\mathcal{D}^{i_{\omega}}_{B_m}\right)^{-1}\right\|_{\mathcal{L}\left(L^2(\partial B_m), L^2(\partial B_m) \right)}}$.
\item For $j\neq m$,
\begin{eqnarray}\label{estinvsmjdblela}
\left\|\mathcal{D}_{mj}\right\|_{\mathcal{L}\left(L^2(\partial D_j), L^2(\partial D_m)\right)}
                          &\leq&\left[\frac{\tilde{C}_7}{d^2}+\tilde{C}_8\right]\frac{1}{4\pi}
                                                                             \left|\partial \c{B}\right|\epsilon^{2},\label{fnormD_{ij}21ela}
\end{eqnarray}
where 
\begin{eqnarray*}\left|\partial \c{B} \right|:=\max\limits_m \partial B_m,&\qquad&\tilde{C}_7:=\left(\frac{\lambda+6\mu}{c_s^2}+\frac{2\lambda+6\mu}{c_p^2}\right) \mbox{ and  }\end{eqnarray*}
\begin{eqnarray*}
\tilde{C}_8&:=&\frac{\omega^2}{c_s^4}\left(\frac{\mu}{2}+(\lambda+4\mu)\frac{1-\left(\frac{1}{2}\kappa_{s^{\omega}}diam(\Omega)\right)^{N_{\Omega}}}{1-\frac{1}{2}\kappa_{s^{\omega}}diam(\Omega)}+(\lambda+4\mu)\frac{1}{2^{N_{\Omega}-1}}\right)\nonumber\\
&&\qquad+\frac{\omega^2}{c_p^4}\left(\frac{\mu}{2}+\frac{(3\lambda+4\mu)}{2}\frac{1-\left(\frac{1}{2}\kappa_{p^{\omega}}diam(\Omega)\right)^{N_{\Omega}}}{1-\frac{1}{2}\kappa_{p^{\omega}}diam(\Omega)}+\frac{(3\lambda+4\mu)}{2}\frac{1}{2^{N_{\Omega}-1}}\right)
\end{eqnarray*}
with $N_{\Omega}=[2diam(\Omega)\max\{\kappa_{s^\omega},\kappa_{p^\omega}\}e^2]$, where $[\cdot]$ denotes the integral part.
\end{itemize}
\end{proposition}
\begin{proof}{\it{of Proposition \ref{propsmjsmmestdblela}}.}
The estimate \eqref{estinvsmmdblela} is nothing but \eqref{nrm1dbllayer31ela} of Lemma \ref{lemmanrm1dbllayer31ela}, replacing $B$ by $B_m$, $z$ by $z_m$ and $D_\epsilon$ by $D_m$ respectively. It remains to prove the estimate \eqref{estinvsmjdblela}.
We have
\begin{eqnarray}\label{fnormD_{ij}ela}
 \left\|\mathcal{D}_{mj}\right\|_{\mathcal{L}\left(L^2(\partial D_j), L^2(\partial D_m)\right)}&=&
\substack{Sup\\ \psi(\neq0)\in L^{2}(\partial D_j)} \frac{||\mathcal{D}_{mj}\psi||_{L^2(\partial D_m)}}{||\psi||_{L^2(\partial D_j)}}.
\end{eqnarray}
\noindent
Let $\psi\in L^2(\partial D_j)$ then for $s\in \partial D_m$,  we have
\begin{eqnarray}\label{D_{ij}ela}
\mathcal{D}_{mj}\psi(s)&=&\int_{\partial D_j}\frac{\partial\Gamma^\omega(s,t)}{\partial \nu_j(t)}\psi(t)dt\nonumber\\
&=&\int_{\partial D_j}\left[\lambda\left(\div_t\,[\Gamma^{\omega}(s,t)]\right)N_t+\mu\left(\nabla_t[\Gamma^{\omega}(s,t)]+(\nabla_t[\Gamma^{\omega}(s,t)])^{\top}\right)N_t\right]\psi(t) dt\nonumber\\
&=&\int_{\partial D_j}\left[\lambda\, I_{1'}\otimes\,N_t+\mu\left(I_{2'}+I_{2'}^{\top}\right)N_t\right]{\psi}(t) dt,
\end{eqnarray}
where the vector $I_{1'}$ and the third order tensor $I_{2'}$ are given by
\begin{eqnarray}
I_{1'}&=&-\frac{1}{4\pi}\sum_{l=0}^{\infty}\frac{i^l}{l!(l+2)}\frac{(l-1)}{\omega^2}\left[
-2\kappa_{s^{\omega}}^{l+2}+(l+4)\kappa_{p^{\omega}}^{l+2}\right]|s-t|^{l-3}(s-t),\label{elaI_11}\\
I_{2'}&=&-\frac{1}{4\pi}\sum_{l=0}^{\infty}\frac{i^l}{l!(l+2)}\frac{(l-1)}{\omega^2}\left[
\left((l+1)\kappa_{s^{\omega}}^{l+2}+\kappa_{p^{\omega}}^{l+2}\right)|s-t|^{l-3} (s-t)\otimes\rm \textbf{I}\right. \nonumber\\
       &&  \left.-\left(\kappa_{s^{\omega}}^{l+2}-\kappa_{p^{\omega}}^{l+2}\right)|s-t|^{l-3}\left((l-3)|s-t|^{-2}\otimes^3(s-t)+\rm \textbf{I}\otimes(s-t)+(s-t)\otimes\rm \textbf{I}\right)\right].\label{elaI_21}
\end{eqnarray}
\noindent
\newline
Then, by using Lemma \ref{stirlingapproxlemma}, we estimate
\noindent
\begin{eqnarray}\label{modD_{ij}ela}
\left|\mathcal{D}_{mj}\psi(s)\right|
&\leq&\frac{\lambda}{4\pi}\left[\frac{\left(\kappa_{s^{\omega}}^{2}+2\kappa_{p^{\omega}}^{2}\right)}{\omega^2}\int_{\partial D_j}|s-t|^{-2}\vert\psi(t)\vert dt\right. \nonumber\\
&&\qquad\left.+\sum_{l=2}^{\infty}\frac{1}{l!(l+2)}\frac{(l-1)}{\omega^2}\left(2\kappa_{s^{\omega}}^{l+2}+(l+4)\kappa_{p^{\omega}}^{l+2}\right)\int_{\partial D_j}|s-t|^{l-2}\vert\psi(t)\vert dt\right] \nonumber\\
&&+\frac{2\mu}{4\pi}\left[\frac{3\left(\kappa_{s^{\omega}}^{2}+\kappa_{p^{\omega}}^{2}\right)}{\omega^2}\int_{\partial D_j}|s-t|^{-2}\vert\psi(t)\vert dt
+\frac{\left(6\kappa_{s^{\omega}}^{4}+4\kappa_{p^{\omega}}^{4}\right)}{8\omega^2}\int_{\partial D_j}\vert\psi(t)\vert dt\right.\nonumber\\
&&\qquad\left.+\sum_{l=3}^{\infty}\frac{1}{l!(l+2)}\frac{(l-1)}{\omega^2}\left(2l\kappa_{s^{\omega}}^{l+2}+l\kappa_{p^{\omega}}^{l+2}\right)
\int_{\partial D_j}|s-t|^{l-2}\vert\psi(t)\vert dt\right] \nonumber\\
&\leq&\frac{\lambda}{4\pi}\|\psi\|_{L^2(D_j)}\vert\partial D_j\vert^\frac{1}{2}\nonumber\\
&&\left[\frac{\left(\kappa_{s^{\omega}}^{2}+2\kappa_{p^{\omega}}^{2}\right)}{\omega^2}\frac{1}{d^2_{mj}}+\sum_{l=2}^{\infty}\frac{1}{l!(l+2)}\frac{(l-1)}{\omega^2}\left(2\kappa_{s^{\omega}}^{l+2}+(l+4)\kappa_{p^{\omega}}^{l+2}\right)
diam(\Omega)^{l-2}\right] \nonumber\\
&&+\frac{2\mu}{4\pi}\|\psi\|_{L^2(D_j)}\vert\partial D_j\vert^\frac{1}{2}\nonumber\\
&&\left[\frac{3\left(\kappa_{s^{\omega}}^{2}+\kappa_{p^{\omega}}^{2}\right)}{\omega^2}\frac{1}{d^2_{mj}}+\frac{\left(\kappa_{s^{\omega}}^{4}+\kappa_{p^{\omega}}^{4}\right)}{4\omega^2}+\sum_{l=2}^{\infty}\frac{1}{(l-2)!(l+2)}\frac{1}{\omega^2}\left(2\kappa_{s^{\omega}}^{l+2}+\kappa_{p^{\omega}}^{l+2}\right)
diam(\Omega)^{l-2}\right] \nonumber\\
&\leq&\frac{\lambda}{4\pi}\|\psi\|_{L^2(D_j)}\vert\partial D_j\vert^\frac{1}{2}\left[\frac{1}{d^2_{mj}}\left(\frac{1}{c_s^2}+\frac{2}{c_p^2}\right)+\frac{\omega^2}{c_s^4}\sum_{l=2}^{N_{\Omega}+1}\left(\frac{1}{2}\kappa_{s^{\omega}}diam(\Omega)\right)^{l-2}+\frac{\omega^2}{c_s^4}\right.\nonumber\\
&&\left.\hspace{.5cm}\sum_{l=N_{\Omega}+2}^{\infty}\left(\frac{1}{2}\right)^{l-2}+\frac{3}{2}\frac{\omega^2}{c_p^4}\sum_{l=2}^{N_{\Omega}+1}\left(\frac{1}{2}\kappa_{p^{\omega}}diam(\Omega)\right)^{l-2}+\frac{3}{2}\frac{\omega^2}{c_p^4}\sum_{l=N_{\Omega}+2}^{\infty}\left(\frac{1}{2}\kappa_{p^{\omega}}diam(\Omega)\right)^{l-2}\right] \nonumber\\
&&+\frac{2\mu}{4\pi}\|\psi\|_{L^2(D_j)}\vert\partial D_j\vert^\frac{1}{2}\left[\frac{3}{d^2_{mj}}\left(\frac{1}{c_s^2}+\frac{1}{c_p^2}\right)+\frac{\omega^2}{4}\left(\frac{1}{c_s^4}+\frac{1}{c_p^4}\right)+\frac{2\omega^2}{c_s^4}\sum_{l=2}^{N_{\Omega}+1}\left(\frac{1}{2}\kappa_{s^{\omega}}diam(\Omega)\right)^{l-2}\right.\nonumber\\
&&\left.\hspace{2cm}+\frac{2\omega^2}{c_s^4}\sum_{l=N_{\Omega}+2}^{\infty}\left(\frac{1}{2}\right)^{l-2}+\frac{\omega^2}{c_p^4}\sum_{l=2}^{N_{\Omega}+1}\left(\frac{1}{2}\kappa_{p^{\omega}}diam(\Omega)\right)^{l-2}
+\frac{\omega^2}{c_p^4}\sum_{l=N_{\Omega}+2}^{\infty}\left(\frac{1}{2}\right)^{l-2}\right] \nonumber\\
&=&\frac{1}{4\pi}\|\psi\|_{L^2(D_j)}\vert\partial D_j\vert^\frac{1}{2}\left[\frac{1}{d^2_{mj}}\left(\frac{\lambda+6\mu}{c_s^2}+\frac{2\lambda+6\mu}{c_p^2}\right)\right.\nonumber\\
&&\qquad\left.+\frac{\omega^2}{c_s^4}\left(\frac{\mu}{2}+(\lambda+4\mu)\sum_{l=0}^{N_{\Omega}-1}\left(\frac{1}{2}\kappa_{s^{\omega}}diam(\Omega)\right)^{l}+(\lambda+4\mu)\sum_{l=N_{\Omega}}^{\infty}\left(\frac{1}{2}\right)^{l}\right)\right.\nonumber\\
&&\qquad\left.+\frac{\omega^2}{c_p^4}\left(\frac{\mu}{2}+\frac{(3\lambda+4\mu)}{2}\sum_{l=0}^{N_{\Omega}-1}\left(\frac{1}{2}\kappa_{p^{\omega}}diam(\Omega)\right)^{l}+\frac{(3\lambda+4\mu)}{2}\sum_{l=N_{\Omega}}^{\infty}\left(\frac{1}{2}\right)^{l}\right)\right]\nonumber\\
&\leq&\frac{\epsilon\|\psi\|_{L^2(D_j)}\vert\partial B_j\vert^\frac{1}{2}}{4\pi}\left[\frac{1}{d^2_{mj}}\left(\frac{\lambda+6\mu}{c_s^2}+\frac{2\lambda+6\mu}{c_p^2}\right)\right.\nonumber\\
&&\qquad\left.+\frac{\omega^2}{c_s^4}\left(\frac{\mu}{2}+(\lambda+4\mu)\frac{1-\left(\frac{1}{2}\kappa_{s^{\omega}}diam(\Omega)\right)^{N_{\Omega}}}{1-\frac{1}{2}\kappa_{s^{\omega}}diam(\Omega)}+(\lambda+4\mu)\frac{1}{2^{N_{\Omega}-1}}\right)\right.\nonumber\\
&&\qquad\left.+\frac{\omega^2}{c_p^4}\left(\frac{\mu}{2}+\frac{(3\lambda+4\mu)}{2}\frac{1-\left(\frac{1}{2}\kappa_{p^{\omega}}diam(\Omega)\right)^{N_{\Omega}}}{1-\frac{1}{2}\kappa_{p^{\omega}}diam(\Omega)}+\frac{(3\lambda+4\mu)}{2}\frac{1}{2^{N_{\Omega}-1}}\right)\right]\nonumber\\
&=&\frac{1}{4\pi}\left[\frac{\tilde{C}_7}{d^2_{mj}}+\tilde{C}_8\right]\epsilon~\left|\partial B_j\right|^\frac{1}{2}||\psi||_{L^2(\partial D_j)}
\end{eqnarray}
form which, we get
\begin{eqnarray}\label{fL2normD_{ij}ela}
\left\|\mathcal{D}_{mj}\psi\right\|_{L^2(\partial D_m)}
                                                                  &=&\left(\int_{\partial D_m}\left|\mathcal{D}_{mj}\psi(s)\right|^2 ds\right)^\frac{1}{2}\nonumber\\
                                                                  &\substack{\leq\\\eqref{modD_{ij}ela}}&\left[\frac{\tilde{C}_7}{d^2_{mj}}+\tilde{C}_8\right]\frac{1}{4\pi}\epsilon
                                                                             \left|\partial B_j\right|^\frac{1}{2}||\psi||_{L^2(\partial D_j)}\left(\int_{\partial D_m} ds\right)^\frac{1}{2}\nonumber\\
                                                                  &=&\left[\frac{\tilde{C}_7}{d^2_{mj}}+\tilde{C}_8\right]\frac{1}{4\pi}\epsilon^2
                                                                             \left|\partial B_j\right|^\frac{1}{2}\left|\partial B_m\right|^\frac{1}{2}||\psi||_{L^2(\partial D_j)}.
\end{eqnarray}
\noindent
Substitution of \eqref{fL2normD_{ij}ela} in \eqref{fnormD_{ij}ela} gives us
\begin{eqnarray*}\label{fnormD_{ij}1ela}
 \left\|\mathcal{D}_{mj}\right\|_{\mathcal{L}\left(L^2(\partial D_j), L^2(\partial D_m)\right)}
                         &\leq&\left[\frac{\tilde{C}_7}{d^2_{mj}}+\tilde{C}_8\right]\frac{1}{4\pi\,d^2_{mj}}\epsilon^2
                                                                            \left|\partial B_j\right|^\frac{1}{2}\left|\partial B_m\right|^\frac{1}{2}\nonumber\\
                          &\leq&\left[\frac{\tilde{C}_7}{d^2}+\tilde{C}_8\right]\frac{1}{4\pi}
                                                                             \left|\partial \c{B}\right|\epsilon^{2}.
\end{eqnarray*}
\end{proof}

\begin{proofe}
\textbf{\textit{End of the proof of Proposition  \ref{normofsigmastmtdblela}}.}
By substituting \eqref{invfnormD_{ii}1ela} in \eqref{invDLnrm} and \eqref{fnormD_{ij}21ela} in \eqref{DKnrm}, we obtain
\begin{eqnarray}
\left\|\left(\frac{1}{2}I+DL\right)^{-1}\right\|
     &\leq&\max\limits_{m=1}^{M}\grave{C}_{6m}\label{invDLnrm1ela}
\end{eqnarray}
and
\begin{eqnarray}
\left\|DK\right\|
    &\leq&\frac{M-1}{4\pi}\left[\frac{\tilde{C}_7}{d^2}+\tilde{C}_8\right]\left|\partial \c{B}\right|\epsilon^{2}\label{DKnrm1ela}.
\end{eqnarray}
Hence, \eqref{DKnrm1ela} and \eqref{invDLnrm1ela} jointly provide
\begin{eqnarray}
\left\|\left(\frac{1}{2}I+DL\right)^{-1}\right\|\left\|DK\right\|
    &\leq&\underbrace{\frac{M-1}{4\pi}\left[\frac{\tilde{C}_7}{d^2}+\tilde{C}_8\right]\left(\max\limits_{m=1}^{M}\grave{C}_{6m}\right)\left|\partial \c{B}\right|\epsilon^2}_{=:\grave{C}_s} ,\label{invDLDKnrm1ela}
\end{eqnarray}
By imposing the condition $\left\|\left(\frac{1}{2}I+DL\right)^{-1}\right\|\left\|DK\right\|<1$, we get the following from \eqref{nrminvDLplusDK} and (\ref{sigmaU^{In}nrmdbl-elasmall}-\ref{sigmaU^{In}nrmdbl1-elasmall})
\begin{eqnarray}\label{nrminvDLplusDK2ela}
 \left\|\sigma_m\right\|_{L^{2}(\partial D_m)}\leq\left\|\sigma\right\|
                                  &\leq&\frac{\left\|\left(\frac{1}{2}I+DL\right)^{-1}\right\|}{1-\left\|\left(\frac{1}{2}I+DL\right)^{-1}\right\|\left\|DK\right\|}\left\|U^{In}\right\|\nonumber\\
                   &\leq&\grave{C}_p\left\|\left(\frac{1}{2}I+DL\right)^{-1}\right\| \max\limits_{m=1}^{M}\left\|U^{i}\right\|_{L^{2}(\partial D_m)}\hspace{.25cm} \left( \grave{C}_p\geq\frac{1}{1-\grave{C}_s}\right)\nonumber\\
                   &\substack{\leq\\ \eqref{invDLnrm1ela} }&\grave{\mathrm{C}}  \max\limits_{m=1}^{M}\left\|U^{i}\right\|_{L^{2}(\partial D_m)}\hspace{.25cm} \left(\grave{\mathrm{C}}:=\grave{C}_p\max\limits_{m=1}^{M}\grave{C}_{6m}\right),
\end{eqnarray}
for all $m\in\{1,2,\dots,M\}$. But, for the plane incident wave of the Lam\'e system, $U^{i}(x,\theta):=\alpha\theta\,e^{\left(i\omega\,x\cdot\theta\slash\,c_p\right)}+\beta\theta^\bot\,e^{\left(i\kappa\,x\cdot\theta\slash\,c_s\right)}$, we have
\begin{eqnarray}\label{L2normU^iela}
\left\|U^{i}\right\|_{L^{2}(\partial D_m)}&\leq&(|\alpha|+|\beta|)~\epsilon~\left|\partial B_m\right|^{\frac{1}{2}}\,\leq\,(|\alpha|+|\beta|)~\epsilon~\left|\partial \c{B}\right|^{\frac{1}{2}},\, \forall m=1,2,\dots,M.
\end{eqnarray}
Now by substituting \eqref{L2normU^iela} in \eqref{nrminvDLplusDK2ela}, for each $m=1,\dots,M$, we obtain
\begin{eqnarray}\label{nrmsigmafdblela}
  \left\|\sigma_m\right\|_{L^{2}(\partial D_m)}&\,\leq\,&  \grave{\mathcal{C}}(\omega)\epsilon,
\end{eqnarray}
where $\hspace{.25cm}\grave{\mathcal{C}}(\omega):=\grave{\mathrm{C}} \left|\partial \c{B}\right|^{\frac{1}{2}}(|\alpha|+|\beta|)$.
\par
The condition $\left\|\left(\frac{1}{2}I+DL\right)^{-1}\right\|\left\|DK\right\|<1$ is
satisfied if $\grave{C}_s<1$, i.e.
\begin{eqnarray}\label{Accond-invLK-singl-smalldblela}
\frac{M-1}{4\pi}\left[\frac{\tilde{C}_7}{d^2}+\tilde{C}_8\right]\left|\partial \c{B}\right|\left(\max\limits_{m=1}^{M}\grave{C}_{6m}\right)\epsilon^2<1.
\end{eqnarray}
The condition \eqref{Accond-invLK-singl-smalldblela}
reads as $\sqrt{M-1}\epsilon<\grave{c}d$ where we set $$\grave{c}:=\left[\frac{1}{4\pi}\left[\tilde{C}_7+\tilde{C}_8{d^2_{\max}}\right]|\partial\c{B}|\max\limits_{m=1}^{M}\grave{C}_{6m}\right]^{-\frac{1}{2}}$$
and it serves our purpose in Proposition \ref{normofsigmastmtdblela}  and hence in Theorem \ref{Maintheorem-ela-small-sing}.
\end{proofe}
\subsection{The single layer potential representation and the total charge}
\subsubsection{The single layer potential representation}
For $m=1,2,\dots,M$, let $U^{\sigma_m}$ be the solution of the problem
\begin{equation}\label{elaimpenetrableUsigma}
\begin{cases}
(\Delta^e + \omega^{2})U^{\sigma_m}=0& \mbox{ in }D_m,\\
U^{\sigma_m}=\sigma_m& \mbox{ on } \partial D_m.
\end{cases}
\end{equation}
The function $\sigma_m$ is in $H^{1}(\partial D_m)$, see Proposition \ref{existence-of-sigmasdbl}. Hence $U^{\sigma_m}\in\,H^{\frac{3}{2}}(D_m)$ and then $\left.\frac{\partial\,U^{\sigma_{m}}}{\partial \nu_m}\right|_{\partial D_m}\in\,L^2(\partial D_m)$.
From Proposition  \ref{existence-of-sigmasdbl}, the solution of the problem (\ref{elaimpoenetrable}-\ref{radiationcela}) has the form
 \begin{equation}\label{qcimprequiredfrm2dblela}
  U^{t}(x)=U^{i}(x)+\sum_{m=1}^{M}\int_{\partial D_m}\frac{\partial\Gamma^\omega(x,s)}{\partial \nu_m(s)}\sigma_{m} (s)ds,~x\in\mathbb{R}^{3}\backslash\left(\mathop{\cup}_{m=1}^M \bar{D}_m\right).
\end{equation}
It can be written in terms of single layer potanetial using Gauss's theorem as
 \begin{equation}\label{qcimprequiredfrm3dblela}
  U^{t}(x)=U^{i}(x)+\sum_{m=1}^{M}\int_{\partial D_m}\Gamma^\omega(x,s)\frac{\partial\,U^{\sigma_{m}} (s)}{\partial \nu_m(s)}ds,~x\in\mathbb{R}^{3}\backslash\left(\mathop{\cup}_{m=1}^M \bar{D}_m\right).
\end{equation}
\begin{description}
\item Indeed, by Betti's third identity,
\begin{eqnarray}\label{revised-greensform-DtN-ela-small}
 \int_{\partial D_m}\frac{\partial\Gamma^\omega(x,s)}{\partial \nu_m(s)}\sigma_{m} (s)ds&=&\int_{\partial D_m}\Gamma^\omega(x,s)\frac{\partial\,U^{\sigma_{m}} (s)}{\partial \nu_m(s)}ds\nonumber\\
 &&\hspace{.5cm}+\int_{D_m}\left[U^{\sigma_{m}} (y)\Delta^e\,\Gamma^\omega(x,y)-\Gamma^\omega(x,y)\Delta^e\,U^{\sigma_{m}} (y)\right]\,dy.
\end{eqnarray}
\end{description}
\begin{lemma}\label{thmestmdhoudhonuela}
For $m=1,2,\dots,M$, $U^{\sigma_m}$, the solutions of the problem \eqref{elaimpenetrableUsigma}, satisfies the estimate
\begin{equation}\label{estmdhoudhonuela}
 \left\|\frac{\partial\,U^{\sigma_{m}} (s)}{\partial \nu_m(s)}\right\|_{H^{-1}(\partial D_m)}\le C_7,
\end{equation}
for some constant $C_7$ depending on $B_m$ through its Lipschitz character but it is independent of $\epsilon$.
\end{lemma}
\begin{proof}{\it{of Lemma \ref{thmestmdhoudhonuela}}.}
 For $m=1,2,\dots,M$,  we write $$\mathcal{U}^{m}(x):=U^{\sigma_{m}}(\epsilon\,x+z_m),\forall\,x\in\,B_m.$$
 Then we obtain
\begin{equation}\label{mathcalUmeqnela}
\begin{cases}
 (\Delta^e+\epsilon^2\omega^2)\mathcal{U}^{m}(x)&=\epsilon^2(\Delta^e+\omega^2)U^{\sigma_{m}}(\epsilon\,x+z_m)=0,\text{ for } x\in\,B_m,\\
\hspace{1cm}\mathcal{U}^{m}(\xi)&=U^{\sigma_{m}}(\epsilon\,\xi+z_m)=\sigma(\epsilon\,\xi+z_m),\text{ for } \xi\in\partial\,B_m,
\end{cases}
\end{equation}
and also 
\begin{eqnarray*}
&\frac{\partial\mathcal{U}^{m}(\xi)}{\partial\nu_m(\xi)}
&:=\lambda(\div_\xi\mathcal{U}^{m}(\xi))N_m(\xi)+\mu(\nabla_\xi\mathcal{U}^{m}(\xi)+\nabla_\xi\mathcal{U}^{m}(\xi)^{\top})N_m(\xi)\nonumber\\
&=&\epsilon\left[\lambda(\div\,U^{\sigma_{m}}(\epsilon\,\xi+z_m))N_m(\epsilon\,\xi+z_m)+\mu(\nabla\,U^{\sigma_{m}}(\epsilon\,\xi+z_m)+\nabla\,U^{\sigma_{m}}(\epsilon\,\xi+z_m)^{\top})N_m(\epsilon\,\xi+z_m)\right]\nonumber\\
&=&\epsilon\frac{\partial{U}^{\sigma_m}}{\partial\nu_m}(\epsilon\,\xi+z_m).
\end{eqnarray*}
Hence,
\begin{eqnarray*}
 \left\|\frac{\partial\mathcal{U}^{m}}{\partial\nu_m}\right\|^2_{L^2(\partial B_m)}&=&\int_{\partial B_m}\left|\frac{\partial\mathcal{U}^{m}(\eta)}{\partial\nu_m(\eta)}\right|^2d\eta\nonumber\\
&=&\int_{\partial D_m}\epsilon^2\left|\frac{\partial\,U^{\sigma_{m}}(s)}{\partial\nu_m(s)}\right|^2\epsilon^{-2}ds,~[s:=\epsilon\eta+z_m]\nonumber\\
&=& \left\|\frac{\partial\,U^{\sigma_{m}}}{\partial\nu_m}\right\|^2_{L^2(\partial D_m)},
\end{eqnarray*}
 which gives us
\begin{eqnarray}\label{D2Nbasicela}
 \frac{\left\|\frac{\partial\,U^{\sigma_{m}}}{\partial\nu_m}\right\|_{L^{2}(\partial D_m)}}{||U^{\sigma_{m}}||_{H^1(\partial D_m)}}
&\substack{\le \\ \eqref{habib1*}}&\frac{||\frac{\partial\,\mathcal{U}^{{m}}}{\partial\nu_m}||_{L^2(\partial B_m)}}{\epsilon||\mathcal{U}^{m}||_{H^1(\partial B_m)}}.
\end{eqnarray}
For every function $\zeta_m\in\,H^1(\partial D_m)$, the corresponding $U^{\zeta_{m}}$ exists in $D_m$ as mentioned in \eqref{elaimpenetrableUsigma} and
then the corresponding functions $\mathcal{U}^{m}$ in $B_m$ and the inequality \eqref{D2Nbasicela} will be satisfied by these functions.
Let $\varLambda_{D_m}:H^1(\partial D_m)\rightarrow\,L^2(\partial D_m)$ and $\varLambda_{B_m}:H^1(\partial\, B_m)\rightarrow\,L^2(\partial\, B_m)$ be the Dirichlet to Neumann maps.
Then we get the following estimate from \eqref{D2Nbasicela}.
$$\|\varLambda_{D_m}\|_{\mathcal{L}\left(H^1( \partial D_m), L^{2}( \partial D_m) \right)}\leq\frac{1}{\epsilon}\|\varLambda_{B_m}\|_{\mathcal{L}\left(H^1(\partial\, B_m), L^{2}(\partial\, B_m) \right)}.$$
This implies that,
\begin{eqnarray}
 \frac{\left\|\frac{\partial\,U^{\sigma_{m}}}{\partial\nu_m}\right\|_{H^{-1}(\partial D_m)}}{||U^{\sigma_{m}}||_{L^2(\partial D_m)}}&\le&\|\varLambda_{D_m}^{*}\|_{\mathcal{L}\left(L^2( \partial D_m), H^{-1}( \partial D_m) \right)}\nonumber\\
&=&\|\varLambda_{D_m}\|_{\mathcal{L}\left(H^1( \partial D_m), L^{2}( \partial D_m) \right)}\nonumber\\
&\leq&\frac{1}{\epsilon}\|\varLambda_{B_m}\|_{\mathcal{L}\left(H^1(\partial\, B_m), L^{2}(\partial\, B_m) \right)}.
\end{eqnarray}
Now, by \eqref{nrmsigmafdblela} and \eqref{elaimpenetrableUsigma},
\begin{eqnarray}\label{estmdhoudhonuela-1}
 \left\|\frac{\partial\,U^{\sigma_{m}}}{\partial\nu_m}\right\|_{H^{-1}(\partial D_m)}
&\leq&\grave{\mathcal{C}}(\omega)\|\varLambda_{B_m}\|_{\mathcal{L}\left(H^1(\partial\, B_m), L^{2}(\partial\, B_m) \right)}.
\end{eqnarray}
Hence the result is true as  $\|\varLambda_{B_m}\|_{\mathcal{L}\left(H^1(\partial\, B_m), L^{2}(\partial\, B_m) \right)}$ is bounded by a
constant depending only on $B_m$ through its size and Lipschitz character of $B_m$.
\end{proof}
\begin{definition}
\label{elaQmdefdbl}
We call $\sigma_m\in L^2(\partial D_m)$ satisfying \eqref{qcimprequiredfrm1dbl}, the solution of the problem (\ref{elaimpoenetrable}-\ref{radiationcela}), as elastic surface charge distributions (in short surface charge distributions). Using these surface charge distributions
we define the total charge on each surface $\partial D_m$ denoted by $Q_m$ as
\begin{eqnarray}\label{eladefofQmdbl}
Q_m:=\int_{ \partial D_m} \frac{\partial\,U^{\sigma_{m}} (s)}{\partial \nu_m(s)} ds.
\end{eqnarray}
\end{definition}
\subsubsection{Estimates on the total charge $Q_m,\,m=1,\dots\,M$}\label{elaDLPR-3}
In the following proposition, we provide an approximate of the far-fields in terms of the total charges $Q_m$.
 \begin{proposition}\label{elafarfldthmdbl}
The P-part, $U^\infty_p(\hat{x},\theta)$, and the S-part, $U^\infty_s(\hat{x},\theta)$, of the far-field pattern of the problem (\ref{elaimpoenetrable}-\ref{radiationcela}) have the following asymptotic expansions respectively;
\begin{eqnarray}
U^\infty_p(\hat{x},\theta)&=&\frac{1}{4\pi\,c_p^{2}}(\hat{x}\otimes\hat{x})\sum_{m=1}^{M}\left[e^{-i\frac{\omega}{c_p}\hat{x}\cdot z_{m}}Q_m+O(a^2)\right],\label{x oustdie1 D_mdblelaP}\\
U^\infty_s(\hat{x},\theta)&=& \frac{1}{4\pi\,c_s^{2}}(I- \hat{x}\otimes\hat{x})\sum_{m=1}^{M}\left[e^{-i\frac{\omega}{c_s}\hat{x}\cdot\,z_m}Q_m+O(a^2)\right].\label{x oustdie1 D_mdblelaS}
\end{eqnarray}
if $\kappa_{p^\omega}\,a<1$ and $\kappa_{s^\omega}\,a<1$ where $O(a^2)\,\leq\,\grave{C}_{sp}\omega\,a^2$
with $$\grave{C}_{sp}:=\frac{(|\alpha|+|\beta|)|\partial\c{B}|\grave{\mathrm{C}}\|\varLambda_{B_m}\|_{\mathcal{L}\left(H^1(\partial\, B_m), L^{2}(\partial\, B_m) \right)} }{\max\limits_{1\leq m \leq M} diam(B_m)}\frac{1}{\min\{c_s,c_p\}}.$$
\end{proposition}
\begin{proof}{\it{of Proposition \ref{elafarfldthmdbl}}.}
From \eqref{qcimprequiredfrm3dblela}, we have
\begin{eqnarray*}
 U^{s}(x)&=&\sum_{m=1}^{M}\int_{\partial D_m}\Gamma^{\omega}(x,s)\frac{\partial\,U^{\sigma_{m}} (s)}{\partial \nu_m(s)}ds,\,\text{ for }x\in\mathbb{R}^{3}\backslash\left(\mathop{\cup}\limits_{m=1}^M \bar{D}_m\right). \nonumber
\end{eqnarray*}
Substitution of the asymptotic behavior of the Kupradze tensor at infinity given in \eqref{elafundatensorasymptotic} in the above scattered field and comparing with \eqref{Lamesystemtotalfieldasymptoticsmall},  will allow us to write the
P-part, $U^\infty_p(\hat{x},\theta)$, and the S-part, $U^\infty_s(\hat{x},\theta)$, of the far-field pattern of the problem (\ref{elaimpoenetrable}-\ref{radiationcela}) respectively as;
\begin{eqnarray}
U^{\infty}_p(\hat{x},\theta)
&=&\frac{1}{4\pi\,c_p^{2}}(\hat{x}\otimes\hat{x})\sum_{m=1}^{M}\int_{S_m}e^{-i\kappa_{p^\omega}\hat{x}\cdot\,s}\frac{\partial\,U^{\sigma_{m}} (s)}{\partial \nu_m(s)}ds \nonumber\\
&=&\frac{1}{4\pi\,c_p^{2}}(\hat{x}\otimes\hat{x})\sum_{m=1}^{M}\left[e^{-i\kappa_{p^\omega}\hat{x}\cdot\,z_m}Q_m+\int_{S_m}[e^{-i\kappa_{p^\omega}\hat{x}\cdot\,s}-e^{-i\kappa_{p^\omega}\hat{x}\cdot\,z_m}]\frac{\partial\,U^{\sigma_{m}} (s)}{\partial \nu_m(s)}ds \right]\label{elaxfarawayimpntdbl-p},\\
U^{\infty}_p(\hat{x},\theta)
&=&\frac{1}{4\pi\,c_s^{2}}(I- \hat{x}\otimes\hat{x})\sum_{m=1}^{M}\int_{S_m}e^{-i\kappa_{s^\omega}\hat{x}\cdot\,s}\frac{\partial\,U^{\sigma_{m}} (s)}{\partial \nu_m(s)}ds\nonumber\\
&=&\frac{1}{4\pi\,c_s^{2}}(I- \hat{x}\otimes\hat{x})\sum_{m=1}^{M}\left[e^{-i\kappa_{s^\omega}\hat{x}\cdot\,z_m}Q_m+\int_{S_m}[e^{-i\kappa_{s^\omega}\hat{x}\cdot\,s}-e^{-i\kappa_{s^\omega}\hat{x}\cdot\,z_m}]\frac{\partial\,U^{\sigma_{m}} (s)}{\partial \nu_m(s)}ds \right]\hspace{-.1cm}.\label{elaxfarawayimpntdbl-s}
\end{eqnarray}
  For every $m=1,2,\dots,M$,  we have from Lemma \ref{thmestmdhoudhonuela};
\begin{eqnarray}\label{elaestimationofintsigmadbl}
\left|\int_{\partial D_m}\left|\frac{\partial\,U^{\sigma_{m}} (s)}{\partial \nu_m(s)}\right| ds\right|
&\leq&||1||_{H^{1}(\partial D_m)}\cdot\left\|\frac{\partial\,U^{\sigma_{m}} }{\partial \nu_m}\right\|_{H^{-1}(\partial D_m)}\nonumber\\
&\leq&\epsilon|\partial \c{B}|^\frac{1}{2}\cdot\left\|\frac{\partial\,U^{\sigma_{m}} }{\partial \nu_m}\right\|_{H^{-1}(\partial D_m)}\nonumber\\
&\substack{\leq\\\eqref{estmdhoudhonuela-1}}&\grave{C}a,
\end{eqnarray}
with $\grave{C}:=\frac{\grave{\mathcal{C}}(\omega)|\partial\c{B}|^{\frac{1}{2}}\|\varLambda_{B_m}\|_{\mathcal{L}\left(H^1(\partial\, B_m), L^{2}(\partial\, B_m) \right)} }{\max\limits_{1\leq m \leq M} diam(B_m)}
~\substack{=\\\eqref{nrmsigmafdblela}}~\frac{(|\alpha|+|\beta|)|\partial\c{B}|\grave{\mathrm{C}}\|\varLambda_{B_m}\|_{\mathcal{L}\left(H^1(\partial\, B_m), L^{2}(\partial\, B_m) \right)} }{\max\limits_{1\leq m \leq M} diam(B_m)}$.
It gives us the following estimate for any $\kappa$, i.e. $\kappa=\kappa_{p^\omega} \mbox{ or } \kappa_{s^\omega}$;
\begin{eqnarray}
\left|\int_{\partial D_m}[e^{-i\kappa\hat{x}\cdot\,s}-e^{-i\kappa\hat{x}\cdot\,z_{m}}]\frac{\partial\,U^{\sigma_{m}} (s)}{\partial \nu_m(s)}ds\right|
                         &\leq&\int_{\partial D_m}\left|e^{-i\kappa\hat{x}\cdot\,s}-e^{-i\kappa\hat{x}\cdot\,z_{m}}\right|\left|\frac{\partial\,U^{\sigma_{m}} (s)}{\partial \nu_m(s)}\right|ds\nonumber\\
                         &\leq&\int_{\partial D_m}\sum_{l=1}^{\infty}\kappa^l|s-z_{m}|^l\left|\frac{\partial\,U^{\sigma_{m}} (s)}{\partial \nu_m(s)}\right|ds\nonumber\\
                         &\leq&\int_{\partial D_m}\sum_{l=1}^{\infty}\kappa^l\left(\frac{a}{2}\right)^l\left|\frac{\partial\,U^{\sigma_{m}} (s)}{\partial \nu_m(s)}\right|ds\nonumber\\
                         &\substack{\leq\\ \eqref{elaestimationofintsigmadbl}}&\grave{C}a\sum_{l=1}^{\infty}\kappa^l\left(\frac{a}{2}\right)^l,\nonumber\\
                         &=&\frac{1}{2}\grave{C}\kappa\,a^2\frac{1}{1-\frac{1}{2}\kappa\,a},\,\text{if}\,a<\frac{2}{\kappa_{\max}}\left(\leq\frac{2}{\kappa}\right).
\end{eqnarray}
which means
\begin{eqnarray}\label{elaestimateforexponentsdifdbl}
\int_{\partial D_m}[e^{-i\kappa\hat{x}\cdot\,s}-e^{-i\kappa\hat{x}\cdot\,z_{m}}]\frac{\partial\,U^{\sigma_{m}} (s)}{\partial \nu_m(s)}ds&\leq&\grave{C}\kappa\,a^2,\,\text{for}\,a\leq\frac{1}{\kappa_{\max}}.
\end{eqnarray}
From \eqref{elaestimateforexponentsdifdbl}, it follows that
\begin{eqnarray}
\int_{S_m}[e^{-i\kappa_{p^\omega}\hat{x}\cdot\,s}-e^{-i\kappa_{p^\omega}\hat{x}\cdot\,z_{m}}]\frac{\partial\,U^{\sigma_{m}} (s)}{\partial \nu_m(s)}ds&<&\grave{C}\kappa_{p^\omega}\,a^2,\,\text{ if }\epsilon\leq\frac{\min\{c_s,c_p\}}{\omega_{\max}\max_{m}diam(B_m)}\label{dblestimateforexponentsdifelap}\\
\int_{S_m}[e^{-i\kappa_{s^\omega}\hat{x}\cdot\,s}-e^{-i\kappa_{s^\omega}\hat{x}\cdot\,z_{m}}]\frac{\partial\,U^{\sigma_{m}} (s)}{\partial \nu_m(s)}ds&<&\grave{C}\kappa_{s^\omega}\,a^2,\,\text{ if }\epsilon\leq\frac{\min\{c_s,c_p\}}{\omega_{\max}\max_{m}diam(B_m)}.\label{dblestimateforexponentsdifelas}
\end{eqnarray}
Now substitution of \eqref{dblestimateforexponentsdifelap} in \eqref{elaxfarawayimpntdbl-p} and \eqref{dblestimateforexponentsdifelas} in \eqref{elaxfarawayimpntdbl-s} gives the required results \eqref{x oustdie1 D_mdblelaP}, \eqref{x oustdie1 D_mdblelaS} respectively.
\end{proof}

\begin{lemma}\label{elaQmestbigodbl}
For $m=1,2,\dots,M$, the absolute value of the total charge $Q_m$ on each surface $\partial D_m$ is bounded by $\epsilon$, i.e.
\begin{equation}\label{elaestofQmdbl}
 |Q_m|\leq\grave{\tilde{c}}\epsilon,
\end{equation}
where $\grave{\tilde{c}}:=(|\alpha|+|\beta|)|\partial \c{B}|\grave{\mathrm{C}}\|\varLambda_{B_m}\|_{\mathcal{L}\left(H^1(\partial\, B_m), L^{2}(\partial\, B_m) \right)}$ 
with $\partial\c{B}$ and $\grave{\mathrm{C}}$ are defined in
\eqref{estinvsmjdblela} and \eqref{nrminvDLplusDK2ela} respectively.
\end{lemma}
\begin{proof}{\it{of Lemma \ref{elaQmestbigodbl}}.}
The proof follows as below;
\begin{eqnarray*}\label{Q2dbl}
 |Q_m|&=&\left|\int_{ \partial D_m} \frac{\partial\,U^{\sigma_{m}} (s)}{\partial \nu_m(s)} ds\right|\nonumber\\
    &\leq&|| 1 ||_{H^{1}(\partial D_m)}\left\| \frac{\partial\,U^{\sigma_{m}} (s)}{\partial \nu_m(s)} \right\|_{H^{-1}(\partial D_m)}\nonumber\\
    &\substack{\leq\\\eqref{estmdhoudhonuela-1}}&|| 1 ||_{L^2(\partial D_m)}\grave{\mathcal{C}}(\omega)\,
\|\varLambda_{B_m}\|_{\mathcal{L}\left(H^1(\partial\, B_m), L^{2}(\partial\, B_m) \right)}\nonumber\\
    &\,\substack{\leq\\\eqref{nrmsigmafdblela}}\,&\epsilon\, |\partial \c{B}|\,(|\alpha|+|\beta|)\,\grave{\mathrm{C}}
\,\|\varLambda_{B_m}\|_{\mathcal{L}\left(H^1(\partial\, B_m), L^{2}(\partial\, B_m) \right)}.
\end{eqnarray*}
\end{proof}
\noindent
For $s_m\in \partial D_m$, using the Dirichlet boundary condition \eqref{elagoverningsupport} , we have
\begin{eqnarray}\label{elaQ_mintdbl}
 0&=&U^{t}(s_m)=U^{i}(s_m)+\sum_{j=1}^{M}\int_{\partial D_j}\Gamma^\omega(s_m,s)\frac{\partial\,U^{\sigma_{j}} (s)}{\partial \nu_j(s)}ds\nonumber \\
&=&U^{i}(s_m)+\sum_{\substack{j=1 \\ j\neq m}}^{M}\int_{\partial D_j}\Gamma^\omega(s_m,s)\frac{\partial\,U^{\sigma_{j}} (s)}{\partial \nu_j(s)}ds+\int_{\partial D_m}\Gamma^\omega(s_m,s)\frac{\partial\,U^{\sigma_{m}} (s)}{\partial \nu_m(s)} ds\nonumber \\
&=&U^{i}(s_m)+\sum_{\substack{j=1 \\ j\neq m}}^{M}\Gamma^\omega(s_m,z_j)Q_j\nonumber\\
&&\qquad+\sum_{\substack{j=1 \\ j\neq m}}^{M}\left(\int_{\partial D_j}[\Gamma^\omega(s_m,s)-\Gamma^\omega(s_m,z_j)]\frac{\partial\,U^{\sigma_{j}} (s)}{\partial \nu_j(s)}ds\right)+\int_{\partial D_m}\Gamma^\omega(s_m,s)\frac{\partial\,U^{\sigma_{m}} (s)}{\partial \nu_m(s)}ds.
\end{eqnarray}
To estimate $\int_{\partial D_j}[\Gamma^\omega(s_m,s)-\Gamma^\omega(s_m,z_j)]\frac{\partial\,U^{\sigma_{j}} (s)}{\partial \nu_j(s)} (s)ds$ for $j\neq\,m$,
we have from Taylor series that,
\begin{eqnarray}\label{elataylorphifardbl}
\Gamma^\omega(s_m,s)-\Gamma^\omega(s_m,z_j)=(s-z_j)\cdot R(s_m,s),\,R(s_m,s)=\int_{0}^{1}\nabla_2\Gamma^\omega(s_m,s-\alpha(s-z_j))d\alpha.
\end{eqnarray}
\begin{itemize}
\item 
From the definition of $\Gamma^\omega(x,y)$ and by using the calculations made in \eqref{modD_{ij}ela}, for $s\in \bar{D}_j$ , we obtain
\begin{eqnarray}\label{elaestimatofRxsdbl}
 |R(s_m,s)|\leq\max\limits_{y\in \bar{D}_j}\left|\nabla_y\Gamma^\omega(s_m,y)\right|<\,\frac{1}{4\pi}\left[\frac{C_9}{d^2_{mj}}+C_{10}\right]
\end{eqnarray}
 with
 \begin{center} $C_9:=3\left(\frac{1}{c_s^2}+\frac{1}{c_p^2}\right)$
\\
and  $C_{10}:=2\frac{\omega^2}{c_s^4}\left(\frac{1}{8}+\frac{1-\left(\frac{1}{2}\kappa_{s^\omega}diam(\Omega)\right)^{N_\Omega}}{1-\frac{1}{2}\kappa_{s^\omega}diam(\Omega)}+\frac{1}{2^{N_{\Omega}-1}}\right)
+\frac{\omega^2}{c_p^4}\left(\frac{1}{4}+\frac{1-\left(\frac{1}{2}\kappa_{p^\omega}diam(\Omega)\right)^{N_\Omega}}{1-\frac{1}{2}\kappa_{p^\omega}diam(\Omega)}+\frac{1}{2^{N_{\Omega}-1}}\right).$
\end{center}
\begin{indeed}
 for $x\in \bar{D}_m$ and $s\in \bar{D}_j$, we have from \eqref{gradkupradzeten1};
 \begin{eqnarray}\label{gradkupradzeten1dfnt}
\left|\nabla_x\Gamma^\omega(x,s)\right|
&\leq&\frac{1}{4\pi}\frac{1}{\omega^2}
 \left[3\left(\kappa_{s^\omega}^{2}+\kappa_{p^\omega}^{2}\right)|x-s|^{-2}+\frac{1}{8}\left(6\kappa_{s^\omega}^{4}+4\kappa_{p^\omega}^{4}\right)\right]  \nonumber\\
&&+\frac{1}{4\pi}\sum_{l=3}^{\infty}\frac{1}{(l-2)!(l+2)}\frac{1}{\omega^2}
 \left(2\kappa_{s^\omega}^{l+2}+\kappa_{p^\omega}^{l+2}\right)|x-s|^{l-2}  \nonumber\\
&\leq&\frac{1}{4\pi}\frac{1}{\omega^2}
 \left[\frac{3}{d^2_{mj}}\left(\kappa_{s^\omega}^{2}+\kappa_{p^\omega}^{2}\right)+\frac{1}{4}\left(\kappa_{s^\omega}^{4}+\kappa_{p^\omega}^{4}\right)\right.\nonumber\\
&&\hspace{2cm}\left.+\sum_{l=2}^{\infty}\frac{1}{(l-2)!(l+2)}
 \left(2\kappa_{s^\omega}^{l+2}+\kappa_{p^\omega}^{l+2}\right)diam(\Omega)^{l-2}\right]\nonumber\\
&\leq&  \frac{1}{4\pi}\left[\frac{3}{d^2_{mj}}\left(\frac{1}{c_s^2}+\frac{1}{c_p^2}\right)
 +\frac{1}{4}\left(\frac{\omega^{2}}{c_s^4}+\frac{\omega^{2}}{c_p^4}\right)\right.\nonumber\\
&&\hspace{2cm}\left.+\sum_{l=2}^{\infty}\frac{1}{(l-2)!(l+2)}\left(2\frac{\omega^2}{c_s^4}\kappa_{s^\omega}^{l-2}+\frac{\omega^2}{c_p^4}\kappa_{p^\omega}^{l-2}\right)diam(\Omega)^{l-2}\right] \nonumber\\
&&\mbox{[By recalling $N_{\Omega}=[2diam(\Omega)\max\{\kappa_{s^\omega},\kappa_{p^\omega}\}e^2]$ and using Lemma \ref{stirlingapproxlemma}]}\nonumber\\
&\leq&  \frac{1}{4\pi}\left[\frac{3}{d^2_{mj}}\left(\frac{1}{c_s^2}+\frac{1}{c_p^2}\right)
  +\frac{1}{4}\left(\frac{\omega^{2}}{c_s^4}+\frac{\omega^{2}}{c_p^4}\right)\right.\nonumber\\
&&\hspace{2cm}\left.+2\frac{\omega^2}{c_s^4}\left(\sum_{l=0}^{N_{\Omega}-1}\left(\frac{1}{2}\kappa_{s^\omega}diam(\Omega)\right)^{l}+\sum_{l=N_{\Omega}}^{\infty}\left(\frac{1}{2}\right)^{l}\right)\right.\nonumber\\
&&\hspace{2cm}\left.+\frac{\omega^2}{c_p^4}\left(\sum_{l=0}^{N_{\Omega}-1}\left(\frac{1}{2}\kappa_{p^\omega}diam(\Omega)\right)^{l}+\sum_{l=N_{\Omega}}^{\infty}\left(\frac{1}{2}\right)^{l}\right)\right] \nonumber\\
&=&  \frac{1}{4\pi}\left[\frac{C_9}{d^2_{mj}}+C_{10}\right].
\end{eqnarray}
\end{indeed}

\end{itemize}
For $m,j=1,\dots,M$, and $j\neq\,m$, by making use of \eqref{elaestimatofRxsdbl} and \eqref{elaestimationofintsigmadbl} we obtain the below;
\begin{eqnarray}\label{elaestphismzjdifdbl}
 \left|\int_{\partial D_j}[\Gamma^\omega(s_m,s)-\Gamma^\omega(s_m,z_j)]\frac{\partial\,U^{\sigma_{j}} (s)}{\partial \nu_j(s)}ds\right|&=&\left|\int_{\partial D_j}(s-z_j)\cdot R(s_m,s)\frac{\partial\,U^{\sigma_{j}} (s)}{\partial \nu_j(s)}ds\right|\nonumber\\
                                                               &\leq&\int_{\partial D_j}\left|s-z_j\right| \left|R(s_m,s)\right| \left|\frac{\partial\,U^{\sigma_{j}} (s)}{\partial \nu_j(s)}\right|ds\nonumber\\
                                                               &<\,&\,\frac{a}{4\pi}\left[\frac{C_9}{d^2}+C_{10}\right]\int_{\partial D_j}\left|\frac{\partial\,U^{\sigma_{j}} (s)}{\partial \nu_j(s)}\right|ds\nonumber\\
                                                               &<\,&\grave{C}\frac{a}{4\pi}\left[\frac{C_9}{d^2}+C_{10}\right]a.
\end{eqnarray}
\noindent
Then \eqref{elaQ_mintdbl} can be written as
\begin{equation}\label{{elaqcimsurfacefrmdbl}}
\begin{split}
 \int_{\partial D_m}\Gamma^{0}(s_m,s)\frac{\partial\,U^{\sigma_{m}} (s)}{\partial \nu_m(s)}ds+&\int_{\partial D_m}[\Gamma^{\omega}(s_m,s)-\Gamma^{0}(s_m,s)]\frac{\partial\,U^{\sigma_{m}} (s)}{\partial \nu_m(s)}ds\\
&=-U^{i}(s_m)-\sum_{\substack{j=1 \\ j\neq m}}^{M}\Gamma^{\omega}(s_m,z_j)Q_j+O\left((M-1)\frac{a^2}{d^2}\right).
\end{split}
\end{equation}
By using the Taylor series expansions of the exponential term $e^{i\kappa|s_m-s|}$, the above can also be written as,
\begin{eqnarray}\label{elaqcimsurfacefrm1dbl}
\int_{\partial D_m}\Gamma^0(s_m,s)\frac{\partial\,U^{\sigma_{m}} (s)}{\partial \nu_m(s)}ds+O( a)&=&-U^{i}(s_m)-\sum_{\substack{j=1 \\ j\neq m}}^{M}\Gamma^\omega(s_m,z_j)Q_j+O\left((M-1)\frac{a^2}{d^2}\right).
\end{eqnarray}
\begin{indeed}
\begin{itemize}
\item $\omega\leq\omega_{\max}$ and for $m=1,\dots,M$, we have
\begin{align*}
\left|\int_{\partial D_m}\right.&\left.[\Gamma^{\omega}(s_m,s)-\Gamma^{0}(s_m,s)]\,\frac{\partial\,U^{\sigma_{m}} (s)}{\partial \nu_m(s)}ds\right|\nonumber\\
                             &\leq\int_{\partial D_m}\vert\Gamma^{\omega}(s_m,s)-\Gamma^{0}(s_m,s)\vert\left|\frac{\partial\,U^{\sigma_{m}} (s)}{\partial \nu_m(s)}\right|ds\nonumber\\
&\leq\int_{\partial D_m}\frac{\omega}{4\pi}\left[\frac{2}{c_s^3}\sum\limits_{l=0}^\infty\left(\frac{1}{2}\right)^{l}\kappa_{s^\omega}^l\,diam(D_m)^l\right.
\left.+\frac{1}{c_p^3}\sum\limits_{l=0}^\infty\left(\frac{1}{2}\right)^{l}\kappa_{p^\omega}^l\,diam(D_m)^l\right]\,\left|\frac{\partial\,U^{\sigma_{m}} (s)}{\partial \nu_m(s)}\right|ds\nonumber\\
                             &\substack{\leq \\ \,\eqref{elaestimationofintsigmadbl}}\frac{\omega}{4\pi}\left[\frac{2}{c_s^3}\sum\limits_{l=0}^\infty\left(\frac{1}{2}\right)^{l}\kappa_{s^\omega}^l\,a^l+\frac{1}{c_p^3}\sum\limits_{l=0}^\infty\left(\frac{1}{2}\right)^{l}\kappa_{p^\omega}^l\,a^l\right]\cdot \grave{C}a\nonumber\\
                             &<\frac{\grave{C}}{\pi}\left[\frac{2}{c_s^3}+\frac{1}{c_p^3}\right] \omega\,a,\,\mbox{for } \epsilon\leq\frac{\min\{c_s,c_p\}}{\omega_{\max}\max_{m}diam(B_m)}.
\end{align*}
\end{itemize}
\end{indeed}
Define $U_m:=\int_{\partial D_m}\Gamma^{0}(s_m,s)\frac{\partial\,U^{\sigma_{m}} (s)}{\partial \nu_m(s)}ds,\,s_m\in\,\partial D_m$. Then \eqref{elaqcimsurfacefrm1dbl} can be written as
\begin{eqnarray}\label{elaqcimsurfacefrm2dbl}
U_m&=&-U^{i}(s_m)-\sum_{\substack{j=1 \\ j\neq m}}^{M}\Gamma^\omega(s_m,z_j)Q_j+O( a)+O\left((M-1)\frac{a^2}{d^2}\right).
\end{eqnarray}
We set
\begin{eqnarray}\label{def-ela-1-dbl-revise-defntn-barum}
\bar{U}_m:=-U^{i}(z_m)-\sum\limits_{\substack{j=1 \\ j\neq m}}^{M}{\Gamma}^{\omega}(z_m,z_j)Q_j, \,\mbox{for}\,m=1,\dots,M.
\end{eqnarray}
For $m=1,\dots,M$, let $\bar{\sigma}_m\in L^2(\partial D_m)$ be the solutions of following the integral equation;
\begin{eqnarray}\label{nbc1dbl-revise-ela}
\frac{\sigma_m(s)}{2}+\int_{\partial D_m}\frac{\partial{\Gamma}^{0}(x,s)}{\partial \nu_m(s)}\sigma_{m} (s)ds=\bar{U}_m \quad\mbox{on}\, \partial D_m.
\end{eqnarray}
Remark here that the left hand side of \eqref{nbc1dbl-revise-ela} is the trace, on $\partial{D_m}$, of the double layer potential $\int_{\partial D_m}\frac{\partial{\Gamma}^{0}(x,s)}{\partial \nu_m(s)}\sigma_{m} (s)ds$,
$x\in\mathbb{R}^3\backslash\bar{D}_m$. Dealing in the similar way as we derived \eqref{revised-greensform-DtN-ela-small}, we obtain
\begin{eqnarray}\label{nbc1dbl-1-revise-ela}
 \int_{\partial D_m}\frac{\partial{\Gamma}^{0}(x,s)}{\partial \nu_m(s)}\sigma_{m} (s)ds=\int_{\partial D_m}{\Gamma}^{0}(x,s)\frac{\partial\,U^{\bar{\sigma}_{m}} (s)}{\partial \nu_m(s)}ds,
\end{eqnarray}
with $U^{\bar{\sigma}_{m}}$ are the solutions of \eqref{elaimpenetrableUsigma} replacing the frequency $\omega$ by zero. As single layer potential is continuous up to the boundary,
combining \eqref{nbc1dbl-revise-ela} and \eqref{nbc1dbl-1-revise-ela}, we deduce that
the constant potentials $\bar{U}_m,\,m=1,\dots,M$ satisfy,
\begin{equation}\label{elabarqcimsurfacefrm1dbl}
\int_{\partial D_m}{\Gamma}^{0}(s_m,s)\frac{\partial\,U^{\bar{\sigma}_{m}} (s)}{\partial \nu_m(s)} ds=\bar{U}_m,\,s_m\in\,\partial D_m.
\end{equation}
The total charge on the surface $\partial D_m$ is given by
$$\bar{Q}_m:=\int_{\partial D_m}\frac{\partial\,U^{\bar{\sigma}_{m}} (s)}{\partial \nu_m(s)}ds.$$
For $m=1,\dots,M$, and $l=1,2,3$, [by proceeding in the similar manner as of \eqref{elaqcimsurfacefrm2dbl}-\eqref{elabarqcimsurfacefrm1dbl}], let $\bar{\sigma}_m^{l}\in L^2(\partial D_m)$ be the surface charge distributions which define,
\begin{itemize}
 \item
The constant potentials $\bar{U}_m^{l}\in\mathbb{C}^{3\times1}$ as
\begin{equation}\label{barqcimsurfacefrm1elaldbl}
\int_{\partial D_m}\Gamma^{0}(s_m,s)\frac{\partial\,U^{\bar{\sigma}_{m}^{l}} (s)}{\partial \nu_m(s)}ds=\bar{U}^{l}_m:=-\left(U^{i}(z_m)\right)(l)e_l-\sum\limits_{\substack{j=1 \\ j\neq m}}^{M}\Gamma^{\omega}(z_m,z_j)Q_j(l)e_l, s_m\in\,\partial D_m
\end{equation}
with $e_1=(1 ,\, 0,\, 0)^\top, e_2=(0 ,\, 1,\, 0)^\top$ and $e_3=(0 ,\, 0,\, 1)^\top$.
\item The charge $\bar{Q}_m^{l}\in\mathbb{C}^{3\times1}$ on surface $S_m$ as
$$\bar{Q}_m^{l}:=\int_{\partial D_m}\frac{\partial\,U^{\bar{\sigma}_{m}^{l}} (s)}{\partial \nu_m(s)}ds,$$
\end{itemize}
from which we can notice that  $\bar{U}_m=\sum\limits_{l=1}^{3}\bar{U}^{l}_m$, $\bar{\sigma}_{m}=\sum\limits_{l=1}^{3}\bar{\sigma}_{m}^{l}$ and $\bar{Q}_m=\sum\limits_{l=1}^{3}\bar{Q}^{l}_m$.
\newline
Now, we set the electrical capacitance $\bar{C}_m\in\mathbb{C}^{3\times3}$ for $1\leq m\leq M$ through
\begin{eqnarray}\label{capac-def-ela-dbl}
 \bar{Q}_m^{l}=\bar{C}_m\,\bar{U}_m^{l},\, l=1,2,3 &\text{ and  hence }& \bar{Q}_m=\bar{C}_m\,\bar{U}_m.
\end{eqnarray}
We can write the above also as $\left[\bar{Q}_m^{1}, \bar{Q}_m^{2}, \bar{Q}_m^{3}\right]=\bar{C}_m\left[\bar{U}_m^{1}, \bar{U}_m^{2}, \bar{U}_m^{3}\right]$ for each $m=1,2,\dots,M$.
\newline
\begin{lemma}\label{lemmadifssbQQbCCbeladbl}
 We have the following estimates for $1\leq\,m\leq\,M$;
 \begin{eqnarray}
  \left\|\frac{\partial\,U^{\sigma_{m}} }{\partial \nu_m}-\frac{\partial\,U^{\bar{\sigma}_{m}} }{\partial \nu_m}\right\|_{H^{-1}(\partial D_m)}&=&O\left( a+(M-1)\frac{a^2}{d^2}\right),\label{eladifssbdbl}\\
  Q_m-\bar{Q}_m&=&O\left( a^2+(M-1)\frac{a^3}{d^2}\right)\label{eladifQmddbl}.
 \end{eqnarray}
where the constants appearing in $O(\cdot)$ depend only on the Lipschitz character of $B_m$.
\end{lemma}
\begin{proof}{\it{of Lemma \ref{lemmadifssbQQbCCbeladbl}}.}
By taking the difference between \eqref{elaqcimsurfacefrm2dbl} and \eqref{elabarqcimsurfacefrm1dbl}, we obtain
\begin{eqnarray}\label{dbarqcimsurfacefrm1eladbl}
U_m-\bar{U}_m&=&\int_{\partial D_m}\Gamma^{0}(s_m,s)\left(\frac{\partial\,U^{\sigma_{m}} }{\partial \nu_m}-\frac{\partial\,U^{\bar{\sigma}_{m}} }{\partial \nu_m}\right) (s)ds\nonumber\\
&=&O( a)+O\left((M-1)\frac{a^2}{d^2}\right),\quad s_m\in \partial D_m.
\end{eqnarray}
\begin{indeed} by using Taylor series,
\begin{itemize}
 \item $U^{i}(s_m)-U^{i}(z_m)=O( a)$.
  \item $\Gamma^{\omega}(s_m,z_j)-\Gamma^{\omega}(z_m,z_j)=O\left(\frac{a}{d^2}\right)$ and the asymptoticity of ${Q}_j$. 
\end{itemize}
\end{indeed}
In operator form we can write \eqref{dbarqcimsurfacefrm1eladbl} as,
\begin{eqnarray*}
(\mathcal{S}^{i_\omega}_{D_m})^{*}\left(\frac{\partial\,U^{\sigma_{m}} }{\partial \nu_m}-\frac{\partial\,U^{\bar{\sigma}_{m}} }{\partial \nu_m}\right) (s_m)
&:=&\int_{\partial D_m}\Gamma^{0}(s_m,s)\left(\frac{\partial\,U^{\sigma_{m}} }{\partial \nu_m}-\frac{\partial\,U^{\bar{\sigma}_{m}} }{\partial \nu_m}\right) (s)ds\nonumber\\
&=&O( a)+O\left((M-1)\frac{a^2}{d^2}\right),\quad s_m\in \partial D_m.\nonumber
\end{eqnarray*}
Here, $(\mathcal{S}^{i_\omega}_{D_m})^{*}:H^{-1}(\partial D_m)\rightarrow L^2(\partial D_m)$ is the adjoint of $\mathcal{S}^{i_\omega}_{D_m}:L^{2}(\partial D_m)\rightarrow H^{1}(\partial D_m)$.  We know that,
\begin{eqnarray*}
\left\|(\mathcal{S}^{i_\omega}_{D_m})^{*}\right\|_{\mathcal{L}\left(H^{-1}(\partial D_m),L^2(\partial D_m)\right)}=
 \left\|\mathcal{S}^{i_\omega}_{D_m}\right\|_{\mathcal{L}\left(L^2(\partial D_m),H^1(\partial D_m)\right)}
\end{eqnarray*}
and
\begin{eqnarray*}
\left\|{((\mathcal{S}^{i_\omega}_{D_m})^{*})}^{-1}\right\|_{\mathcal{L}\left(L^2(\partial D_m),H^{-1}(\partial D_m)\right)}=
 \left\|{(\mathcal{S}^{i_\omega}_{D_m})}^{-1}\right\|_{\mathcal{L}\left(H^{1}(\partial D_m),L^2(\partial D_m)\right)},\nonumber
\end{eqnarray*}
then from \eqref{nrm1singulayer2ela} of Lemma \ref{rep1singulayerela}, we obtain 
$\left\|{((\mathcal{S}^{i_\omega}_{D_m})^{*})}^{-1}\right\|_{\mathcal{L}\left(L^2(\partial D_m),H^{-1}(\partial D_m)\right)}=O(a^{-1})$. 
Hence, we get the required results in the following manner.
\begin{itemize}
\item First,
\begin{align*}
 \left\|\frac{\partial\,U^{\sigma_{m}} }{\partial \nu_m}-\frac{\partial\,U^{\bar{\sigma}_{m}} }{\partial \nu_m}\right\|_{H^{-1}(\partial D_m)}
 &\leq\left\|{((\mathcal{S}^{i_\omega}_{D_m})^{*})}^{-1}\right\|_{\mathcal{L}\left(L^2(\partial D_m),H^{-1}(\partial D_m)\right)}
                                                    \left\| O( a)+O\left((M-1)\frac{a^2}{d^2}\right)\right\|_{L^2(\partial D_m)}\nonumber\\
 &=O\left( a+(M-1)\frac{a^2}{d^2}\right).\nonumber
\end{align*}
\item Second, \begin{eqnarray*}
 |Q_m-\bar{Q}_m|&=&\left|\int_{\partial D_m}\left(\frac{\partial\,U^{\sigma_{m}} }{\partial \nu_m}-\frac{\partial\,U^{\bar{\sigma}_{m}} }{\partial \nu_m}\right) (s)ds\right|\nonumber\\
                &\leq& \left\|\frac{\partial\,U^{\sigma_{m}} }{\partial \nu_m}-\frac{\partial\,U^{\bar{\sigma}_{m}} }{\partial \nu_m}\right\|_{H^{-1}(\partial D_m)} \|1\|_{H^1(\partial D_m)}\nonumber\\
                &=& O\left( a^2+(M-1)\frac{a^3}{d^2}\right).
\end{eqnarray*}
\end{itemize}
\end{proof}
\begin{lemma}\label{lemmadifssbQQbCCb1eladbl}
For every $1\leq m\leq M$, the capacitance $\bar{C}_m$  and charge $\bar{Q}_m$ are of the form;
\begin{eqnarray}\label{asymptotCapeladbl}
\bar{C}_m\,=\,\frac{\bar{C}_{B_m}}{\max\limits_{1\leq m \leq M} diam(B_m)}a
 & \mbox{ and }&
\bar{Q}_m\,=\,\frac{\bar{Q}_{B_m}}{\max\limits_{1\leq m \leq M} diam(B_m)}a,
 \end{eqnarray}
where $\bar{C}_{B_m}$ and $\bar{Q}_{B_m}$ are the capacitance and the charge of $B_m$ respectively.
\end{lemma}
\begin{proof}{\it{of Lemma \ref{lemmadifssbQQbCCb1eladbl}}.} 
Take $0<\epsilon\leq 1$, $z\in \mathbb{R}^3$ and 
write, $D_\epsilon:=\epsilon B+z\subset \mathbb{R}^3$.
For $\psi_\epsilon\in L^2(\partial D_\epsilon)$ and $\psi\in L^2(\partial B)$, define the operators $\mathcal{S}^{i_{\omega}}:L^{2}(\partial D_\epsilon)\rightarrow H^{1}(\partial D_\epsilon)$
and $\mathcal{S}^{i_{\omega}}_B:L^{2}(\partial B)\rightarrow H^{1}(\partial B)$ as;
\begin{eqnarray*}
 \mathcal{S}^{i_{\omega}} \psi_\epsilon(x):=\int_{\partial D_\epsilon} \Gamma^{0}(x,y)\psi_\epsilon(y) dy,
&\text{ and }&\mathcal{S}^{i_{\omega}}_B \psi(\xi):=\int_{\partial B} \Gamma^{0}(\xi,\eta)\psi(\eta) d\eta.
\end{eqnarray*}
Define $U^{\psi_\epsilon}$ and $U^{\psi}$ as the functions on $\bar{D}_\epsilon$ and $\bar{B}$ respectively in the similar way of \eqref{elaimpenetrableUsigma}. Then
 the operators
  \begin{eqnarray*}
 \underline{{\mathcal{S}}}^{i_{\omega}} U^{\psi_\epsilon}(x):=\int_{\partial D_\epsilon} \Gamma^{0}(x,y)\frac{\partial\,U^{\psi_\epsilon}} {\partial \nu_y}(y) dy,
&\text{ and }&\underline{{\mathcal{S}}}^{i_{\omega}}_B U^{\psi}(\xi):=\int_{\partial B} \Gamma^{0}(\xi,\eta)\frac{\partial\,U^{\psi}} {\partial \nu_{\eta}}(\eta) d\eta.
\end{eqnarray*}
define the corresponding potentials $\bar{U}_\epsilon$, $\bar{U}_B$ on the surfaces $\partial D_\epsilon$ and $\partial B$ w.r.t the
surface charge distributions $\psi_\epsilon$ and $\psi$ respectively. Let, these potentials be equal to some constant vector$D\in\mathbb{C}^{3\times1}$. 
Let the total charge of these conductors $D_\epsilon$, $B$ are $\bar{Q}_\epsilon$ and $\bar{Q}_B$, and the capacitances are $\bar{C}_\epsilon$ and $\bar{C}_B$ respectively.
Then we can write these as,
\begin{eqnarray*}
 \bar{U}_\epsilon:=\underline{{\mathcal{S}}}^{i_{\omega}} U^{\psi_\epsilon}(x)=D,
&\quad&\bar{U}_B:=\underline{{\mathcal{S}}}^{i_{\omega}}_B U^{\psi}(\xi)=D,
\hspace{.10cm} \forall x\in\partial D_\epsilon,\forall \xi\in\partial B.
\end{eqnarray*}
We have by definitions,
$
\bar{Q}_\epsilon=\int_{\partial D_\epsilon} \frac{\partial\,U^{\psi_\epsilon}} {\partial \nu_y}(y) dy,\,\bar{Q}_B=\int_{\partial B} \frac{\partial\,U^{\psi}} {\partial \nu_\eta}(\eta) d\eta,\text{ and }
\bar{C}_\epsilon\bar{U}_\epsilon=\bar{Q}_\epsilon,\, \bar{C}_B\bar{U}_B=\bar{Q}_B.
$
\newline
Observe that,
\begin{align*}
                 D                       =&\underline{{\mathcal{S}}}^{i_{\omega}} U^{\psi_\epsilon}(x)
       & D                               =&\underline{{\mathcal{S}}}^{i_{\omega}}_B U^{\psi}(\xi)
\nonumber\\
                                         =&\int_{\partial D_\epsilon} \Gamma^{0}(x,y)\frac{\partial\,U^{\psi_\epsilon}} {\partial \nu_y}(y) dy
       &                                 =&\int_{\partial B} \Gamma^{0}(\xi,\eta)\frac{\partial\,U^{\psi}} {\partial \nu_\eta}(\eta) d\eta
\nonumber\\
                                        =&\int_{\partial B} \frac{1}{\epsilon}\Gamma^{0}(\xi,\eta)\frac{1}{\epsilon}\frac{\partial\,U^{\psi_\epsilon}} {\partial \nu_\eta}(\epsilon\eta+z)\epsilon^2 d\eta
       &                                =&\int_{\partial D_\epsilon} \epsilon\Gamma^{0}(x,y)\epsilon\frac{\partial\,U^{\psi}} {\partial \nu_y}\left(y-z\slash\epsilon\right)\epsilon^{-2} dy
\nonumber\\
                                        =&\int_{\partial B} \Gamma^{0}(\xi,\eta)\frac{\partial\,\hat{U}^{\psi_\epsilon}} {\partial \nu_\eta}(\eta) d\eta
       &                                =&\int_{\partial D_\epsilon} \Gamma^{0}(x,y)\frac{\partial\,\check{U}^{\psi}} {\partial \nu_y}(y) dy
\nonumber\\
                                        =&\underline{{\mathcal{S}}}^{i_{\omega}}_B  \hat{U}^{\psi_\epsilon}(\xi).\hspace{.25cm} \left[\hat{\psi}_\epsilon(\eta):=\psi_\epsilon(\epsilon\eta+z)\right]
      &                                 =&\underline{{\mathcal{S}}}^{i_{\omega}}  \check{U}^{\psi}(x).\hspace{.25cm} \left[\check{\psi}(y):=\psi\left(\frac{y-z}{\epsilon}\right)\right]
\end{align*}
Hence,
$
 U^{\psi_\epsilon}=\check{U}^{\psi}
\mbox{ and }
U^{\psi}=\hat{U}^{\psi_\epsilon}.
$
Now we have,
\begin{eqnarray*}
 \bar{Q}_\epsilon&=&\int_{\partial D_\epsilon} \frac{\partial\,U^{\psi_\epsilon}} {\partial \nu_y}(y) dy
                 =\int_{\partial D_\epsilon}\frac{\partial\,\check{U}^{\psi}} {\partial \nu_y}(y) dy,\nonumber\\
                 &=&\int_{\partial B}\frac{1}{\epsilon}\frac{\partial\,\check{U}^{\psi}} {\partial \nu_\eta}(\epsilon\eta+z) \epsilon^2 d\eta
                 =\epsilon\int_{\partial B}\frac{\partial\,\check{U}^{\psi}} {\partial \nu_\eta}(\epsilon\eta+z)d\eta,\nonumber\\
                 &=&\epsilon\int_{\partial B}\frac{\partial\,\hat{\check{U}}^{\psi}} {\partial \nu_\eta}(\eta) d\eta
                 =\epsilon\int_{\partial B}\frac{\partial\,U^{\psi}} {\partial \nu_\eta}(\eta) d\eta,\nonumber\\
                 &=&\epsilon\bar{Q}_B\nonumber\\
\end{eqnarray*}
which gives us,
\begin{eqnarray*}
 \bar{C}_\epsilon\,D&=&\bar{C}_\epsilon\bar{U}_\epsilon\,=\,\bar{Q}_\epsilon\,=\,\epsilon\bar{Q}_B
                 \,=\,\epsilon\bar{C}_B\bar{U}_B\,=\,\epsilon\bar{C}_B\,D.
\end{eqnarray*}
It is true for every constant vector $D$ and hence $\bar{C}_\epsilon=\epsilon\bar{C}_B$.
As we have $D_m=\epsilon B_m+z_m$ and $a=\max\limits_{1\leq m \leq M} diam D_m=\epsilon\max\limits_{1\leq m \leq M} diam(B_m)$, we obtain
\begin{eqnarray*}
\bar{Q}_m\,=\,\epsilon\bar{Q}_{B_m}\,=\,\frac{\bar{Q}_{B_m}}{\max\limits_{1\leq m \leq M} diam(B_m)}a & \mbox{ and }&
 \bar{C}_m\,=\,\epsilon\bar{C}_{B_m}\,=\,\frac{\bar{C}_{B_m}}{\max\limits_{1\leq m \leq M} diam(B_m)}a.
\end{eqnarray*}
\end{proof}
\begin{lemma}\label{fracela-dbl-invert-prop}
For $m=1,2,\dots,M$, the elastic capacitances $\bar{C}_m\in\mathbb{C}^{3\times3}$ defined through \eqref{capac-def-ela-dbl} are non-singular.
\end{lemma}
 \begin{proof}{\it{of Proposition  \ref{fracela-dbl-invert-prop}}.}
As the capacitances $\bar{C}_m$ depend only on the scatterers, let  $\sigma_{m}^{l}\in L^2(\partial D_m)$ be surface charge distributions which define the potentials $e_l$ for $l=1,2,3$. i.e.
\begin{eqnarray}\label{invert-capacitance-slp-eladbl}
 \int_{\partial D_m}\Gamma^{0}(s_m,s)\frac{\partial\,U^{\sigma_{m}} }{\partial \nu_m}^{l} (s)ds=e_l=:U_m^{l},\text{ for } l=1,2,3,\, m=1,\dots,M.
\end{eqnarray}
We also have $\left[\int_{\partial D_m}\frac{\partial\,U^{\sigma_{m}} }{\partial \nu_m}^{1}(s)ds, \int_{\partial D_m}\frac{\partial\,U^{\sigma_{m}} }{\partial \nu_m}^{2}(s)ds, \int_{\partial D_m}\frac{\partial\,U^{\sigma_{m}} }{\partial \nu_m}^{3}(s)ds\right]=\bar{C}_m\left[U_m^{1}, U_m^{2}, U_m^{3}\right]=\bar{C}_m.$
Hence, it is enough if we show that the matrix $\left[\int_{\partial D_m}(\frac{\partial\,U^{\sigma_{m}} }{\partial \nu_m}^{l})_{j}(s) ds\right]_{l,j=1}^{3}$ is invertible. In order to prove this,
assume the linear combination $\sum\limits_{l=1}^{3}a_l\int_{\partial D_m}\frac{\partial\,U^{\sigma_{m}} }{\partial \nu_m}^{l}(s) ds=0$ for the scalars $a_l\in\mathbb{C}$. 
From \eqref{invert-capacitance-slp-eladbl}, we can deduce that
 $$\int_{\partial D_m}\int_{\partial D_m}\Gamma^{0}(s_{m_1},s_{m_2})\left(\sum\limits_{l=1}^{3}a_l\frac{\partial\,U^{\sigma_{m}} }{\partial \nu_m}^{l}(s_{m_2})\right)\cdot\frac{\partial\,U^{\sigma_{m}} }{\partial \nu_m}^{j} (s_{m_1})ds_{m_1}ds_{m_2}=0,\,j=1,2,3,$$
 and hence $$\int_{\partial D_m}\int_{\partial D_m}\Gamma^{0}(s_{m_1},s_{m_2})\left(\sum\limits_{l=1}^{3}a_l\frac{\partial\,U^{\sigma_{m}} }{\partial \nu_m}^{l}(s_{m_2})\right)\cdot\left(\sum\limits_{j=1}^{3}a_j\frac{\partial\,U^{\sigma_{m}} }{\partial \nu_m}^{j} (s_{m_1})\right)ds_{m_1}ds_{m_2}=0.$$
The positivity of the single layer operator implies, $\sum\limits_{l=1}^{3}a_l\frac{\partial\,U^{\sigma_{m}} }{\partial \nu_m}^{l}(s)=0,\, s\in \partial D_m$.
\par Again now by making use of \eqref{invert-capacitance-slp-eladbl}, we deduce
 $$\sum_{l=1}^{3}a_le_l=\int_{\partial D_m}\Gamma^{0}(s_{m},s)\left(\sum\limits_{l=1}^{3}a_l\frac{\partial\,U^{\sigma_{m}} }{\partial \nu_m}^{l}(s)\right)ds=0,\,s_m\in \partial D_m,$$
 and hence $a_l=0$ for $l=1,2,3$.
\end{proof}
\begin{proposition} \label{fracqfracc-aceladbl}For $m=1,2,\dots,M$, the total charge $\bar{Q}_m$ on each surface $\partial D_m$ of the small scatterer $D_m$ can be calculated from the algebraic system
  \begin{eqnarray}\label{elafracqfracdbl}
 \bar{C}_m^{-1}\bar{Q}_m &=&-U^{i}(z_m)-\sum_{\substack{j=1 \\ j\neq m}}^{M} \Gamma^{\omega}(z_m,z_j)\bar{C}_j(\bar{C}_j^{-1}\bar{Q}_j),
  \end{eqnarray}
with an error of order $O\left((M-1)\frac{ a^2}{d}+(M-1)^2\frac{a^3}{d^3}\right)$.
\end{proposition}
\begin{proof}{\it{of Proposition  \ref{fracqfracc-aceladbl}}.}
We can rewrite \eqref{elabarqcimsurfacefrm1dbl} as
\begin{eqnarray*}
 \bar{C}_m^{-1}\bar{Q}_m &=&-U^{i}(z_m)-\sum_{\substack{j=1 \\ j\neq m}}^{M} \Gamma^{\omega}(z_m,z_j)Q_j\nonumber\\
                             &=&-U^{i}(z_m)-\sum_{\substack{j=1 \\ j\neq m}}^{M} \Gamma^{\omega}(z_m,z_j)\bar{Q}_j-\sum_{\substack{j=1 \\ j\neq m}}^{M} \Gamma^{\omega}(z_m,z_j)(Q_j-\bar{Q}_j)\nonumber\\
                             &=&-U^{i}(z_m)-\sum_{\substack{j=1 \\ j\neq m}}^{M} \Gamma^{\omega}(z_m,z_j)\bar{Q}_j
+O\left((M-1)\frac{ a^2}{d}+(M-1)^2\frac{a^3}{d^3}\right),\nonumber\\
\end{eqnarray*}

where we used \eqref{eladifQmddbl} and the fact $\Gamma^{\omega}(z_m,z_j)=O\left(\frac{ 1}{d}+\omega\right)$, $\omega\leq\omega_{\max}$ and $d\leq d_{\max}$.
\begin{indeed}
 \begin{eqnarray}\label{modgamma_{ij}ela}
\left|\Gamma^\omega(z_m,z_j)\right|
&\leq& \frac{1}{4\pi}\frac{1}{\omega^2}\left(\kappa_{s^\omega}^{2}+\kappa_{p^\omega}^{2}\right)|z_m-z_j|^{-1}
   +\frac{1}{4\pi}\sum_{l=1}^{\infty}\frac{1}{(l-1)!(l+2)}\frac{1}{\omega^2}
    \left(2\kappa_{s^\omega}^{l+2}+\kappa_{p^\omega}^{l+2}\right)|z_m-z_j|^{l-1} \nonumber\\
&\leq&  \frac{1}{4\pi}\frac{1}{\omega^2}\left[\frac{1}{d_{mj}}\left(\kappa_{s^\omega}^{2}+\kappa_{p^\omega}^{2}\right)\right.
   \left.+\sum_{l=1}^{\infty}\frac{1}{(l-1)!(l+2)}\left(2\kappa_{s^\omega}^{l+2}+\kappa_{p^\omega}^{l+2}\right)diam(\Omega)^{l-1}\right]\nonumber\\
&\leq&  \frac{1}{4\pi}\left[\frac{1}{d_{mj}}\left(\frac{1}{c_s^2}+\frac{1}{c_p^2}\right)\right.
  \left.+\sum_{l=1}^{\infty}\frac{1}{(l-1)!(l+2)}\left(2\frac{\omega}{c_s^3}\kappa_{s^\omega}^{l-1}+\frac{\omega}{c_p^3}\kappa_{p^\omega}^{l-1}\right)diam(\Omega)^{l-1}\right] \nonumber\\
&&\mbox{[By recalling $N_{\Omega}=[2diam(\Omega)\max\{\kappa_{s^\omega},\kappa_{p^\omega}\}e^2]$ and using Lemma \ref{stirlingapproxlemma}]}\nonumber\\
&\leq&  \frac{1}{4\pi}\left[\frac{1}{d_{mj}}\left(\frac{1}{c_s^2}+\frac{1}{c_p^2}\right)\right.
 \left.+2\frac{\omega}{c_s^3}\left(\sum_{l=1}^{N_{\Omega}}\left(\frac{1}{2}\kappa_{s^\omega}diam(\Omega)\right)^{l-1}+\sum_{l=N_{\Omega}+1}^{\infty}\frac{1}{2^{l-1}}\right)\right.\nonumber\\
&&\hspace{5cm}\left.+\frac{\omega}{c_p^3}\left(\sum_{l=1}^{N_{\Omega}}\left(\frac{1}{2}\kappa_{p^\omega}diam(\Omega)\right)^{l-1}+\sum_{l=N_{\Omega}+1}^{\infty}\frac{1}{2^{l-1}}\right)\right] \nonumber\\
&=&  \frac{1}{4\pi}\left[\frac{1}{d_{mj}}\left(\frac{1}{c_s^2}+\frac{1}{c_p^2}\right)\right.
 \left.+2\frac{\kappa_{s^\omega}}{c_s^2}\left(\frac{1-\left(\frac{1}{2}\kappa_{s^\omega}diam(\Omega)\right)^{N_\Omega}}{1-\frac{1}{2}\kappa_{s^\omega}diam(\Omega)}+\frac{1}{2^{N_{\Omega}-1}}\right)\right.\nonumber\\
&&\hspace{5cm}\left.+\frac{\kappa_{p^\omega}}{c_p^2}\left(\frac{1-\left(\frac{1}{2}\kappa_{p^\omega}diam(\Omega)\right)^{N_\Omega}}{1-\frac{1}{2}\kappa_{p^\omega}diam(\Omega)}+\frac{1}{2^{N_{\Omega}-1}}\right)\right]\nonumber\\
&=&\frac{1}{4\pi}\left[\frac{C_7}{d_{mj}}+C_8\right]
\end{eqnarray}
with
 $$C_7:=\left[\frac{1}{c_s^2}+\frac{2}{c_p^2}\right]$$
and $$C_8:=2\frac{\kappa_{s^\omega}}{c_s^2}\left(\frac{1-\left(\frac{1}{2}\kappa_{s^\omega}diam(\Omega)\right)^{N_\Omega}}{1-\frac{1}{2}\kappa_{s^\omega}diam(\Omega)}
+\frac{1}{2^{N_{\Omega}-1}}\right)+\frac{\kappa_{p^\omega}}{c_p^2}\left(\frac{1-\left(\frac{1}{2}\kappa_{p^\omega}diam(\Omega)\right)^{N_\Omega}}{1-\frac{1}{2}\kappa_{p^\omega}diam(\Omega)}+\frac{1}{2^{N_{\Omega}-1}}\right).$$
\end{indeed}

\end{proof}
\subsubsection{The algebraic system}\label{sec-algebraicsys-small}
Define the  algebraic system,
  \begin{eqnarray}\label{fracqcfracela}
 \bar{C}_m^{-1}\tilde{Q}_m &:=&-U^{i}(z_m)-\sum_{\substack{j=1 \\ j\neq m}}^{M} \Gamma^{\omega}(z_m,z_j)\bar{C}_j(\bar{C}_j^{-1}\tilde{Q}_j),
  \end{eqnarray}
 for all $m=1,2,\dots,M$. 
It can be written in a compact form as
\begin{equation}\label{compacfrm1ela}
 \mathbf{B}\tilde{Q}=U^I,
\end{equation}
\noindent
where $\tilde{Q},U^I \in \mathbb{C}^{3M\times 1}\mbox{ and } \mathbf{B}\in\mathbb{C}^{3M\times 3M}$ are defined as
\begin{eqnarray}
\mathbf{B}:=\left(\begin{array}{ccccc}
   -\bar{C}_1^{-1} &-\Gamma^{\omega}(z_1,z_2)&-\Gamma^{\omega}(z_1,z_3)&\cdots&-\Gamma^{\omega}(z_1,z_M)\\
-\Gamma^{\omega}(z_2,z_1)&-\bar{C}_2^{-1}&-\Gamma^{\omega}(z_2,z_3)&\cdots&-\Gamma^{\omega}(z_2,z_M)\\
 \cdots&\cdots&\cdots&\cdots&\cdots\\
-\Gamma^{\omega}(z_M,z_1)&-\Gamma^{\omega}(z_M,z_2)&\cdots&-\Gamma^{\omega}(z_M,z_{M-1}) &-\bar{C}_M^{-1}
   \end{array}\right),\nonumber\\
\nonumber\\
 \tilde{Q}:=\left(\begin{array}{cccc}
    \tilde{Q}_1^\top & \tilde{Q}_2^\top & \ldots  & \tilde{Q}_M^\top
   \end{array}\right)^\top \text{ and }
U^I:=\left(\begin{array}{cccc}
     U^i(z_1)^\top & U^i(z_2)^\top& \ldots &  U^i(z_M)^\top
   \end{array}\right)^\top.
\nonumber
\end{eqnarray}
The above linear algebraic system is solvable for the 3D vectors $\tilde{Q}_j,~1\leq j\leq M$, when the matrix $\mathbf{B}$ is invertible.
We discuss its invertibility in Section \ref{Solvability-of-the-linear-algebraic-system-elastic-small}.
\par
Now, the difference between \eqref{elafracqfracdbl} and \eqref{fracqcfracela} produce the following
  \begin{eqnarray}\label{qcdiftildeela}
  \bar{C}_m^{-1} (\bar{Q}_m-\tilde{Q}_m) &=&-\sum_{\substack{j=1 \\ j\neq m}}^{M} \Gamma^{\omega}(z_m,z_j)\left(\bar{Q}_j-\tilde{Q}_j\right)+O\left((M-1)\frac{ a^2}{d}+(M-1)^2\frac{a^3}{d^3}\right).
  \end{eqnarray}
for $m=1,2,\dots,M$.
Considering the above system of equations \eqref{qcdiftildeela} in the place of \eqref{fracqcfracela} and then
by making use of the Corollary \ref{Mazyawrkthmela-cor} and the fact that acoustic capacitances of the sactterers are
bounded above and below by their diameters multiplied by constants which depend only on the Lipschitz character of $B_m$'s, see \cite[Lemma 2.11 and Remark 2.23]{DPC-SM13}, we obtain
 \begin{eqnarray}\label{unncmaybeela}
\sum_{m=1}^{M}(\bar{Q}_m-\tilde{Q}_m)&=&O\left(M(M-1)\frac{ a^3}{d}+M(M-1)^2\frac{a^4}{d^3}\right).
\end{eqnarray}
\subsection{End of the proof of Theorem \ref{Maintheorem-ela-small-sing}}\label{sec-mainthrmproof-small}
\begin{proofe}
The use of \eqref{eladifQmddbl}, \eqref{unncmaybeela} in
 \eqref{x oustdie1 D_mdblelaP} and \eqref{x oustdie1 D_mdblelaS} allows us to represent the asymptotic expansions of the P part, $U^\infty_p(\hat{x},\theta)$, and the S part, $U^\infty_p(\hat{x},\theta)$, of the far-field pattern of the problem (\ref{elaimpoenetrable}-\ref{radiationcela}) in terms of
$\tilde{Q}_m$ respectively as below;
 \begin{eqnarray}
  U^\infty_p(\hat{x},\theta)
  &=&\frac{1}{4\pi\,c_p^{2}}(\hat{x}\otimes\hat{x})\sum_{m=1}^{M}e^{-i\frac{\omega}{c_p}\hat{x}\cdot z_{m}}\left[Q_m+O(a^2)\right]\nonumber\\
  &=&\frac{1}{4\pi\,c_p^{2}}(\hat{x}\otimes\hat{x})\sum_{m=1}^{M}e^{-i\frac{\omega}{c_p}\hat{x}\cdot z_{m}}\left[[\tilde{Q}_m+(Q_m-\bar{Q}_m)+(\bar{Q}_m-\tilde{Q}_m)]+O(a^2)\right]\nonumber\\
  &=&\frac{1}{4\pi\,c_p^{2}}(\hat{x}\otimes\hat{x})
  \left(\sum_{m=1}^{M}e^{-i\frac{\omega}{c_p}\hat{x}\cdot z_{m}}\left[\tilde{Q}_m+O\left(a^2+(M-1)\frac{a^3}{d^2}\right)\right]+\sum_{m=1}^{M}e^{-i\frac{\omega}{c_p}\hat{x}\cdot z_{m}}(\bar{Q}_m-\tilde{Q}_m)\right)\nonumber\\
  &=&\frac{1}{4\pi\,c_p^{2}}(\hat{x}\otimes\hat{x})\left[\sum_{m=1}^{M}e^{-i\frac{\omega}{c_p}\hat{x}\cdot z_{m}}\tilde{Q}_m
      +O\left(M a^2+M(M-1)\frac{a^3}{d^2}
+M(M-1)^2\frac{a^4}{d^3}\right)\right],
\\
 U^\infty_s(\hat{x},\theta)
 &=& \frac{1}{4\pi\,c_s^{2}}(I- \hat{x}\otimes\hat{x})\sum_{m=1}^{M}e^{-i\frac{\omega}{c_s}\hat{x}\cdot\,z_m}\left[Q_m+O(a^2)\right]\nonumber\\
 &=& \frac{1}{4\pi\,c_s^{2}}(I- \hat{x}\otimes\hat{x})\left[\sum_{m=1}^{M}e^{-i\frac{\omega}{c_s}\hat{x}\cdot\,z_m}\tilde{Q}_m+O\left(M a^2+M(M-1)\frac{a^3}{d^2}
+M(M-1)^2\frac{a^4}{d^3}\right)\right].
  \end{eqnarray}
Hence, Theorem \ref{Maintheorem-ela-small-sing} is proved by setting $\bar{\sigma}_m:=\frac{\bar{\sigma}_m}{\bar{U}_m}$ as the surface density which defines $\tilde{Q}_m$. Finally, let us remark that
\begin{enumerate}
   \item The constant $\grave{c}:=\left[\frac{1}{4\pi}\left[\tilde{C}_7+\tilde{C}_8{d^2_{\max}}\right]|\partial\c{B}|\max\limits_{m=1}^{M}\grave{C}_{6m}\right]^{-\frac{1}{2}}$ appearing in Proposition \ref{normofsigmastmtdblela}
will serve our purpose in Theorem \ref{Maintheorem-ela-small-sing} by defining $c_0:=\grave{c}\,\max\limits_{ 1\leq\,m\leq\,M } diam (B_m)$ respectively.
\item The coefficients $\bar{\sigma}_m {\bar{U}_m}^{-1},\,\tilde{Q}_m,\, \bar{C}_m$ plays the roles of $\sigma_m,\,Q_m,\,  C_m$ respectively in Theorem \ref{Maintheorem-ela-small-sing}.
\item The constant appearing in $O\left(M a^2+M(M-1)\frac{a^3}{d^2}+M(M-1)^2\frac{a^4}{d^3}\right)$ is
 $$C^e\max\left\{1+\frac{\max\limits_{1\leq m \leq M}\bar{C}_{B_m}}{\max\limits_{1\leq m \leq M} diam(B_m)}\frac{C_7+C_8d_{\max}}{4\pi},1+\frac{\grave{C}~\omega}{C^e\min\{c_s,c_p\}}\right\}$$
 with
 $C^e:=\frac{\max\limits_{1\leq m \leq M}\left\|{\mathcal{S}^{i_{\omega}}_{B_m}}^{-1}\right\|_{\mathcal{L}\left(H^1(\partial B_m), L^2(\partial B_m) \right)}|\partial\c{B}|}{\max\limits_{1\leq m \leq M} diam(B_m)}
 \max\left\{\left(\frac{\vert\alpha\vert+\vert\beta\vert}{\min\{c_s,c_p\}}+\frac{\grave{C}}{\pi}\left[\frac{2}{c_s^3}+\frac{1}{c_p^3}\right]\right)\omega,\quad\right.$ $\left.\frac{\grave{C}}{4\pi}[C_9+C_{10}d^2_{\max}]\right\}$.
The constants $|\partial\c{B}|$ and $\grave{C}$ are defined in
Proposition \ref{propsmjsmmestdblela} and Proposition \ref{elafarfldthmdbl} respectively.

 \item The constant $a_0$ appearing in \eqref{conditions-elasma} of Theorem \ref{Maintheorem-ela-small-sing} is the minimum among $\frac{1}{\omega_{\max}}\min \left\{c_s,\,c_p\right\}$,
 and
$\frac{2\sqrt{\pi}\max\limits_{1\leq m \leq M} diam(B_m)}{\omega_{\max}\left(\left[\frac{4\lambda+17\mu}{2c_s^4}+\frac{12\lambda+9\mu}{2c_p^4}\right]|\partial\c{B}|\max\limits_{1\leq m \leq M}\left\|\left(\frac{1}{2}I+{\mathcal{D}_{B_m}^{i_{\omega}}}\right)^{-1}\right\|_{\mathcal{L}\left(L^2(\partial B_m),
L^2(\partial B_m) \right)}\right)^\frac{1}{2}}$.
 \item The constant $c_1$ appearing in \eqref{invertibilityconditionsmainthm-ela} of Theorem \ref{Maintheorem-ela-small-sing} is
 $\frac{5\pi}{3}\frac{\mu}{(\lambda+2\mu)^2}\frac{\min\limits_{1\leq m \leq M} C^a(B_m)}{\max\limits_{1\leq m \leq M}C^a(B_m)}\frac{\max\limits_{1\leq m \leq M} diam(B_m)}{\max\limits_{1\leq m \leq M}C^a(B_m)}$ with ${C}^a(B_m)$ denoting the acoustic capacitance of the bodies $B_m$
 and it follows from 
 Corollary \ref{Mazyawrkthmela-cor} and from \cite[Lemma 2.11]{DPC-SM13}.
\end{enumerate}
From the last points, we see that the constants appearing in Theorem \ref{Maintheorem-ela-small-sing} depend only on
 $d_{\max}$, $\omega_{\max}$, $\lambda$, $\mu$ and $B_m$'s through their diameters, capacitances and the norms of the boundary operators ${\mathcal{S}^{i_{\omega}}_{B_m}}^{-1}:
H^1(\partial B_m) \rightarrow L^2(\partial B_m)$, $\left(\frac{1}{2}I+
{\mathcal{D}_{B_m}^{i_{\omega}}}\right)^{-1}:L^2(\partial B_m)\rightarrow L^2(\partial B_m)$ and
$\varLambda_{B_m}: H^1(\partial\, B_m) \rightarrow L^{2}(\partial\, B_m)$. As it was explained in the acoustic case in \cite[Remark 2.23]{DPC-SM13},
the capacitances and the bounds of the operators  ${\mathcal{S}^{i_{\omega}}_{B_m}}^{-1}$ and $\left(\frac{1}{2}I+{\mathcal{D}_{B_m}^{i_{\omega}}}\right)^{-1}$
depend on $B_m$'s actually only through their Lipschitz character.
\end{proofe}

\subsection{Proof of corollary \ref{corMaintheorem-ela-small-sing}}
\begin{proofe}
 For $m=1,\dots,M$ fixed, we distinguish between the obstacles $D_j$, $j\neq\,m$ which are near to $D_m$ from the ones which are far from $D_m$ as follows.  Let $\Omega_m$, $1\leq\,m\leq\,M$ be the balls
of center $z_m$ and of radius $(\frac{a}{2}+d^\alpha)$ with $0<\alpha\leq1$. The bodies lying in $\Omega_m$ will fall into the category, $N_m$, of near by obstacles
 and the others into the category, $F_m$, of far obstacles to $D_m$. Since the obstacles $D_m$ are balls with same diameter, the number of obstacles
 near by $D_m$ will not exceed $\left(\frac{a+2d^\alpha}{a+d}\right)^3$   $\left[=\frac{\frac{4}{3}\pi\left((a+2d^\alpha)/2\right)^3}{\frac{4}{3}\pi\left((a+d)/2\right)^3}\right]$.\\
 \par With this observation, instead of (\ref{x oustdie1 D_m farmainp}-\ref{x oustdie1 D_m farmains}), the P and the S parts of the far field
 will have the asymptotic expansions (\ref{x oustdie1 D_m farmainp-near}-\ref{x oustdie1 D_m farmains-near}).
 \begin{indeed}\\
 $\bullet$ For the bodies $D_j\in\,N_m,j\neq{m},$ we have the estimate \eqref{elaestimatofRxsdbl}
  but for the bodies $D_j\in F_m$, we obtain the following estimate 
\begin{eqnarray}\label{elaestimatofRxsdbl-near}
 |R(s_m,s)|\leq\max\limits_{y\in \bar{D}_j}\left|\nabla_y\Gamma^\omega(s_m,y)\right|<\,\frac{1}{4\pi}\left[\frac{C_9}{d^{2\alpha}_{mj}}+C_{10}\right]
\end{eqnarray}
$\bullet$ Due to the estimates \eqref{elaestimatofRxsdbl} and \eqref{elaestimatofRxsdbl-near},
corresponding changes will take place in (\ref{elaestphismzjdifdbl}-\ref{elaqcimsurfacefrm1dbl}), \eqref{elaqcimsurfacefrm2dbl}, (\ref{eladifssbdbl}-\ref{eladifQmddbl})
and in (\ref{elafracqfracdbl}-\ref{modgamma_{ij}ela}) which inturn modify (\ref{qcdiftildeela}-\ref{unncmaybeela}) and
hence the asymptotic expansion \eqref{x oustdie1 D_m farmainp} as follows
 \begin{eqnarray}\label{x oustdie1 D_m farmainp-recent}
U^\infty_p(\hat{x},\theta)&=&\frac{1}{4\pi\,c_p^{2}}(\hat{x}\otimes\hat{x})\left[\sum_{m=1}^{M}e^{-i\frac{\omega}{c_p}\hat{x}\cdot z_{m}}Q_m\right.+
O\left(M a^2+M(M-1)\frac{a^3}{d^{2\alpha}}+M\left(\frac{a+2d^\alpha}{a+d}\right)^3\frac{a^3}{d^2}\right.\nonumber\\
&&\left.\left.+M(M-1)^2\frac{a^4}{d^{3\alpha}}+M(M-1)\left(\frac{a+2d^\alpha}{a+d}\right)^3\frac{a^4}{d^{2+\alpha}}\right.\right.\nonumber\\
&&\left.\left.+M(M-1)\left(\frac{a+2d^\alpha}{a+d}\right)^3\frac{a^4}{d^{2\alpha+1}}+M\left(\frac{a+2d^\alpha}{a+d}\right)^6\frac{a^4}{d^3}\right)\right].
 \end{eqnarray}
$\bullet$ Since $\kappa\leq\kappa_{\max}$, $d\leq d^\alpha,\, 0<\alpha\leq1$ and $\frac{a}{d}<\infty$, we have
$$\left(\frac{a+2d^\alpha}{a+d}\right) =d^{\alpha-1}\frac{ad^{-\alpha}+2}{ad^{-1}+1}=O(d^{\alpha-1}),$$
which can be used to derive \eqref{x oustdie1 D_m farmainp-near} from \eqref{x oustdie1 D_m farmainp-recent}. In the similar way, we can obtain \eqref{x oustdie1 D_m farmains-near}.
Finally, it is easily seen that the above analysis applies also for non-flat Lipschitz domains $D_m$ by using the double inclusions (\ref{non-flat-condition}) and the fact that $t_m$'s are uniformly
bounded from below by a positive constant.
 \end{indeed}
\end{proofe}

\section{Solvability of the linear-algebraic system \eqref{compacfrm1ela}}\label{Solvability-of-the-linear-algebraic-system-elastic-small}
The main object of this section is to give a sufficient condition in order to get the invertibility of the linear algebraic system \eqref{compacfrm1ela}.
To achieve this, first we state the following lemma which estimates the eigenvalues of the elastic capacitance matrix of each scatterer in terms of its acoustic capacitance.
\begin{lemma}\label{capacitance-eig-single} Let $\lambda^{min}_{eig_m}$ and $\lambda^{max}_{eig_m}$ be the minimal and maximal eigenvalues of the elastic capacitance matrices $\bar{C}_m$, for $m=1,2,\dots,M$. Denote
by ${C}^a_m$ the capacitance of each scatterer in the acoustic case,\footnote{Recall that, for $m=1,\dots,M$, ${C}^a_m:=\int_{\partial D_m}\sigma_m(s)ds$ and $\sigma_{m}$ is
the solution of the integral equation of the first kind $\int_{\partial D_m}\frac{\sigma_{m} (s)}{4\pi|t-s|}ds=1,~ t\in \partial D_m$, see \cite{DPC-SM13}.} then we have the following estimate;
\begin{eqnarray}\label{lowerupperestforintgradub-m}
 \mu\,C^a_m\,\leq\,\lambda^{min}_{eig_m}\,\leq\,\lambda^{max}_{eig_m}\,\leq\,(\lambda+2\mu)\,C^a_m,\quad\text{for}\quad m=1,2,\dots,M.
\end{eqnarray}
\end{lemma}
\begin{proof}{\it{of Lemma \ref{capacitance-eig-single}}}.
Proof of this Lemma follows as in \cite[Lemma 6.3.6]{M-M-N:Springerbook:2013}. See also \cite[Lemma 10]{M-M-N:AAn:2007}.
\end{proof}
Now, we prove the main lemma of this section.
\begin{lemma}\label{Mazyawrkthmela}
The matrix $\mathbf{B}$
is invertible and the solution vector $\tilde{Q}$ of \eqref{compacfrm1ela} satisfies the estimate:
\begin{eqnarray}\label{mazya-fnlinvert-small-ela-2}
 \sum_{m=1}^{M}\|\tilde{Q}_m\|_2^{2}
\leq& 4 &\left(\min_{m=1}^{M}\lambda^{min}_{eig_m}-\frac{3t}{5\pi\,d}\max_{m=1}^{M}\lambda^{max^2}_{eig_m}\right)^{-2}\left(\max_{m=1}^{M}\lambda^{max}_{eig_m}\right)^4\sum_{m=1}^{M}\|U^i(z_m)\|_2^2,
\end{eqnarray}
if we consider $\left(\max\limits_{1\leq\,m\leq\,M}\lambda^{max^2}_{eig_m}\right)<t^{-1}\left(\frac{5\pi}{3}d\min\limits_{1\leq\,m\leq\,M}\lambda^{min}_{eig_m}\right)$ with the positively assumed value \\
$
t:=\left[\frac{1}{c_p^2}-2diam(\Omega)\frac{\omega}{c_s^3}\left(\frac{1-\left(\frac{1}{2}\kappa_{s^\omega}diam(\Omega)\right)^{N_\Omega}}{1-\left(\frac{1}{2}\kappa_{s^\omega}diam(\Omega)\right)}+\frac{1}{2^{N_{\Omega}-1}}\right)
-diam(\Omega)\frac{\omega}{c_p^3}\left(\frac{1-\left(\frac{1}{2}\kappa_{p^\omega}diam(\Omega)\right)^{N_\Omega}}{1-\left(\frac{1}{2}\kappa_{p^\omega}diam(\Omega)\right)}+\frac{1}{2^{N_{\Omega}-1}}\right)\right].
$
\end{lemma}
\begin{proof}{\it{of Lemma \ref{Mazyawrkthmela}}.}
We can factorize $\mathbf{B}$ as $\mathbf{B}=-(I+\mathbf{B}_{n}\mathbf{C})\mathbf{C}^{-1}$ where $\mathbf{C}:=Diag(\bar{C}_1,\bar{C}_2,\dots,\bar{C}_M)\in\mathbb{R}^{3M\times\,3M}$, $I$ is the identity matrix and
$\mathbf{B}_{n}:=-\mathbf{C}^{-1}-\mathbf{B}$. Hence, the solvability of the system \eqref{compacfrm1ela}, depends on the existence of the inverse of $(I+\mathbf{B}_{n}\mathbf{C})$.
We have $(I+\mathbf{B}_{n}\mathbf{C}):\mathbb{C}^{3M}\rightarrow\mathbb{C}^{3M}$, so it is enough to prove the injectivity in order to prove its invetibility. For this purpose,  let $X,Y$ are vectors in $\mathbb{C}^{M}$ and
 consider the system
\begin{eqnarray}\label{systemsolve1-small-ela}(I+\mathbf{B}_{n}\mathbf{C})X&=&Y.\end{eqnarray}
Let ${(\cdot)}^{real}$ and ${(\cdot)}^{img}$ denotes the real and the imaginary parts of the corresponding complex number/vecctor/matrix. Now, the following can be written from \eqref{systemsolve1-small-ela};
\begin{eqnarray}
 (I+\mathbf{B}^{real}_{n}\mathbf{C})X^{real}-\mathbf{B}^{img}_{n}\mathbf{C}X^{img}&=&Y^{real}\label{systemsolve1-small-sub1-ela},\\
 (I+\mathbf{B}^{real}_{n}\mathbf{C})X^{img}+\mathbf{B}^{img}_{n}\mathbf{C}X^{real}&=&Y^{img}\label{systemsolve1-small-sub2-ela},
\end{eqnarray}
which leads to
\begin{eqnarray}
 \langle\,(I+\mathbf{B}^{real}_{n}\mathbf{C})X^{real},\mathbf{C}X^{real}\rangle\,-\langle\,\mathbf{B}^{img}_{n}\mathbf{C}X^{img},\mathbf{C}X^{real}\rangle&=&\langle\,Y^{real},\mathbf{C}X^{real}\rangle\label{systemsolve1-small-sub1-ela1},\\
 \langle\,(I+\mathbf{B}^{real}_{n}\mathbf{C})X^{img},\mathbf{C}X^{img}\rangle\,+\langle\,\mathbf{B}^{img}_{n}\mathbf{C}X^{real},\mathbf{C}X^{img}\rangle&=&\langle\,Y^{img},\mathbf{C}X^{img}\rangle\label{systemsolve1-small-sub2-ela1}.
\end{eqnarray}
By summing up \eqref{systemsolve1-small-sub1-ela1} and \eqref{systemsolve1-small-sub2-ela1} will give
\begin{equation}
\begin{split}
\langle\,X^{real},\mathbf{C}X^{real}\rangle\,+\langle\,\mathbf{B}^{real}_{n}\mathbf{C}X^{real},\mathbf{C}X^{real}\rangle\,+\langle\,X^{img},\mathbf{C}X^{img}\rangle\,+\langle\,\mathbf{B}^{real}_{n}\mathbf{C}X^{img},\mathbf{C}X^{img}\rangle\,\\
=\langle\,Y^{real},\mathbf{C}X^{real}\rangle+\langle\,Y^{img},\mathbf{C}X^{img}\rangle.\label{systemsolve1-small-sub1-2-ela1}
\end{split}
\end{equation}
Indeed,
\[ \langle\,\mathbf{B}^{img}_{n}\mathbf{C}X^{img},\mathbf{C}X^{real}\rangle=\langle\,\mathbf{C}X^{img},\mathbf{B}^{img^{*}}_{n}\mathbf{C}X^{real}\rangle
=\langle\,\mathbf{C}X^{img},\mathbf{B}^{img}_{n}\mathbf{C}X^{real}\rangle=\langle\,\mathbf{B}^{img}_{n}\mathbf{C}X^{real},\mathbf{C}X^{img}\rangle.\]
We can observe that, the right-hand side in \eqref{systemsolve1-small-sub1-2-ela1} does not exceed
\begin{equation}\label{{systemsolve1-small-sub1-2-ela2}}
\begin{split}
\langle\,X^{real},\mathbf{C}X^{real}\rangle^{1\slash\,2}\langle\,Y^{real},\mathbf{C}Y^{real}\rangle^{1\slash\,2}+\langle\,X^{img},\mathbf{C}X^{img}\rangle^{1\slash\,2}\langle\,Y^{img},\mathbf{C}Y^{img}\rangle^{1\slash\,2}\\
\leq2\langle\,X^{\|\cdot\|},(\mathbf{C}X)^{\|\cdot\|}\rangle^{1\slash\,2}\langle\,Y^{\|\cdot\|},(\mathbf{C}Y)^{\|\cdot\|}\rangle^{1\slash\,2}.
\end{split}
\end{equation}
Here $W^{\|\cdot\|}_m:={\left[\|W^{{real}}_m\|^2+\|W^{{img}}_m\|^2\right]}^{1\slash2}= \|W_m\|_2$,
for $W=X,Y$ and $m=1,\dots,M$.
Consider the second term in the left-hand side of \eqref{systemsolve1-small-sub1-2-ela1}. Using the mean value theorem for harmonic functions we deduce
\begin{eqnarray*}
\langle\,\mathbf{B}^{real}_{n}\mathbf{C}X^{real},\mathbf{C}X^{real}\rangle&=&\sum_{\substack{1\leq\,j,m\leq\,M\\j\neq\,m}}X^{real^\top}_{m}\bar{C}_m^\top\left[\Gamma^{\omega}(z_m,z_j)\right]^{real}\bar{C}_jX^{real}_{j}\\
&\geq&t\sum_{\substack{1\leq\,j,m\leq\,M\\j\neq\,m}}X^{real^\top}_{m}\bar{C}_m^\top\left(\frac{1}{|B^{(j)}||B^{(m)}|}\int_{B^{(j)}}\int_{B^{(m)}}\Phi_0(x,y)~dx~dy\right)\bar{C}_jX^{real}_{j},
\end{eqnarray*}
Similarly, if we consider the fourth term in the left-hand side of \eqref{systemsolve1-small-sub1-2-ela1}, we deduce
\begin{eqnarray*}
\langle\,\mathbf{B}^{real}_{n}\mathbf{C}X^{img},\mathbf{C}X^{img}\rangle&=&\sum_{\substack{1\leq\,j,m\leq\,M\\j\neq\,m}}X^{img^\top}_{m}\bar{C}_m^\top\left[\Gamma^{\omega}(z_m,z_j)\right]^{real}\bar{C}_jX^{img}_{j}\\
&\geq&t\sum_{\substack{1\leq\,j,m\leq\,M\\j\neq\,m}}X^{img^\top}_{m}\bar{C}_m^\top\left(\frac{1}{|B^{(j)}||B^{(m)}|}\int_{B^{(j)}}\int_{B^{(m)}}\Phi_0(x,y)~dx~dy\right)\bar{C}_jX^{img}_{j},
\end{eqnarray*}
where
\begin{center}
$
t:=\left[\frac{1}{c_p^2}-2diam(\Omega)\frac{\omega}{c_s^3}\left(\frac{1-\left(\frac{1}{2}\kappa_{s^\omega}diam(\Omega)\right)^{N_\Omega}}{1-\left(\frac{1}{2}\kappa_{s^\omega}diam(\Omega)\right)}+\frac{1}{2^{N_{\Omega}-1}}\right)-diam(\Omega)\frac{\omega}{c_p^3}\left(\frac{1-\left(\frac{1}{2}\kappa_{p^\omega}diam(\Omega)\right)^{N_\Omega}}{1-\left(\frac{1}{2}\kappa_{p^\omega}diam(\Omega)\right)}+\frac{1}{2^{N_{\Omega}-1}}\right)\right]
$
\end{center}
assumed to be positive, $\Phi_0(x,y):=1\slash(4\pi|x-y|),x\neq\,y$ and $B^{(m)}:=\{x:|x-z_m|<d\slash2\},m=1,\dots,M$, are non-overlapping balls of radius $d\slash2$ with centers at $z_m$, and $|B^{(m)}|=\pi\,d^3\slash6$ are the volumes of the balls.
Also, we use the notation $B_d$ to denote the balls of radius $d\slash2$   with the center at the origin.\\
\begin{indeed} we can write  $\Gamma^{\omega}(z_m,z_j)$ from \eqref{kupradzeten1} as,
 \begin{eqnarray}\label{kupradzeten1distant}
 \Gamma^\omega(z_m,z_j)
&=&\frac{1}{4\pi|z_m-z_j|}\left(\frac{1}{2}\left[\frac{1}{c_s^2}+\frac{1}{c_p^2}\right]\rm \textbf{I}+\underbrace{\frac{1}{2}\left[\frac{1}{c_s^2}-\frac{1}{c_p^2}\right]\frac{(z_m-z_j)}{|z_m-z_j|}\otimes\frac{(z_m-z_j)}{|z_m-z_j|}}_{b_\Gamma}\right.\nonumber\\
&&\left.+\underbrace{\sum_{l=1}^{\infty}\frac{i^l}{l!(l+2)}\frac{1}{\omega^2}\left((l+1)\kappa_{s^\omega}^{l+2}+\kappa_{p^\omega}^{l+2}\right)|z_m-z_j|^{l}\rm \textbf{I}}_{c1_{\Gamma}}\right.\nonumber\\
         & &\left.\underbrace{-\sum_{l=1}^{\infty}\frac{i^l}{l!(l+2)}\frac{(l-1)}{\omega^2}\left(\kappa_{s^\omega}^{l+2}-\kappa_{p^\omega}^{l+2}\right)|z_m-z_j|^{l-2}(z_m-z_j)\otimes(z_m-z_j)}_{c2_{\Gamma}}\right),
 \end{eqnarray}
from which, we get the required result by estimating $\Gamma^\omega(z_m,z_j)$. Notice that
\begin{eqnarray*}
 |b_\Gamma|&\leq&\frac{1}{2}\left[\frac{1}{c_s^2}-\frac{1}{c_p^2}\right]\mbox{ and }\\
 |c1_{\Gamma}+c2_{\Gamma}|&\leq&\sum_{l=1}^{\infty}\frac{1}{(l-1)!(l+2)}\frac{1}{\omega^2}
    \left(2\kappa_{s^\omega}^{l+2}+\kappa_{p^\omega}^{l+2}\right)|z_m-z_j|^{l}\\
    &&\mbox{[By recalling $N_{\Omega}=[2diam(\Omega)\max\{\kappa_{s^\omega},\kappa_{p^\omega}\}e^2]$ and using Lemma \ref{stirlingapproxlemma}]}\nonumber\\
&\leq&diam(\Omega)\left[\frac{2\omega}{c_s^3}\left(\sum_{l=1}^{N_{\Omega}}\left(\frac{1}{2}\kappa_{s^\omega}diam(\Omega)\right)^{l-1}+\sum_{l=N_{\Omega}+1}^{\infty}\frac{1}{2^{l-1}}\right)\right.\\
&&\left.\qquad+\frac{\omega}{c_p^3}\left(\sum_{l=1}^{N_{\Omega}}\left(\frac{1}{2}\kappa_{p^\omega}diam(\Omega)\right)^{l-1}+\sum_{l=N_{\Omega}+1}^{\infty}\frac{1}{2^{l-1}}\right)\right]\\
&=&diam(\Omega)\left[2\frac{\omega}{c_s^3}\left(\frac{1-\left(\frac{1}{2}\kappa_{s^\omega}diam(\Omega)\right)^{N_\Omega}}{1-\left(\frac{1}{2}\kappa_{s^\omega}diam(\Omega)\right)}+\frac{1}{2^{N_{\Omega}-1}}\right)\right.\\
&&\qquad\left.+\frac{\omega}{c_p^3}\left(\frac{1-\left(\frac{1}{2}\kappa_{p^\omega}diam(\Omega)\right)^{N_\Omega}}{1-\left(\frac{1}{2}\kappa_{p^\omega}diam(\Omega)\right)}+\frac{1}{2^{N_{\Omega}-1}}\right)\right],
\end{eqnarray*}
\end{indeed}

Let $\Omega$ be a large ball with radius $R$. Also let $\Omega_s\subset\Omega$ be a ball with fixed radius $r(\leq\,R)$, which consists of all our small obstacles $D_m$ and also the balls $B^{(m)}$, for $m=1,\dots,M$.

Let $\Upsilon^{real}(x)$ and $\Upsilon^{img}(x)$ be piecewise constant functions defined on $\mathbb{R}^3$ as
\begin{equation}\label{def-upsilon-ela}
 \Upsilon^{real\,(img)}(x)=\begin{cases}
\begin{array}{ccc}
              \bar{C}_mX^{real\,(img)}_m &\mbox{in }  B^{(m)},&m=1,\dots,M,\\
              0 &\mbox{otherwise}.
\end{array}
             \end{cases}
\end{equation}
Then
\begin{equation}\label{systemsolve1-small-sub1-2-ela1-basedonupsilon-real}
\begin{split}
\langle\,\mathbf{B}^{real}_{n}\mathbf{C}X^{real},\mathbf{C}X^{real}\rangle\geq\frac{36t}{\pi^2\,d^6}&\left(\int_\Omega\int_\Omega\Phi_0(x,y)\Upsilon^{real^\top}(x)\Upsilon^{real}(y)~dx~dy\right.\\
&\left.-\sum_{m=1}^{M}\left|\bar{C}_mX^{real}_{m}\right|^2\int_{B^{(m)}}\int_{B^{(m)}}\Phi_0(x,y)~dx~dy\right)
\end{split}
\end{equation}
\begin{equation}\label{systemsolve1-small-sub1-2-ela1-basedonupsilon-img}
\begin{split}
\langle\,\mathbf{B}^{real}_{n}\mathbf{C}X^{img},\mathbf{C}X^{img}\rangle\geq\frac{36t}{\pi^2\,d^6}&\left(\int_\Omega\int_\Omega\Phi_0(x,y)\Upsilon^{img^\top}(x)\Upsilon^{img}(y)~dx~dy\right.\\
&\left.-\sum_{m=1}^{M}\left|\bar{C}_mX^{img}_{m}\right|^2\int_{B^{(m)}}\int_{B^{(m)}}\Phi_0(x,y)~dx~dy\right)
\end{split}
\end{equation}
Applying the mean value theorem to the harmonic function $\frac{1}{4\pi\vert x-y \vert}$, as done in \cite[p:109-110]{M-M:MathNach2010}, we have the following estimate
\begin{equation}
 \begin{split}\label{systemsolve1-small-sub1-2-ac1-basedonupsilon-2ndpart*}
  \int_{B^{(m)}}\int_{B^{(m)}}\Phi_0(x,y)~dx~dy\,=\,\frac{1}{4\pi}\int_{B_d}\int_{B_d}\frac{1}{|x-y|}~dx~dy \,\leq\, \frac{\pi\,d^5}{60}.
 \end{split}
\end{equation}
Consider the first term in the right-hand side of \eqref{systemsolve1-small-sub1-2-ela1-basedonupsilon-real}, denote it by $A_R^{real}$, then by Green's theorem
\begin{eqnarray}\label{systemsolve1-small-sub1-2-ela1-basedonupsilon-1stpart}
A_R^{real}&:=&\int_\Omega\int_\Omega\Phi_0(x,y)\Upsilon^{real^\top}(x)\Upsilon^{real}(y)~dx~dy\\
&=&\underbrace{\int_\Omega\left|\nabla_x\int_\Omega\Phi_0(x,y)\Upsilon^{real}(y)~dy\right|^2~dx}_{=:B_R^{real}\geq0}\nonumber\\
&&\qquad-\underbrace{\int_{\partial\Omega}\left(\frac{\partial}{\partial\nu_x}\int_\Omega\Phi_0(x,y)\Upsilon^{real}(y)~dy\right)^\top\left(\int_\Omega\Phi_0(x,y)\Upsilon^{real}(y)~dy\right)dS_x}_{=:C_R^{real}}.\nonumber
\end{eqnarray}
We have
\begin{equation}\label{systemsolve1-small-sub1-2-ela1-basedonupsilon-1stpart-1}
 \begin{split}
  C_R^{real}
&=\int_{\partial\Omega}\left(\int_\Omega\frac{\partial}{\partial\nu_x}\Phi_0(x,y)\Upsilon^{real}(y)~dy\right)^\top\left(\int_\Omega\Phi_0(x,y)\Upsilon^{real}(y)~dy\right)dS_x\\
&=\int_{\partial\Omega}\left(\int_{\Omega_s}\frac{\partial}{\partial\nu_x}\Phi_0(x,y)\Upsilon^{real}(y)~dy\right)^\top\left(\int_{\Omega_s}\Phi_0(x,y)\Upsilon^{real}(y)~dy\right)dS_x\\
&=\int_{\partial\Omega}\left(\int_{\Omega_s}\frac{-(x-y)}{4\pi|x-y|^3}\Upsilon^{real}(y)~dy\right)^\top\left(\int_{\Omega_s}\frac{1}{4\pi|x-y|}\Upsilon^{real}(y)~dy\right)dS_x,
 \end{split}
\end{equation}
which gives the following estimate;
\begin{equation}\label{systemsolve1-small-sub1-2-ela1-basedonupsilon-1stpart-2}
 \begin{split}
|C_R^{real}|&\leq\frac{1}{16\pi^2}\int_{\partial\Omega}\frac{1}{|R-r|^3}\left(\int_{\Omega_s}|\Upsilon^{real}(y)|~dy\right)^2dS_x\\
&\leq\frac{1}{16\pi^2}\frac{1}{(R-r)^3}\int_{\partial\Omega}|\Omega_s|~||\Upsilon^{real}||^2_{\mathcal{L}^2({\Omega_s})}dS_x\\
 &=\frac{r^3}{12\pi(R-r)^3}\sum_{m=1}^{M}\left|\bar{C}_m\,X^{real}_m\right|^2~|\Omega|\\
&=\frac{R^2r^3}{3(R-r)^3}\sum_{m=1}^{M}\left|\bar{C}_m\,X^{real}_m\right|^2.
\end{split}
\end{equation}
Substitution of \eqref{systemsolve1-small-sub1-2-ela1-basedonupsilon-1stpart-2} in \eqref{systemsolve1-small-sub1-2-ela1-basedonupsilon-1stpart} gives
\begin{eqnarray}\label{systemsolve1-small-sub1-2-ela1-basedonupsilon-1stpart-3}
&\int_\Omega\int_\Omega\Phi_0(x,y)\Upsilon^{real}(x)\Upsilon^{real}(y)~dx~dy\nonumber\\
&\geq&\hspace{-2cm}\int_\Omega\left|\nabla_x\int_\Omega\Phi_0(x,y)\Upsilon^{real}(y)~dy\right|^2~dx-\frac{R^2r^3}{3(R-r)^3}\sum_{m=1}^{M}\left|\bar{C}_m\,X^{real}_m\right|^2.
\end{eqnarray}
By considering the first term in the right-hand side of \eqref{systemsolve1-small-sub1-2-ela1-basedonupsilon-img}, and following the  same procedure as mentioned in
\eqref{systemsolve1-small-sub1-2-ela1-basedonupsilon-1stpart}, \eqref{systemsolve1-small-sub1-2-ela1-basedonupsilon-1stpart-1} and \eqref{systemsolve1-small-sub1-2-ela1-basedonupsilon-1stpart-2}, we obtain
\begin{eqnarray}\label{systemsolve1-small-sub1-2-ela1-basedonupsilon-1stpart-4}
&&\hspace{-1cm}\int_\Omega\int_\Omega\Phi_0(x,y)\Upsilon^{img}(x)\Upsilon^{img}(y)~dx~dy\nonumber\\
&\geq&\hspace{0cm}\int_\Omega\left|\nabla_x\int_\Omega\Phi_0(x,y)\Upsilon^{img}(y)~dy\right|^2~dx-\frac{R^2r^3}{3(R-r)^3}\sum_{m=1}^{M}\left|\bar{C}_m\,X^{img}_m\right|^2.
\end{eqnarray}
Under our assumption $t>0$,  \eqref{systemsolve1-small-sub1-2-ela1-basedonupsilon-real}, \eqref{systemsolve1-small-sub1-2-ela1-basedonupsilon-img}
\eqref{systemsolve1-small-sub1-2-ac1-basedonupsilon-2ndpart*}, \eqref{systemsolve1-small-sub1-2-ela1-basedonupsilon-1stpart-3} and \eqref{systemsolve1-small-sub1-2-ela1-basedonupsilon-1stpart-4} lead to
\begin{eqnarray}
&\langle\,\mathbf{B}^{real}_{n}\mathbf{C}X^{real},\mathbf{C}X^{real}\rangle\nonumber\\
&\hspace{-1.5cm}\geq&\hspace{-2.35cm}\frac{36t}{\pi^2\,d^6}\left(\int_\Omega\left|\nabla_x\int_\Omega\Phi_0(x,y)\Upsilon^{real}(y)~dy\right|^2~dx-\left[\frac{R^2r^3}{3(R-r)^3}+\frac{\pi\,d^5}{60}\right]\sum_{m=1}^{M}\left|\bar{C}_m\,X^{real}_m\right|^2\right),
\label{systemsolve1-small-sub1-2-ela1-basedonupsilon-real-1}\\
&\langle\,\mathbf{B}^{real}_{n}\mathbf{C}X^{img},\mathbf{C}X^{img}\rangle\nonumber\\
&\hspace{-1.5cm}\geq&\hspace{-2.35cm}\frac{36t}{\pi^2\,d^6}\left(\int_\Omega\left|\nabla_x\int_\Omega\Phi_0(x,y)\Upsilon^{img}(y)~dy\right|^2~dx-\left[\frac{R^2r^3}{3(R-r)^3}+\frac{\pi\,d^5}{60}\right]\sum_{m=1}^{M}\left|\bar{C}_m\,X^{img}_m\right|^2\right).
\label{systemsolve1-small-sub1-2-ela1-basedonupsilon-img-1}
\end{eqnarray}
Then \eqref{systemsolve1-small-sub1-2-ela1}, \eqref{systemsolve1-small-sub1-2-ela1-basedonupsilon-real-1} and \eqref{systemsolve1-small-sub1-2-ela1-basedonupsilon-img-1} imply
\begin{equation}\label{fnlinvert-small-ela-1}
\begin{split}
 \left(\min_{m=1}^{M}\lambda^{min}_{eig_m}-\frac{36t}{\pi^2\,d^6}\right.&\left.\left[\frac{R^2r^3}{3(R-r)^3}+\frac{\pi\,d^5}{60}\right]\max_{m=1}^{M}\lambda^{max^2}_{eig_m}\right)\sum_{m=1}^{M}\|X_m\|_2^{2}\\
&\leq2\left(\max_{m=1}^{M}\lambda^{max}_{eig_m}\right)\left(\sum_{m=1}^{M}\|X_m\|_2^2\right)^{1/2}\left(\sum_{m=1}^{M}\|Y_m\|_2^2\right)^{1/2}.
\end{split}
\end{equation}
As we have $R$ arbitrary, by tending $R$ to $\infty$, we can write 
\eqref{fnlinvert-small-ela-1} as
\begin{equation}\label{fnlinvert-small-ela-2}
\begin{split}
\left(\min_{m=1}^{M}\lambda^{min}_{eig_m}-\frac{3t}{5\pi\,d}\max_{m=1}^{M}\lambda^{max^2}_{eig_m}\right)\sum_{m=1}^{M}\|X_m\|_2^{2}
\leq2\left(\max_{m=1}^{M}\lambda^{max}_{eig_m}\right)\left(\sum_{m=1}^{M}\|X_m\|_2^2\right)^{1/2}\left(\sum_{m=1}^{M}\|Y_m\|_2^2\right)^{1/2}.
\end{split}
\end{equation}
which  yields
\begin{equation}\label{fnlinvert-small-ela-2f}
\begin{split}
\sum_{m=1}^{M}\|X_m\|_2^{2}
\leq4 \left(\min_{m=1}^{M}\lambda^{min}_{eig_m}-\frac{3t}{5\pi\,d}\max_{m=1}^{M}\lambda^{max^2}_{eig_m}\right)^{-2}\left(\max_{m=1}^{M}\lambda^{max}_{eig_m}\right)^2\sum_{m=1}^{M}\|Y_m\|_2^2.
\end{split}
\end{equation}
Thus, if $\left(\max\limits_{1\leq\,m\leq\,M}\lambda^{max^2}_{eig_m}\right)<t^{-1}\left(\frac{5\pi}{3}d\min\limits_{1\leq\,m\leq\,M}\lambda^{min}_{eig_m}\right)$ then the matrix $\mathbf{B}$
in algebraic system \eqref{compacfrm1ela} is invertible and the estimate \eqref{fnlinvert-small-ela-2} and so \eqref{mazya-fnlinvert-small-ela-2} holds.
\end{proof}
\begin{corollary}\label{Mazyawrkthmela-cor}
 If $(\lambda+2\mu)^2\left(\max\limits_{1\leq\,m\leq\,M}\,C^a_m\right)^2<t^{-1}\left(\frac{5\pi}{3}\mu{d}\min\limits_{1\leq\,m\leq\,M}C^a_m\right)$, then the matrix $\mathbf{B}$
is invertible and the solution vector $\tilde{Q}$ of \eqref{compacfrm1ela} satisfies the estimate:
\begin{eqnarray}\label{mazya-fnlinvert-small-ela-3-cor} 
 \sum_{m=1}^{M}\|\tilde{Q}_m\|_2
\leq&2 &\left(1-\frac{3t}{5\pi}\frac{(\lambda+2\mu)^2}{\mu}\frac{\max\limits_{m=1}^{M}{C^a_m}^2}{d\min\limits_{m=1}^{M}C^a_m}\right)^{-1}\left(\frac{\max\limits_{m=1}^{M}C^a_m}{\min\limits_{m=1}^{M}C^a_m}\right)M\max_{m=1}^{M}C^a_m\max_{m=1}^{M}\|U^i(z_m)\|_2.
\end{eqnarray}
\end{corollary}
\begin{proof}{\it{of Corollary \ref{Mazyawrkthmela-cor}}}.
Let us assume the condition $(\lambda+2\mu)^2\left(\max\limits_{1\leq\,m\leq\,M}\,C^a_m\right)^2<t^{-1}\left(\frac{5\pi}{3}\mu{d}\min\limits_{1\leq\,m\leq\,M}C^a_m\right)$,
then from Lemma \ref{capacitance-eig-single} the sufficient condition of Lemma \ref{Mazyawrkthmela} is satisfied and hence \eqref{mazya-fnlinvert-small-ela-2} holds. Now, by applying the
norm inequalities to \eqref{mazya-fnlinvert-small-ela-2}, we obtain
\begin{eqnarray}\label{mazya-fnlinvert-small-ela-3} 
 \sum_{m=1}^{M}\|\tilde{Q}_m\|_2
\leq&2 &\left(\min_{m=1}^{M}\lambda^{min}_{eig_m}-\frac{3t}{5\pi\,d}\max_{m=1}^{M}\lambda^{max^2}_{eig_m}\right)^{-1}\left(\max_{m=1}^{M}\lambda^{max}_{eig_m}\right)^2M\max_{m=1}^{M}\|U^i(z_m)\|_2.
\end{eqnarray}
Now, again by applying Lemma \ref{capacitance-eig-single} to the above inequality \eqref{mazya-fnlinvert-small-ela-3} gives the result \eqref{mazya-fnlinvert-small-ela-3-cor}.
\end{proof}

\section{Appendix}
The object of this section is to derive some used properties of the single layer operator
$\mathcal{S}_{ D_\epsilon}:L^2(\partial D_{\epsilon})\rightarrow H^1(\partial D_{\epsilon})$ defined by
\begin{eqnarray}\label{defofSpartialDeela}
\left(\mathcal{S}_{ D_\epsilon}\psi\right) (x):=\int_{\partial D_\epsilon}\Gamma^\omega(x,y)\psi(y)dy.
\end{eqnarray}
\begin{lemma}\label{invertibility-of-slpela}
There exists $\epsilon_0$ such that if $\epsilon < \epsilon_0$ then the operator $\mathcal{S}_{ D_\epsilon}$ is invertible.
\end{lemma}
\begin{proof}{\it{of Lemma \ref{invertibility-of-slpela}}.}
Proof of this Lemma follows as the one of Proposition \ref{existence-of-sigmasdbl}.
\end{proof}
\begin{lemma}\label{rep1singulayerela}
 Let $\phi\in H^{1}(\partial D_\epsilon)$ and $\psi \in L^{2}(\partial D_\epsilon)$. Then,
\begin{equation}\label{rep1singulayer1}
 \mathcal{S}_{ D_\epsilon}\psi=\epsilon~ (\mathcal{S}^\epsilon_B \hat{\psi})^\vee,
\end{equation}
\begin{equation}\label{rep1singulayer2ela}
 \mathcal{S}_{ D_\epsilon}^{-1}\phi=\epsilon^{-1} ({\mathcal{S}^\epsilon_B}^{-1} \hat{\phi})^\vee
\end{equation}
and
\begin{equation}\label{nrm1singulayer2ela}
 \left\|\mathcal{S}_{ D_\epsilon}^{-1}\right\|_{\mathcal{L}\left(H^1(\partial D_\epsilon), L^2(\partial D_\epsilon) \right)}\leq
\epsilon^{-1}\left\|{\mathcal{S}^\epsilon_B}^{-1}\right\|_{\mathcal{L}\left(H^1(\partial B), L^2(\partial B) \right)}
\end{equation}
with $\mathcal{S}^\epsilon_B \hat{\psi}(\xi):=\int_{\partial B} \Gamma^{\epsilon\omega}(\xi,\eta)\hat{\psi}(\eta) d\eta$.
\end{lemma}
\begin{proof}{\it{of Lemma \ref{rep1singulayerela}}.}
\ ~ \ \\
\begin{itemize}
\item We have,
 \begin{eqnarray*}
 \mathcal{S}_{ D_\epsilon}\psi(x)&=&\int_{\partial D_\epsilon}\Gamma^{\omega}(x,y)\psi(y) dy\\
          &=&\int_{\partial B}\frac{1}{\epsilon}\Gamma^{\epsilon\omega}(\xi,\eta)\psi(\epsilon\eta+z) \epsilon^{2}d\eta\\
          &=&\epsilon ~\mathcal{S}^\epsilon_B \hat{\psi}(\xi).
 \end{eqnarray*}
The above gives us \eqref{rep1singulayer1}.
\item The following equalities, using \eqref{rep1singulayer1},
\begin{eqnarray*}
\mathcal{S}_{ D_\epsilon}({\mathcal{S}^\epsilon_B}^{-1} \hat{\phi})^\vee
                                       &=& \epsilon~(\mathcal{S}^\epsilon_B{\mathcal{S}^\epsilon_B}^{-1} \hat{\phi})^\vee 
                                       \,=\,\epsilon~\hat{\phi}^{\vee}
                                       \,=\, \epsilon~\phi
\end{eqnarray*}
provides us \eqref{rep1singulayer2ela}.
\item We have from the estimate,
\begin{eqnarray*}
 \left\|\mathcal{S}_{ D_\epsilon}^{-1}\right\|_{\mathcal{L}\left(H^1(\partial D_\epsilon), L^2(\partial D_\epsilon) \right)}&:=&
\substack{Sup\\ \phi(\neq0)\in H^{1}(\partial D_\epsilon)} \frac{||\mathcal{S}_{ D_\epsilon}^{-1}\phi||_{L^2(\partial D_\epsilon)}}{||\phi||_{H^1(\partial D_\epsilon)}}\\
&\,\substack{\leq\\\eqref{habib2*},\eqref{habib1*}}\,&\substack{Sup\\ \phi(\neq0)\in H^{1}(\partial D_\epsilon)}
\frac{\epsilon~\left\|(\mathcal{S}_{ D_\epsilon}^{-1}\phi)^{\wedge}\right\|_{L^2(\partial B)}}{\epsilon~||\hat{\phi}||_{H^1(\partial B)}}\\
&\,\substack{=\\\eqref{rep1singulayer2ela}}\,&\substack{Sup\\ \hat{\phi}(\neq0)\in H^{1}(\partial D_\epsilon)}
\frac{\epsilon^{-1}\left\| {\mathcal{S}^\epsilon_B}^{-1} \hat{\phi}\right\|_{L^2(\partial B)}}{||\hat{\phi}||_{H^1(\partial B)}} \\
&=&\epsilon^{-1}\left\|{\mathcal{S}^\epsilon_B}^{-1}\right\|_{\mathcal{L}\left(H^1(\partial B), L^2(\partial B) \right)}. 
\end{eqnarray*}
\end{itemize}
\end{proof}
\bibliographystyle{abbrv}

\end{document}